\documentclass[11pt]{amsart}
\usepackage{pb-diagram,amssymb,epic,eepic,verbatim,transparent}
\usepackage{tikz-cd,braket}
\usepackage{xr-hyper}
\usepackage{hyperref,graphics,epsfig,psfrag}
\usepackage[inline]{enumitem}
\hypersetup{colorlinks}
\usepackage{color}
\definecolor{darkred}{rgb}{0.5,0,0}
\definecolor{darkgreen}{rgb}{0,0.5,0}
\definecolor{darkblue}{rgb}{0,0,0.5}
\hypersetup{ colorlinks,
linkcolor=darkblue,
filecolor=darkgreen,
urlcolor=darkred,
citecolor=darkblue }
\newtheorem{theorem}{Theorem}[section]

\newtheorem{corollary}[theorem]{Corollary}

\newtheorem{proposition}[theorem]{Proposition}
\newtheorem{lemma}[theorem]{Lemma}
\newtheorem{lem}[theorem]{}
\theoremstyle{definition}
\newtheorem{definition}[theorem]{Definition}
\theoremstyle{remark}
\newtheorem{remark}[theorem]{Remark}
\newtheorem{example}[theorem]{Example}
\newcommand{\blem}{\begin{lem} \rm}
\newcommand{\elem}{\end{lem}}

\newcounter{mparcnt}

\newcommand\M{\mathcal{M}}
\newcommand\D{\mathbb{D}}

\newcommand\mC{\mathcal{C}}

\renewcommand\S{\mathcal{S}}

\newcommand{\W}{\mathcal{W}}

\newcommand{\T}{\mathcal{T}}
\newcommand{\J}{\mathcal{J}}

\newcommand{\U}{\mathcal{U}}
\newcommand{\F}{\mathcal{F}}

\newcommand{\R}{\mathbb{R}}

\newcommand{\C}{\mathbb{C}}
\newcommand{\CC}{C\kern-1.3ex|}

\newcommand{\Z}{\mathbb{Z}}
\newcommand{\Q}{\mathbb{Q}}

\renewcommand{\P}{\mathbb{P}}

\newcommand{\PP}{\mathcal{P}}

\newcommand{\Pe}{\mathfrak{p}}
\newcommand{\PPe}{\mathcal{P}}

\newcommand\lie[1]{\mathfrak{#1}}

\renewcommand{\t}{\lie{t}}

\newcommand{\on}{\operatorname}

\newcommand{\Diff}{\on{Diff}}
\newcommand{\ann}{\on{ann}}

\newcommand{\ainfty}{{$A_\infty$\ }}

\newcommand{\pre}{{\on{pre}}}

\newcommand{\fr}{{\on{fr}}}

\newcommand{\red}{{\on{red}}}

\newcommand{\dual}{\vee}

\newcommand\abs[1]{\lvert #1 \rvert}

\newcommand{\Edge}{\on{Edge}}

\newcommand{\Ver}{\on{Vert}}

\newcommand{\Hol}{ \on{Hol} }

\renewcommand{\ker}{ \on{ker}}

\newcommand{\im}{ \on{im}}

\newcommand{\mult}{  \on{mult}}

\newcommand{\Disc}{{\on{Disc}}}

\newcommand{\codim}{\on{codim}}

\newcommand\dirac{/\kern-1.2ex\partial} 
\newcommand\qu{/\kern-.7ex/} 
\newcommand\lqu{\backslash \kern-.7ex \backslash} 
\newcommand\bs{\backslash}

\newcommand\etasp{\eta_{\on{gen}}}

\newcommand{\labell}\label

\renewcommand{\d}{{\on{d}}}

\newcommand{\ol}{\overline}

\newcommand\Phinv{\Phi^{-1}}

\newcommand\lam{\lambda}
\newcommand\Lam{\Lambda}

\newcommand\sig{\sigma}
\newcommand\eps{\epsilon}

\newcommand\om{\omega}

\newcommand{\lan}{\langle}
\newcommand{\ran}{\rangle}
\newcommand{\hh}{{\tfrac{1}{2}}}

\newcommand{\ti}{\tilde}
\newcommand{\tM}{\M}
\newcommand{\tP}{{\tilde P}}

\newcommand\cP{\mathcal{P}}
\newcommand\cT{\mathcal{T}}

\newcommand\cI{\mathcal{I}}

\newcommand\Map{\on{Map}}

\newcommand\ev{\on{ev}}

\newcommand\ul{\underline}

\newcommand\reg{{\on{reg}}}
\newcommand\bran[1]{ \lan {#1} \ran}

\newcommand\cwo[1]{}
\newcommand\cwl[1]{{}}
\newcommand\wwo[1]{ { \color{darkpurple} } }

\newcommand\pwo[1]{ { \color{darkpink} } }

\newcommand\Cone{\on{Cone}}

\newcommand\Id{\on{Id}}

\newcommand\XX{\mathfrak{X}} 
\newcommand\XC{\XX} 
\newcommand\XB{X} 

\newcommand\JJ{\mathfrak{J}}

\newcommand\crit{{\on{crit}}}

\newcommand\sx{*\kern-.5ex_X}

\newcommand\white{{\includegraphics[width=.05in]{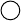}}}

\newcommand\black{{\includegraphics[width=.05in]{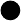}}}
\newcommand\whitet{{\includegraphics[width=.07in]{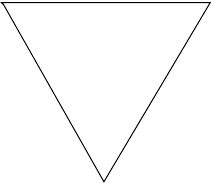}}}
\newcommand\greyt{\includegraphics[width=.07in]{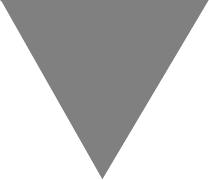}}
\newcommand\blackt{{\includegraphics[width=.07in]{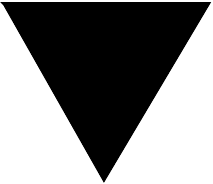}}}

\newcommand{\tX}{\tilde X}

\newcommand{\tC}{\tilde C}

\newcommand{\tensor}{\otimes}
\renewcommand{\rm}{\rmfamily}

\renewcommand{\em}{\textit}

\newcommand{\cyl}{{\on{cyl}}}

\newcommand\br{{\on{brok}}}

\newcommand{\trop}{{\on{trop}}}
\renewcommand{\split}{{\on{split}}}

\newcommand\tGam{{\tilde \Gamma}}

\newcommand\bGam{\Gamma}
\newcommand\Gam{\Gamma}

\newcommand{\deform}{{\on{def}}}
\newcommand{\spl}{{\on{split}}}

\newcommand\fw{w}

\renewcommand\em{\textit}

\def\mathunderaccent#1{\let\theaccent#1\mathpalette\putaccentunder}
\def\putaccentunder#1#2{\oalign{$#1#2$\crcr\hidewidth \vbox
to.2ex{\hbox{$#1\theaccent{}$}\vss}\hidewidth}}

\usepackage{xr}
\begin{document}


\title{Splitting the diagonal for broken maps}

\begin{abstract} In previous work \cite{vw:trop} we introduced a
  version of the Fukaya algebra associated to a degeneration of a
  symplectic manifold, whose structure maps count collections of maps
  in the components of the degeneration satisfying matching
  conditions.  In this paper, we introduce a further degeneration of the 
  matching conditions (similar in spirit to Bourgeois' version of
  symplectic field theory \cite{bo:com}) which results in a \em{ split
    Fukaya algebra} whose structure maps are, in good cases, sums of
  products over vertices of tropical graphs.  In the case of toric
  Lagrangians contained in a toric component of the degeneration, an
  invariance argument implies the existence of projective
  Maurer-Cartan solutions,
  which gives an alternate proof of the unobstructedness result of
  Fukaya-Oh-Ohta-Ono
  \cite{fooo:toric1} for toric manifolds. Our result also proves unobstructedness in more general cases, such as
   for toric Lagrangians in almost toric four-manifolds.
\end{abstract}

\author{Sushmita Venugopalan and Chris T. Woodward}

\thanks{C.W.  was partially supported by NSF grant DMS 2105417.  Any
  opinions, findings, and conclusions or recommendations expressed in
  this material are those of the author(s) and do not necessarily
  reflect the views of the National Science Foundation.}

\maketitle

\tableofcontents

\section{Introduction}

In our previous work \cite{vw:trop} we studied the limit of the Fukaya
algebra under multi-directional neck stretching.  The main result was
a homotopy equivalence between the Fukaya algebra
of a Lagrangian
in a smooth manifold, and a broken analog,
called the \em{broken Fukaya algebra}, 
defined by counting broken maps
associated to the neck-stretching limit.
Broken maps
have to satisfy matching conditions at
the nodes, which means that the moduli space of broken maps is not an
honest product of the moduli spaces of the components of the broken
map, but rather a fiber product over a diagonal.

In this paper, we introduce a further degeneration of broken maps into
split maps.  In split maps, there is no edge matching condition on a
subset of edges, called \em{split edges}. A homotopy equivalence
between the Fukaya algebras defined using broken maps and split maps
is constructed by deforming the edge matching condition at the split
edges. The maps with a deformed edge matching condition on split edges
are called \em{deformed maps}. Assuming that the matching conditions
are deformed in a generic direction,
the limit of
deformed maps then produces \em{split maps}. The tropical graph of a split map, called the \em{split tropical graph}
does not have a matching condition at split edges, however the set of discrepancies from the matching condition is a cone containing the deformation direction.
%
%
The moduli spaces of split maps are honest products with respect to the subgraphs obtained by cutting
edges, rather than fiber products cut out by matching conditions. Our
approach of splitting the diagonal resembles Bourgeois' approach
\cite{bo:com} of degenerating matching conditions in symplectic field
theory, and the Fulton-Sturmfels formula \cite{fulton} for the
splitting of a diagonal in a toric variety; the relations to these
results are explored in Section \ref{subsec:fs}.

We introduce some notation. Let $\XX$ be the broken manifold
that is the neck-stretching limit of a symplectic manifold $X$. Let
$L \subset X$ be a Lagrangian submanifold that is disjoint from the
necks, and let $CF_\br(\XX,L)$ be the broken Fukaya algebra.

Any map in a broken manifold, whether broken, deformed, or split, has
a \em{tropical symmetry group}, obtained by applying one of the
translational symmetries on each component in a way that preserves the
matching conditions. 
The tropical symmetry group of a rigid
broken/deformed map (that is, maps of index zero) is finite.
For split maps, the tropical symmetry group is larger, because there is no matching condition on the split edges.
In particular, 
for a
converging sequence of rigid deformed maps, the dimension of the
tropical symmetry group of the limit split map is equal to the
codimension of the matching condition at split nodes.  We define the
composition maps of the \em{split Fukaya algebra}
$CF_{\split}(\XX,L,{\etasp})$ 
of a broken manifold $\XX$ 
via counts of tropical symmetry orbits
of split maps,
for which the discrepancy cone of the split tropical graph
contains the deformation direction 
%
as in Definition \ref{def:splitgr}.
Split maps are informally described in Section \ref{subsec:idea}.

\begin{theorem} \label{thm:tfuk} Let ${\etasp} \in \t$ be a generic
  direction of deformation.  The broken Fukaya algebra
  $CF_{\br}(\XX,L)$ is $A_\infty$-homotopy equivalent to the split
  Fukaya algebra $CF_{\split}(\XX, L, {\etasp})$.
\end{theorem}
\noindent After the first version of this preprint appeared, similar
results appeared Wu \cite{wu} in the algebro-geometric context.  The
proof of the Theorem for split maps is completed in Proposition
\ref{prop:hequiv} below.

Using the split Fukaya algebra, we prove unobstructedness of certain
Lagrangian tori, generalizing a result of Fukaya-Oh-Ohta-Ono result
\cite{fooo:toric1} on Lagrangian torus orbits in toric varieties.
By our assumption that the Lagrangian $L$ is away from the necks, $L$ descends to a Lagrangian
submanifold in a component $X_P$ of the broken manifold $\XX$ that is not a neck piece.
Furthermore, such a Lagrangian $L \subset X$ is a \em{tropical torus} if
$\ol X_P$ is a toric
manifold (that is, is a compact symplectic manifold with a completely
integrable Hamiltonian torus action), whose torus-invariant divisors
are relative divisors, and $L \subset X_P$ is a Lagrangian torus
orbit.  
A tropical torus
appears to be a somewhat stronger requirement than the notion of
moment fiber in a toric degeneration used in mirror symmetry, where
``discriminant'' singularities are allowed as in Gross
\cite{grosstoric}.

Homotopy equivalence with the split Fukaya algebra is used to prove
that tropical tori are weakly unobstructed. We describe the term `weak
unobstructedness': The geometric Fukaya algebra generated by the
critical points of a Morse function on the Lagrangian $L$ can be
enlarged to a Fukaya algebra $CF(X,L)$ in order to include a strict unit
$1_L \in CF(X,L)$ by the \em{homotopy unit construction} as in
Fukaya-Oh-Ohta-Ono \cite[(3.3.5.2)]{fooo}.
The cohomology $H(A)$ of a (possibly curved) \ainfty-algebra $A$ is
well-defined if the first order composition map $m_1$ satisfies
$m_1^2=0$. The condition $m_1^2=0$ may fail to hold if the
\em{curvature} $m_0(1) \in A$ is not a $\Lam$-multiple of the unit
$1_L$, which is an `obstruction' to the definition of Floer
cohomology.  \em{Weak unobstructedness} is a more general condition
under which the Floer cohomology of a Lagrangian brane can be
defined. A Lagrangian brane $L$ is \em{weakly unobstructed} if the
projective Maurer-Cartan equation
\begin{equation} \label{eq:mc-intro} m_0(1) + m_1(b) + m_2(b,b) +
  \ldots = W(b) 1_{L} \quad \text{for some } W(b) \in \Lambda
  .\end{equation}
has an odd solution $b \in CF(X,L)$. Any odd solution to the
Maurer-Cartan equation is called a \em{weakly bounding cochain} and
the set of all the odd solutions is denoted $MC(L)$.  Given a weakly
bounding cochain, the Fukaya algebra $CF(X,L)$ may be `deformed' by
$b$ (see \eqref{T-eq:mnb} of \cite{vw:trop}) to yield an
\ainfty-algebra $CF(X,L,b)$ with composition maps $(m_n^b)_{n \geq 0}$
satisfying
\[m_0^b=W(b) 1_L, \quad \text{for some} \quad W(b) \in \Lam.\]
Consequently, $(m_1^b)^2=0$, and the Floer cohomology
\[HF(L,b):=\ker(m_1^b)/\im(m_1^b)\]
is well-defined.  The function
\[ W: MC(L) \to \Lambda , \quad b \mapsto W(b) \]
is called the \em{potential} 
of the curved \ainfty algebra $CF(X,L,b)$.
The same discussion applies to $CF_\br(\XX,L)$ and $CF_\spl(\XX,L,\etasp)$ also.

The following result shows that a tropical Lagrangian is unobstructed,
and it has a distinguished solution of the Maurer-Cartan equation.

\begin{corollary} \label{cor:mc} {\rm(Unobstructedness of a tropical
    torus)} Suppose that $L \subset X_{P_0}$ is a tropical torus
  equipped with a brane structure in a broken manifold $\XX$.  For
  any generic cone direction ${\etasp} \in \t$, the split Fukaya
  algebra $CF_{\split}(\XX,L,{\etasp})$ is weakly unobstructed with a
  distinguished solution $b \in MC(L)$ of the Maurer-Cartan equation.
  Consequently, $CF(X,L)$ is also weakly unobstructed.
\end{corollary}

Corollary \ref{cor:mc} is re-stated and proved as Proposition
\ref{prop:mc-restate} in Section \ref{sec:unobst}. The proof relies on
the \ainfty homotopy equivalence from Theorem \ref{thm:tfuk}, and the
fact that weak unobstructedness is preserved under homotopy
equivalence. In fact, a solution $w x^{\greyt} \in  CF_{\split}(\XX,L,{\etasp})$ of the Maurer-Cartan equation for the split Fukaya algebra yields a solution
$\F(w x^{\greyt}) \in CF(X,L)$ for the Maurer-Cartan equation on the unbroken
Fukaya algebra, where $\F : CF_{\split}(\XX,L,{\etasp}) \to CF(X,L)$
is a $A_\infty$-homotopy equivalence.

Corollary \ref{cor:mc}, when specialized to the case of toric
manifolds, gives an alternate proof of the result of
Fukaya-Oh-Ohta-Ono \cite{fooo:toric1} on the unobstructedness of toric
Lagrangians in a toric manifold; we state this result as Corollary
\ref{cor:fooo} and prove it as part of Proposition \ref{prop:blas}.
Additionally, in the toric case, the solution $b \in MC(L)$ of the Maurer-Cartan
solution given by Corollary \ref{cor:mc} is such that the leading
order terms in the potential $CF_{\split}(\XX,L,b)$ coincide with the
leading order terms in the Batyrev-Givental potential
\eqref{eq:bgpot}.

\begin{corollary}
  \label{cor:fooo}
  {\rm(Disk potentials in a toric manifold)} Let $X$ be a symplectic
  toric manifold, and let $L \subset X$ be a toric Lagrangian brane.
  Then, $CF(X,L)$ is weakly unobstructed. Furthermore, there is a
  solution $b \in MC(L)$ of the Maurer-Cartan equation for which the
  leading order terms in the potential $W(b)$ coincide with the
  leading order terms of the Batyrev-Givental potential
  \eqref{eq:bgpot}.
\end{corollary}

The paper is organized as follows. Background from \cite{vw:trop} is
summarized in Section \ref{sec:background}. The proof of Theorem
\ref{thm:tfuk} takes up Sections \ref{sec:defmap}-
\ref{sec:homo2split}. Corollaries \ref{cor:mc} and \ref{cor:fooo} are
proved in Section \ref{sec:apps}.

\section{Background}\label{sec:background}
We recall some background on broken manifolds and broken maps from
\cite{vw:trop}.  We freely use the definitions and results from
\cite{vw:trop}, with the convention that a reference of the form
T-(equation number), Theorem T-(theorem number) and so on refers to
the earlier paper.
\subsection{Broken manifold}   

We start by recalling the notion of a broken manifold associated with a tropical Hamiltonian action.  Let $n >0$ be an integer and $T \simeq (S^1)^n$ a torus with Lie
algebra $\t \cong \R^n$.  A \em{simplicial polyhedral decomposition}
of $\t^\dual$ is a collection
\[ \PP = \{ P \subset \t^\dual \} \]
of simple polytopes whose interiors $P^\circ$ cover $\t^\dual$, with
the property that for any $\sig_1, \dots, \sig_k \in \PP$,
the intersection
$\sig_1 \cap \dots \cap \sig_k$ is a polytope in $\PP$ that is a face
of each of the polytopes $\sig_1,\dots,\sig_k$.
Any polytope $P$ in a simplicial polyhedral decomposition $\PP$
corresponds to a sub-torus $T_P \subseteq T$ whose Lie algebra is
$\t_P:=\ann(TP)$.
Thus $P$ and $T_P$ have complementary dimensions, and $T_P=\{\Id\}$
for top-dimensional polytopes $P \in \PP$.

\begin{definition} \label{def:tropmanifold-sp} {\rm(Tropical
    Hamiltonian action)} A \em{tropical Hamiltonian action}
  $(X,\PP,\Phi)$ consists of
  \begin{enumerate}
  \item \label{part:trop1} a simplicial polyhedral decomposition $\PP$
    of $\t^\dual$, and
  \item \label{part:trop2} a compact symplectic manifold $(X,\om_X)$
    with a \em{tropical moment map}
    \[\Phi:X \to \t^\dual, \]
    such for any $P \in \PP$ there is a neighborhood $U_P \subset X$
    of $\Phinv(P)$ on which the projection
    \[U_P \to \t^\dual \to \t_P^\dual\]
    is a moment map for a free Hamiltonian action of a torus $T_P$.
  \end{enumerate}
 We assume, without stating, that any polyhedral
decomposition we encounter possesses a cutting datum, which is described below.
\end{definition}

\noindent

Given a tropical Hamiltonian action $X$ the output of a multiple cut
is a collection of \em{cut spaces}
\begin{equation}
  \label{eq:cutspace}
  X_P:=\Phinv(P^\circ), \quad P \in \PP, 
\end{equation}
where $P^\circ$ is the complement of the faces of $P$; the cut space
$X_P$ compactifies to an orbifold
\[\ol X_P=\Phinv(P)/\sim,\]
where $\Phinv(P)$ is a manifold with corners and the equivalence
relation $\sim$ quotients any codimension one boundary
$\Phinv(Q) \subset \Phinv(P)$, $\codim_P(Q)=1$ by the action of
$S^1 \simeq T_Q/T_P$. The boundary $\ol X_P \bs X_P$ of the compactification is a union of submanifolds $\ol X_Q$, $Q \subset P$, called \em{relative submanifolds}.  

A \em{broken manifold} corresponding to a multiple cut $(X,\PP)$ is a
disjoint union of thickenings of cut spaces
\begin{equation} \label{eq:xx} \XX:=\XX_\PP:=\bigsqcup_{P \in
    \PP}\XC_P.\end{equation}
For top-dimensional polytopes $P \in \PP$, $\XX_P:=X_P$ is the cut
space from \eqref{eq:cutspace}, and for lower dimensional polytopes
$P$, $\XX_P$ is a torus bundle over the cut space $X_P$
\begin{equation}
  \label{eq:xp-proj}
T_{P,\C} \to \XX_P \xrightarrow{\pi_{X_P}} X_P,   
\end{equation}
whose fiber $T_{P,\C}$ is the complexification of the compact torus
$T_P$.  The compactification of $\XC_P$ is a fibration
$\ol \XC_P \to \ol X_P$ with toric orbifold fibers.

\begin{figure}[h]
  \centering \scalebox{.8}{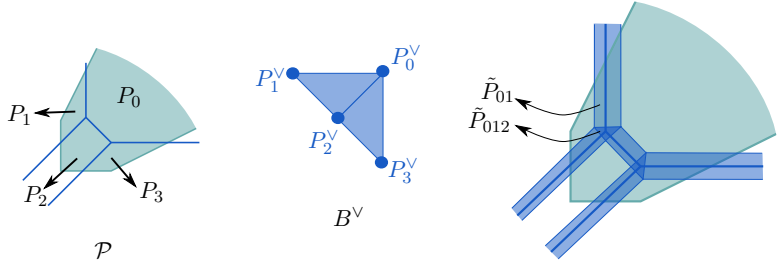}
  \caption{A polyhedral decomposition $\PP$, its dual complex $B^\dual$ and a cutting datum.}
  \label{fig:cutting}
\end{figure}

The broken manifold is equipped with the datum of a
\em{dual complex}
and a \em{cutting datum}. 
The dual complex 
 encodes the proportion in which the neck in
$X$ is stretched in different directions in order to produce the
broken manifold.  The dual complex
\[B^\dual= \cup_{P \in \PP}P^\dual/\sim,\]
consists of a complementary dimensional polytope denoted
$P^\dual \subset \t_P$ for every $P \in \PP$. For any pair of
polytopes $Q \subset P$, the dual polytope $Q^\dual$ has a face that
is (affine) isomorphic to $P^\dual$, and the equivalence $\sim$
identifies this face to $P^\dual$. Any polyhedral decomposition has a dual complex.
To realize broken manifolds as a
limit of neck-stretched manifolds, we also assume that
the polyhedral decomposition 
 has a 
cutting datum corresponding to the dual complex, which is needed for
identifying thickenings of the cut loci to subsets of the symplectic
manifold $X$. See Figure \ref{fig:cutting} and T-Section \ref{T-sec:symp-broken} for detailed definitions.

Components of broken manifolds have cylindrical structures on them and
some resulting translational symmetries.  For a polytope $P$, the
broken manifold $\XX_P$ is a $P$-cylinder in the sense that
\[\XX_P=Z_P \times \t_P\]
where 
$Z_P := \Phinv(P)$ is a $T_P$-principal bundle over $X_P$.
An element $t \in \t_P$ corresponds to a \em{translation} on $\XX_P$
\begin{equation}
  \label{eq:ptrans}
  e^t(z,\tau)=(z,\tau+ t), \quad z \in Z_P, \tau \in \t_P.  
\end{equation}
The manifolds $X_P$, $\XX_P$ have cylindrical ends in the following
sense: For any proper face $Q \subset P$, $Q \in \PP$, there is a neighborhood of
$U_{\ol X_Q} \ol X_P$ of $\ol X_Q$ in $\ol X_P$ such that for
$U_Q(X_P):= U_{\ol X_Q} \ol X_P \cap X_P$,
there is an embedding
\begin{equation}
  \label{eq:p2q}
  U_Q(\XX_P) \to \XX_Q, \quad \text{where } U_Q(\XX_P):=\pi_P^{-1}(U_Q(X_P)).   
\end{equation}
The subsets $U_Q(X_P) \subset X_P$ and
$U_Q(\XX_P) \subset \XX_P$ are called $Q$-cylindrical ends and are
equipped with $Q$-translation maps
\begin{equation}
  \label{eq:pqtrans}
  e^{-t} : U_Q(\XX_P) \to \XX_Q  
\end{equation}
all $t \in \Cone_{P^\dual}Q^\dual \subset \t_Q$ defined by composing the embedding in
\eqref{eq:p2q} with the translation in \eqref{eq:ptrans}.

Broken maps to a broken manifold are defined using a suitable class of
almost complex structures.  An almost complex structure $\JJ$ on a
broken manifold consists of an almost complex structure $\JJ_P$ on all
the compactified components $\ol \XX_P$. The almost complex structure
$\JJ$ are required to be \em{cylindrical}: On $\XX_P$, $\JJ_P$ is
$P$-cylindrical in the sense that it is $T_{P,\C}$-equivariant and the
fibers of $\pi_{X_P}$ in \eqref{eq:xp-proj} are biholomorphic to the
complex torus $T_{P,\C} \simeq (\C^\times)^k$. Automatically, there is a \em{base
  almost complex structure} $J_P:=\d\pi_{X_P}(\JJ_P)$ on $\ol X_P$
such that $\pi_{X_P}$ is
$(\JJ_P,J_P)$-holomorphic.  For $Q \subset P$, $\JJ|\XX_P$ is
$Q$-cylindrical on $U_Q(\XX_P)$.  In the terminology of
\cite[Definition 5.9]{vw:trop}, such an almost complex structure $\JJ$
is $\XX$-cylindrical.  We denote by
\begin{equation}
  \label{eq:omxxtau}
  \J^\cyl_{\om_\XX,\tau}(\XX)
\end{equation}
the space of such cylindrical $\om_\XX$-tamed almost complex structures on $\XX$; $\om_\XX$-tamedness of $\JJ$ means that the base almost complex structure $J_P$ is $\om_{X_P}$-tamed for every $P \in \PP$.

\subsection{Broken maps}
Broken maps in a broken manifold arise as limits of sequences of holomorphic maps in a tropical manifold with neck-stretched almost complex structures. 
A broken map $u$ is modelled on a graph $\Gamma$ and consists of a
collection of holomorphic maps and a tropical structure on $\Gamma$,
which is then called a \em{tropical graph}.  The ``map part'' of $u$
is a collection of holomorphic maps on punctured curves
\[ u_v:C_v^\circ \to \XC_{P(v)}, \quad v \in \Ver(\Gamma), \]
that satisfy certain matching conditions (explained later in the
paragraph) on the lifts of nodal points.  Here, for any vertex $v$ of
$\Gamma$,
\[C_v^\circ :=C_v \bs \{\text{interior nodes}\}, \quad v \in
  \Ver(\Gamma),\]
is a punctured curve (possibly with boundary), each $C_v \subset C$ is
an irreducible component of a nodal curve $C$ modelled on $\Gamma$,
$P(v) \in \PP$ is a polytope, and thus, $\XX_{P(v)} \subset \XX$ is a
piece of the broken manifold.  The punctures in the domain are
removable singularities when $u_v$ is viewed as a map to the
compactification $\ol \XC_{ P(v)}$.  For any edge $e=(v_+,v_-)$ of
$\Gamma$, there is an assigned polytope
\[P(e) := P(v_+) \cap P(v_-) \in \PP,\] and an \em{edge direction}
$\cT_{e,\pm}$, which is an element in the integer lattice
$\t_{P(e),\Z} \subset \t_{P(e)}$.  The matching condition at a node
$w_e \in C$ corresponding to the edge $e=(v_+,v_-) \in \Edge(\Gamma)$
is that the map $(\pi_{\cT(e)}^\perp \circ u_{v_\pm})$ has a removable
singularity at the node $w_e^\pm$, and
\begin{multline}
  \label{eq:match-proj}
  \text{(Projected matching condition)} \\
  (\pi_{\cT(e)}^\perp \circ u_{v_+})(w_e^+)=(\pi_{\cT(e)}^\perp \circ u_{v_-})(w_e^-) \in \XC_{P(e)}/T_{\cT(e),\C}.
\end{multline}
The quantities in the left-hand side and right-hand side of
\eqref{eq:match-proj} are called \em{projected tropical
  evaluations}.  The ``tropical part'' of $u$ consists of polytope
assignments $P(v)$ for the vertices $v \in \Ver(\Gamma)$ and the
directions $\cT(e)$ for the edges $e \in \Edge(\Gamma)$, which together
constitute the \em{tropical graph} of the broken map. These are
required to be \em{realizable} in the following sense: There is a
collection of \em{tropical vertex positions}
\begin{equation} \label{eq:tv} \cT(v) \in P(v)^{\dual,\circ} \subset
  B^\dual, \quad v \in \Ver(\Gamma)
\end{equation}
satisfying the following \em{direction condition}
\begin{equation} \label{eq:bslope} {\text{\rm (Direction condition)}}
  \quad \cT(v_+) - \cT(v_-) \in \R_{> 0} \cT(e) .
\end{equation}
Changing the tropical vertex positions $\{\cT(v)\}_v$ produces a
different realization of the same tropical graph; an example $\Gamma_2$ is shown
in Figure \ref{fig:rigid}.
\begin{figure}[ht]
  \centering \scalebox{1}{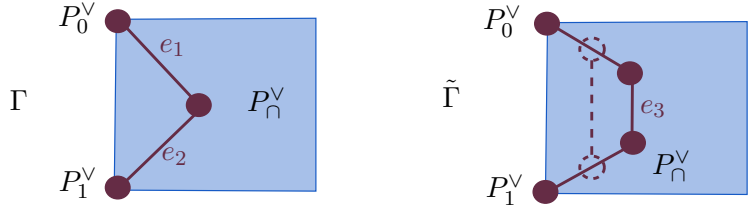}
  \caption{The tropical graph $\Gamma_1$ is rigid, but
    $\Gamma_2$ is not rigid since the vertices inside the
    square can be moved to the dotted positions.}
  \label{fig:rigid}
\end{figure}

Broken maps may also contain nodes corresponding to the standard nodal
degeneration encountered in Gromov-Witten theory. The edges
corresponding to such nodes do not appear in the tropical graph, and
thus, the set of edges of $\Gamma$ has a partition
\[\Edge(\Gamma)=\Edge_\trop(\Gamma) \cup \Edge_{\on{int}}(\Gamma)\]
into \em{tropical edges} and \em{internal edges}.  The object $u$
defined so far is an \em{unframed broken map}.

A \em{broken map} consists of an unframed broken map as in the
previous paragraph and a \em{framing} defined as follows.  For an
unframed broken map, at any tropical node $w_e$ corresponding to an
edge $e$, there exist holomorphic coordinates $z_+$, $z_-$ in
neighborhoods of the nodal lifts $w_e^+$, $w_e^-$ such that
\begin{equation} \label{eq:nodematch-intro}
  \text{(Matching condition)} \quad \lim_{z_+ \to  0}z_+^{-\cT(e)}u=\lim_{z_- \to 0}z_-^{-\cT(e)}u.
\end{equation}
The coordinates $z_+$, $z_-$ are called \em{matching coordinates} at
the node $w_e$.  The quantities in the left-hand side and right-hand
side of \eqref{eq:nodematch-intro} are called the \em{tropical
  evaluations} at $w_e^+$ and $w_e^-$ respectively.  A \em{framing} is a
linear isomorphism
\begin{equation} \label{eq:framingeq} \fr_e : T_{w^+_e} C_{v_+}
  \tensor T_{w^-_e} C_{v_-} \to \C,
\end{equation}
such that any pair of holomorphic coordinates $z_+$, $z_-$ in the
neighborhood of $w_e^+$, $w_e^-$ satisfying
\begin{equation}
  \label{eq:framingcoord}
  \d z_+(w^+_e) \tensor \d z_-(w^-_e)=\fr_e
\end{equation}
are matching coordinates as in \eqref{eq:nodematch-intro}.  An
unframed broken map has a finite number of framings as explained in
Remark T-\ref{T-rem:nfr}.

\subsection{Tropical symmetry}
The set of broken maps of a fixed combinatorial type has the action of a tropical symmetry group arising from translations on the neck regions. 
  For a tropical graph $\Gamma$, the tropical symmetry group
  $T_{\on{trop}}(\Gamma)$, from Definition T-\ref{T-def:tsym}, 
  is generated by the degrees of freedom of
  $\Gamma$: Each element of the real part of this
  group corresponds to ways of moving the vertex positions
  $\{\cT(v)\}_v$ without changing edge directions $\{\cT(e)\}_e$, as
  shown in Figure \ref{fig:rigid}.  In particular, the symmetry group
  $T_{\on{trop}}(\Gamma)$ is finite if in the tropical graph vertex
  positions $\cT(v)$ are uniquely \index{Rigid!  tropical graph}
  determined by the edge directions $\cT(e)$.  Such tropical graphs
  are called \em{rigid}.  In Figure
  \ref{fig:rigid}, there is one degree of freedom in moving the edges of
  the tropical graph $\Gamma_2$, and this generates a one-dimensional complex
  torus $T_\trop(\Gamma_2)$.

\subsection{Broken treed holomorphic disks}\label{subsec:broken-hol-disks}
Broken treed holomorphic disks combine the features of broken maps and
treed holomorphic disks, where the latter is similar to pearly
trajectories in Biran-Cornea \cite{bc:ql}.  The domain is a \em{treed
  disk} (Definition T-\ref{T-def:treeddisk}), which is a marked nodal disk with disk nodes
replaced by tree segments, and semi-infinite tree segments attached at
boundary markings.  The \em{type} of a treed disk is a graph $\Gamma$
underlying the marked nodal disk, whose vertex set has a partition
\[\Ver(\Gamma)=\Ver_\white(\Gamma) \cup \Ver_\black(\Gamma)\]
into subsets of vertices corresponding to disk components and sphere
components respectively, and the edge set has a partition
\[\Edge(\Gamma)=\Edge_\white(\Gamma) \cup \Edge_\black(\Gamma)\]
into subsets of edges corresponding to boundary nodes and interior
nodes respectively, and the edges in $\Edge_\white(\Gamma)$ have the
additional data of whether each edge has length zero, finite non-zero,
or infinite.  The treed disk $C=S \cup T$ is a union of the surface
part $S$ and a tree part $T$, each made up of components
\[S=\cup_{v \in \Ver(\Gamma)}S_v, \quad T=\cup_{e \in
    \Edge(\Gamma)}T_e.\]
The moduli space of curves of type $\Gamma$ is denoted by $\M_\Gamma$
with compactification $\ol \M_\Gamma$ which is a smooth manifold. The
universal curve $\ol \U_\Gamma \to \ol \M_\Gamma$ partitions into the
surface part $\ol \S_\Gamma$ and tree part $\ol \T_\Gamma$, that is,
$\ol \U_\Gamma=\ol \S_\Gamma \cup \ol \T_\Gamma$.

Broken maps from disks are required to satisfy a Lagrangian boundary condition. 
We consider a Lagrangian submanifold $L \subset X_{P_0}$ with brane
structure in a top-dimensional cut space $X_{P_0}$ of the broken
manifold $\XX$, that is equipped with a Morse function $F: L \to \R$.  A
\em{broken treed holomorphic disk} is a broken holomorphic disk in the
broken manifold $\XX$ whose boundary lies in $L$, with the additional
feature that treed segments map to gradient trajectories of $F$ in the
Lagrangian submanifold. See Figure T-\ref{T-fig:brokendisk}.  To ensure that the moduli
spaces of broken disks are transversally cut out, the almost complex
structure $J$ and Morse function $F$ are perturbed in a
domain-dependent way using the Cieliebak-Mohnke perturbation scheme,
which uses a broken Donaldson-type divisor in $\XX$, which we call the
\em{stabilizing divisor}.  To define moduli spaces of broken maps, we
fix a coherent system of perturbations $\ul \Pe=(\Pe_\Gamma)_\Gamma$
where $\Gamma$ ranges over all domain treed disk types (without
tropical structure).  The \em{type} $\Gamma$ of a broken treed disk
consists of the domain treed disk type, the tropical structure, and
the intersection multiplicities with the stabilizing divisor at
interior markings.  The \em{moduli space of broken treed holomorphic
  disks} of type $\Gamma$ and leaves labelled by
\[ \ul x = ( x_0,\dots, x_{d(\white)} \in \crit(F))  \] 
is denoted by
$\M^\br_\Gamma(L, \ul \Pe, \ul x)$. The tropical symmetry group
$T_\trop(\Gamma)$ acts freely on the moduli space of broken
holomorphic disks, and the quotient
\[\M_{\Gamma,\red}^\br(\XX,L,\Gamma, \ul x):= \M^\br_\Gamma(\XX,L,\ul
  x)/T_\trop(\Gamma).\]
is called the \em{reduced moduli space}.  A broken holomorphic treed
disk type $\Gamma$ is \em{rigid} if it can not be deformed to a
different type, that is, if all its non-tropical nodes are disk nodes
with a treed segment of finite non-zero length, and its tropical graph
is rigid, and all intersections with the stabilizing divisor are
simple.

\subsection{Convergence of broken maps}\label{subsec:conv}
We describe the notion of Gromov convergence for sequences of broken maps.  
The Gromov limit of a sequence of broken disks of type $\Gamma$ may have configurations with a different tropical graph $\Gamma'$, in which case, the graph $\Gamma$ can be recovered from $\Gamma'$ by
collapsing a subset of edges. The edge collapse $\Gamma' \to \Gamma$
respects the tropical structure in the following sense:

A \em{tropical edge collapse} is a morphism of tropical graphs
$\Gamma' \xrightarrow{\kappa} \Gamma$ that collapses a subset of edges
$\Edge(\Gamma') \backslash \Edge(\Gamma)$ in $\Gamma'$ inducing a
surjective map on the vertex sets
\[ \kappa:\Ver(\Gamma') \to \Ver(\Gamma), \]
and satisfies the following conditions:
\begin{enumerate}
\item for any vertex $v \in \Ver(\Gamma')$,
  $P(v) \subseteq P(\kappa(v))$; and
\item the edge {direction} is unchanged for uncollapsed edges, i.e. if
  $\cT$, $\cT'$ are the edge {direction} functions for $\Gamma$,
  $\Gamma'$, then $\cT(\kappa(e)) = \cT'(e)$ for any uncollapsed edge
  $e \in \Edge(\Gamma')$.
\end{enumerate}

Convergence of sequences of broken maps is modulo
translation in the target space, as some components
may escape into cylindrical ends in the limit. 
For example, for a sequence of maps $u_\nu$ of type $\Gamma$, and a
vertex $v$ of $\Gamma$, if the components
$u_{\nu,v} : C_\nu \to X_{P(v)}$ collapse into the $Q$-cylindrical end
for some $Q \subset P$ as $\nu \to \infty$, the corresponding Gromov limit is given by the convergence of the sequence of maps $e^{-t_\nu(v)}u_{\nu,v}$
where $t_\nu(v) \in \t_Q$ is a translation in the $Q$-cylinder (see \eqref{eq:p2q}).
For any $\nu$, the collection of translations $t_{\nu,v}$ ranging over all vertices $v$ satisfies the direction condition, and is called a  \em{relative translation sequence} (called so, because
the translations are `relative' to positions in $\Gamma$) defined as below. 

\begin{definition} \label{def:relweight} {\rm(Relative translations)}
  \index{Translation!Relative translation} Suppose
  $\kappa: \Gamma' \to \Gamma$ is a tropical edge-collapse morphism.
  For any vertex $v$ of $\Gamma'$, let
  \begin{multline}\label{eq:conekv}
    \Cone(\kappa,v) := \Cone_{ P(\kappa(v))^\dual}(P(v)^\dual)\\
    :=\R_{\geq 0}\{t - t_0 \in \t: t \in P(v)^\dual, t_0 \in
    P(\kappa v)^\dual,\alpha \in \R_{\geq 0}\} \subset \t_{P(v)}
  \end{multline}
 A \em{relative translation} or a  $(\Gamma',\Gamma)$-translation is an element
  \(t =(t(v) \in \Cone(\kappa,v))_{v \in \Ver(\Gamma')} \)
  satisfying
  \begin{equation}\label{eq:direction-rel} \text{(Direction condition)} \quad 
    t(v_+) - t(v_-) \in
    \begin{cases}
      \R_{\geq 0} \cT(e), & e \in \Edge(\Gamma') \bs \Edge(\Gamma),\\
      \R\cT(e), & e \in \Edge(\Gamma)
    \end{cases}
  \end{equation}
  for any edge $e=(v_+,v_-)$ in $\Gamma'$. The space of $(\Gamma',\Gamma)$-translations is denoted by $w(\Gamma',\Gamma)$. 
 \end{definition}

Relative translations are closely related to ``relative vertex
positions'', which are differences between vertex positions of
$\Gamma'$ and $\Gamma$ (Definition \ref{def:rvp}): Given a
$(\Gamma',\Gamma)$-translation $t$ and a vertex position $\cT$ for the
graph $\Gamma$, for small enough $\eps>0$, $\cT + \eps t$ is a vertex
position for $\Gamma'$. In other words, $w(\Gamma',\Gamma)$ is the
cone generated by relative vertex positions.

\begin{example}
  In Figure \ref{fig:rigid}, collapsing the edge $e$ gives a tropical edge collapse $\kappa : \Gamma_2 \to \Gamma_1$. The set of vertex position maps for $\Gamma_2$, denoted by $\W(\Gamma_2)$,  is $(0,1)$. Here, $0$ resp. $1$ corresponds to the vertex position  map where the edge $e_3$ resp. edges $e_1$ and $e_2$ have lengths zero. The cone $\fw(\tGam,\Gamma)$ generated by relative vertex positions is $\Cone_0[0,1] \simeq \R_{\geq 0}$. 
\end{example}

With this terminology, we can now state the definition
of Gromov convergence.   
A sequence of broken maps $u_\nu : C_\nu \to \XX$ of type $\Gamma$ Gromov converges to a map $u$ of type $\Gamma'$ if there is a tropical
edge collapse $\kappa: \Gamma' \to \Gamma$ and a sequence $(t_\nu)_\nu$ of relative
translations such that the for any $v \in \Ver(\Gamma')$, the
translated sequence of maps $e^{-t_\nu(v)}u_{\nu,\kappa(v)}$ converges
to $u_v : C_v \to \XX_{P(v)}$ under appropriate domain
identifications.  If the tropical graphs $\Gamma'$, $\Gamma$ are
distinct, then, a non-zero relative translation $t$ generates a
subgroup $T_{\trop}(\Gamma')$ that is not contained in
$T_{\trop}(\Gamma)$, and therefore, the former has higher dimension
than the latter.
  
\subsection{Boundary strata}

To define broken Fukaya algebras, we consider moduli spaces of broken
maps of rigid type with expected dimension zero and one. We describe
the configurations in the compactification of such a stratum
$\M^\br(\XX,L,\Gamma, \ul x)$. The compactification
$\ol \M^\br(\XX,L,\Gamma, \ul x)$ does not contain configurations with
additional interior nodes than those of $\Gamma$. Indeed, non-tropical
interior nodes are ruled out because they occur in codimension two
strata, and tropical nodes are ruled out as follows: Suppose a
sequence of maps of type $\Gamma$ converges to a limit of map type
$\Gamma'$.  Additional tropical nodes do not contribute negatively to
the expected dimension formula for broken maps because the matching
condition at a tropical edge has codimension $(\dim X - 2)$, and
therefore, the moduli spaces $\M^\br(\Gamma, \ul x)$ 
and
$\M^\br(\Gamma', \ul x)$ have the same dimension, which is at most
one.  However, the tropical graph $\Gamma'$, when it is distinct from
$\Gamma$, necessarily has a larger tropical symmetry group. Since
$\Gamma$ is rigid, $T_\trop(\Gamma)$ is zero-dimensional, and
$T_\trop(\Gamma')$ is at least two-dimensional. We obtain a
contradiction because $T_\trop(\Gamma')$ acts freely on
$\M^\br(\XX,L,\Gamma', \ul x)$.  Therefore, the codimension one
boundary in the compactification $\ol \M^\br(\XX,L,\Gamma', \ul x)$
consists of configurations with a disk node whose treed segment has
length zero or infinity.  Of these two possibilities, a stratum
containing a zero length edge is actually a \em{fake boundary stratum}
as it is the boundary of two different rigid strata with opposite
induced boundary orientations. Thus, the \em{true boundary strata} are
those that have a broken boundary edge. (See Section
T-\ref{T-sec:tubular} and Figure T-\ref{T-fig:true-fake}.)

  \subsection{Broken Fukaya algebras}

Counts of rigid broken maps define the broken Fukaya algebra, as follows. 
  For a coherent system of perturbations $\ul \Pe$, the \em{broken
    Fukaya algebra} is an \ainfty 
    algebra whose underlying graded
  vector space is
  \[CF_{\br}(L, \ul \Pe):=CF^{\on{geom}}(L, \ul \Pe) \oplus \Lam
    x^{\greyt}[1] \oplus \Lam x^{\whitet},\]
  where $CF^{\on{geom}}(L,\ul \Pe)$ is the vector space generated by
  the set of critical points of the Morse function $F$ over the
  Novikov ring $\Lam_{\geq 0}$ (see \cite[Section 10.1]{vw:trop}), and
  the additional generators $x^{\greyt}$, $x^{\whitet}$ are added as
  part of the homotopy unit construction (see \cite[Section
  10.3]{vw:trop}) so that the perturbation systems admit forgetful
  maps, and the \ainfty algebra has a strict unit.  For any
  $d(\white) \geq 0$, the $d(\white)$-ary composition maps on the
  broken Fukaya algebra is
  \begin{equation*}
    m^{d_\white}_\br(x_1,\dots,x_{d_\white})= \sum_{x_0,u \in {\tM^\br_{\Gamma}(\ul \Pe, \ul{x})_0}} w_s(u) x_0,
  \end{equation*}
  where the type $\Gamma$ of $u$ ranges over all rigid broken map
  types with $d_\white$ inputs, and
  \begin{equation}
    \label{eq:wtwsubr}
    w_s(u):=(-1)^{\heartsuit}(d_\black(\Gam)!)^{-1}  \Hol([\partial u]) \eps(u)
    q^{A(u)},
  \end{equation}
  where $d_\black(\Gam)$ is the number of interior markings, 
  the holonomy $\Hol$ is part of the brane structure of the Lagrangian and the sign $\eps(u) \in \{\pm 1\}$ is as in
  T-\eqref{T-eq:wtwu}. The \ainfty
  algebra $CF_\br(\ul \Pe, L)$ is convergent and has a strict unit
  $x^\whitet$.

  The main result of \cite{vw:trop} is that the broken Fukaya algebra
  $CF_\br(\XX,L)$ is \ainfty homotopy equivalent to the ordinary
  Fukaya algebra $CF(X,L)$, whose composition maps are defined by
  counting treed holomorphic disks in $X$ with boundary in $L$.

  \begin{theorem} \label{thm:bfuk} For a rational Lagrangian
    submanifold $L \subset X$ and polyhedral decomposition $\cP$ as
    above, the unbroken Fukaya algebra $CF(X,L)$ admits a curved
    $A_\infty$ homotopy equivalence to the broken Fukaya algebra
    $CF_{\br}(\XX,L)$.
  \end{theorem}

  \section{Deformed maps}\label{sec:defmap}

  In this section, we introduce deformed maps which are defined
  similarly to broken maps but where the edge matching condition for
  split edges is replaced by a deformed matching condition. We prove
  that the \ainfty algebra defined by counts of deformed maps, called
  the \em{deformed Fukaya algebra}, is \ainfty homotopy equivalent to
  the broken Fukaya algebra.

  The matching condition is deformed on a subset of edges called
  \em{split edges}.  Roughly speaking, these are edges corresponding
  to neck pieces whose target is the complexified torus.
  \begin{definition}{\rm(Split edges)}
    \label{def:splitedge}
    Let $\PP_s \subset \PP$ be the set of polytopes $P$ for which
    \begin{enumerate}
    \item the tropical moment map $\Phi$ generates a $T/T_P$-action on
      the cut space $\ol X_P$ and this action makes $\ol X_P$ a toric
      manifold,
    \item and any torus-invariant divisor $D$ of $\ol X_P$ is a
      relative divisor. That is, there is a polytope $Q \subset P$
      such that $\ol X_Q=D$.
    \end{enumerate}
    For a tropical graph $\bGam$, $e \in \Edge(\bGam)$ is a \em{split
      edge} if $P(e) \in \PP_s$. The set of split edges of a tropical
    graph is denoted by
    \[\Edge_s(\bGam) \subset \Edge(\bGam).\]
  \end{definition}

  By definition, a split edge $e$ satisfies the property that the
  space $\ol \XC_{P(e)}$ is a toric variety that is a compactification
  of $T_\C$, and the same is true of the space $\ol \XC_Q$ where
  $Q \subset P(e)$ is a face of the polytope $P(e)$.

\begin{example}
  An edge whose polytope $P(e) \in \PP$ is zero dimensional is a split
  edge. The example of more interest to us will be a top-dimensional
  component $\ol X_{P_0} \subset \XX$ which is a toric variety, and
  all whose torus-invariant divisors are relative divisors. That is,
  for any torus-invariant divisor $D \subset \ol X_{P_0}$, there is a
  polytope $Q \subset P$ in $\PP$ such that $D=\ol X_Q$.
\end{example}

\begin{definition}
  A \em{deformation parameter} for a tropical graph $\bGam$ is an
  element
  \begin{equation}
    \label{eq:taudef}
    \eta = (\eta_e)_{e \in \Edge_s(\bGam)} \in 
    \bigoplus_{e \in \Edge_s(\bGam)}\t/\t_{\cT(e)} =:\t_{\Gam},      
  \end{equation}
  where $\t_{\cT(e)} \subset \t$ is the linear span $\bran{\cT(e)}$ of
  the direction $\cT(e) \in \t_{P(e),\Z}$ of the edge $e$.
\end{definition}

The following is a preliminary definition of a deformed map with a
fixed deformation parameter.

\begin{definition}\label{def:deformpre}
  {\rm(Deformed map)} Let $\eta$ be a deformation parameter for a
  tropical graph $\Gamma$. An $\eta$-deformed map of type $\Gamma$ is
  the same as a broken map of type $\Gamma$ with the difference that
  at any split edge $e=(v_+,v_-) \in \Edge_s$, the matching condition
  \eqref{eq:match-proj} at the corresponding node $w_e$ is replaced by
  the \em{deformed matching condition}
  \begin{equation}\label{eq:defmatch}
    (\pi_{\cT(e)}^\perp \circ u_{v_+})(w_e^+)=e^{i\eta_e}(\pi_{\cT(e)}^\perp \circ u_{v_-})(w_e^-) \in \XC_{P(e)}/T_{\cT(e),\C}.
  \end{equation}
  Here $\cT(e) \in \t_{P(e),\Z}$ is the edge direction, and
  $\eta_e \in \t/\t_{\cT(e)}$ is the deformation parameter for the
  edge $e$.
\end{definition}

Compactifications of moduli spaces of deformed maps of a fixed type
$\Gamma$ may contain maps with a larger tropical graph $\Gamma'$ with a collapse morphism to $\Gamma$,
in which case, we need to ``remember'' the set of split edges while
passing from $\Gamma$ to $\Gamma'$. To address this requirement, the
domain of a deformed map is a based curve, which carries information
of a base tropical graph with the data of split edges. The reader
might wish to keep in mind that for generic perturbation, the tropical
graph of a deformed map is the same as the base tropical
graph. Deformed maps whose tropical graph is larger than the base tropical graph
occur in the compactifications of moduli spaces that have dimension at
least two; however to define \ainfty algebras, we only work with
moduli spaces of dimension at most one.
 
\begin{definition} \label{def:basedcurve} {\rm(Based curves)}A
  \em{curve with base type} consists of
  \begin{enumerate}
    \index{Disk! Treed disk with base type}
  \item a stable treed disk $C$ of type $\tGam$,
  \item a \em{base tropical graph} $\bGam$, \index{Base tropical
      graph} \index{Tropical graph!Base tropical graph}
  \item and an edge collapse morphism $\kappa:\tGam \to \bGam$ that
    necessarily collapses all disk edges $e \in \Edge_\white(\tGam)$
    and treed segments $T_e$, $e \in \Edge_\white(\tGam)$, and
    possibly some interior edges. (The map $\kappa$ is an edge
    collapse of graphs, and not of tropical graphs, because $\tGam$
    does not have a tropical structure.)
  \end{enumerate}
  The \em{type} of such a curve consists of the datum
  \index{Combinatorial type! of a curve with base} \index{Based!curve
    type} \index{Moduli space!of treed disks with base
    $\M_{\tGam,\Gamma}$} $\tGam \xrightarrow{\kappa} \bGam$.  The
  moduli space of stable treed curves of type $\tGam$ with base type
  $\bGam$ is denoted by $\M_{\tGam,\Gamma}$, or simply $\M_{\tGam}$
  when the context allows it.  The topology on the space of based
  curves is the standard topology on the space of stable curves with
  the additional axiom that the base tropical graph is preserved in
  the limit.  This ends the Definition.
\end{definition}
\noindent 
The compactified moduli space of based curves is a product
\[\ol \M_{\tGam, \Gamma} = \ol \M_{\tGam_1} \times \dots \times \ol
  \M_{\tGam_{s(\tGam)}}\]
where $\tGam_1,\dots, \tGam_{s(\tGam)}$ are the connected components
of $\tGam \bs \Edge_s(\Gamma)$, with each edge $(v_+,v_-) \in \Edge_s(\Gamma)$ replaced by a pair of leaves on $v_+$ and $v_-$. 
Assuming that the disk vertices of
$\tGam$ lie in $\tGam_1$, $\ol \M_{\tGam_1}$ is the compactified
moduli space of treed disks, and the others $\ol \M_{\tGam_i}$,
$i \neq 1$, are moduli spaces of marked nodal spheres.

The deformation datum for a deformed map is domain-dependent and
satisfies coherence conditions under morphisms (Cutting edges),
(Collapsing edges), (Making an edge length/weight finite/non-zero),
(Forgetting edges) for based curve types. These are the same as
morphisms for ordinary curve types (Definition T-\ref{T-def:pertops}),
with the added feature that the base tropical graph is preserved by
all of the operations except (Cutting edges), where it gets changed in
the obvious way (Definition T-\ref{T-def:bgm}).
\begin{definition}\label{def:coh-deform}
  {\rm(Deformation datum)} A coherent deformation datum
  $\ul \eta=(\eta_{\tGam,\bGam})_{(\tGam,\bGam)}$ consists of a
  continuous map
  \[\eta_{\tGam,\bGam} : \ol \M_\tGam \to \t_\Gam \simeq \bigoplus_{e
      \in \Edge_s(\bGam)}\t/\t_{\cT(e)} \]
  for all types $(\tGam,\bGam)$ of treed disks with base, that are
  coherent based curve type morphisms (described above) and the
  following (Marking independence) axiom: For any type $\tGam$, the
  map $\eta_{\tGam,\bGam}|\M_\tGam$ factors as
  \begin{equation}
    \label{eq:taufactor}
    \text {\rm{(Marking independence)}} \quad 
    \eta_{\tGam,\bGam}=\eta_\tGam' \circ f_\tGam, \quad \M_{\tGam} \xrightarrow{f_\tGam} [0,\infty]^{\Edge_{\white,-}(\tGam)} \xrightarrow{\eta_\tGam'} \t_\Gam,  
  \end{equation}
  where $f_\tGam([C])$ is equal to the edge lengths $\ell(e)$ of treed
  components at the boundary nodes of the treed curve $C$.
\end{definition}
\begin{remark}\label{rem:gwspl}
  To define moduli spaces of deformed maps, we need to use a
  domain-dependent deformation datum, and not just a fixed deformation
  parameter; this is needed to achieve coherence with respect to the
  (Cutting an edge) morphism. Note that this requirement is absent for
  the moduli space of spheres. So, if one were defining moduli space
  of deformed spheres (instead of deformed disks), then the
  deformation datum can be taken to be constant on the moduli
  space. In other words, one could just assign a constant deformation
  parameter $\eta_e$ to each split edge $e$.
\end{remark}
A perturbation datum for deformed maps $\ul \Pe =(\Pe_\tGam)_\tGam$
consists of coherent perturbations (as in Definition
T-\ref{T-def:coherent})
\[\Pe_\tGam=(J_\tGam, F_\tGam), \quad J_\tGam : \ol \S_\tGam \to
  \J^\cyl_{\om_\XX,\tau}(\XX), \enspace F_\tGam : \ol \T_\tGam \to C^\infty(L,\R)\]
for all types $\tGam$ of treed disks, with the domain-dependent almost complex structures being
cylindrical and tamed (as in \eqref{eq:omxxtau}). 
The perturbation datum
$\Pe_\tGam$ does not depend on the base tropical graph.
\begin{definition} \label{def:admap} {\rm(Deformed holomorphic maps)}
  Let $\tGam \to \bGam$ be a type of based curve. Let $\Pe_\tGam$ be a
  perturbation datum and $\eta_\tGam$ be a deformation datum. A
  \em{$(\eta_\tGam, \Pe_\tGam)$-deformed holomorphic map} is the same as a $\Pe_\tGam$-holomorphic
  broken map $u:C \to \XX$ (see Definition
  T-\ref{T-def:pdisks}) with the difference that on any split edge
  $e \in \Edge_s(\bGam)$ the edge matching condition is $\eta_\tGam([C],e)$-deformed as in
  \eqref{eq:defmatch}.  The \em{combinatorial type} of a deformed holomorphic map includes the
 type data of a broken holomorphic map and the datum of the base tropical graph. 
 The type of a deformed holomorphic map is \em{rigid} if it is rigid in the sense of
 a broken holomorphic map type, that is, without considering the base tropical graph).
 The rigidity condition implies that the tropical graph $\tGam$ of the map is the same as the base tropical graph; otherwise $\tGam$ is not a rigid tropical graph.
\end{definition}
For a type $\tGam$ of deformed holomorphic maps 
and deformation datum $\ul \eta$, 
let
\[\tM^\deform_{\tGam}(\ul \Pe,\ul \eta) \]
denote the moduli space of $(\eta_\tGam, \Pe_\tGam)$-deformed
holomorphic maps modulo domain reparametrizations.  The tropical
symmetry group $T_\trop(\tGam)$ acts on the moduli space
$\tM^\deform_{\tGam}(\ul \Pe,\ul \eta)$, the tropical symmetry group
does not depend on the base tropical graph.

\begin{remark} 
  Compactified moduli spaces of  deformed holomorphic maps of different
  base tropical types do not intersect. Indeed, a stratum
  $\M_{\deform,\tGam_0}$ of deformed maps is contained in the
  compactification $\ol \M_{\deform,\tGam}$ only if $\tGam_0$ is
  obtained by (Collapsing edges) in $\tGam$ or by (Making an edge
  length/weight finite/non-zero) in $\tGam$, and both these morphisms
  preserve the base tropical graph.
\end{remark}

The following result describes
zero and one-dimensional moduli spaces of deformed holomorphic maps along with their
compactifications. 
The result is an analogue of the theorems for broken (undeformed)
maps; \eqref{part:pdtrans1} is an analogue of Theorem
T-\ref{T-thm:transversality}, and \eqref{part:pdtrans2},
\eqref{part:pdtrans3} are analogues of Proposition
T-\ref{T-prop:truebdry}.  The proofs are analogous.  Similar to the
case of broken maps (Definition T-\ref{T-def:crowded}), a type
$\Gamma$ of rigid deformed maps is \em{crowded} if there a surface
component corresponding to a vertex $v \in \Ver(\Gamma)$ on which the
projection $\pi_{P(v)} \circ u_v: S_v \to X_{P(v)}$ of the map is
constant, and is uncrowded otherwise.
\begin{proposition}\label{prop:pdtrans}
  Given a coherent deformation datum $\ul \eta$, there is a co-meager
  set $\PPe^{\reg}$ of coherent regular perturbations for which the
  following hold.
  \begin{enumerate}
  \item \label{part:pdtrans1} For any uncrowded type $\tGam$ of
    deformed holomorphic maps, regular perturbation
    $\Pe_{\tGam} \in \PPe^{\reg}$, and for any disk output and inputs
    $x_0,\dots,x_{d(\white)} \in \hat \cI(L)$ such that
    $i(\tGam,\ul x) \leq 1$, the moduli space
    $\tM_{\deform,\tGam}(\ul \Pe, \ul \eta, \ul x)$ is a manifold of
    expected dimension.
  \item \label{part:pdtrans2} Any one-dimensional component of the
    moduli space of rigid deformed broken treed disks
    $\tM_{\deform,\tGam}(\ul \Pe,\ul \eta,\ul x)$ admits a
    compactification as a topological manifold with boundary.  The
    true boundary is equal to the union of zero-dimensional strata
    whose domain treed disks $u: C \to \XX_{\cP}$ have a boundary edge
    $e \in \Edge_\white(\Gamma)$ that is broken.
  \item \label{part:pdtrans3} For any $E>0$, there are finitely many
    zero and one-dimensional components of the moduli space of rigid
    broken treed disks with area $\leq E$.
  \end{enumerate}
\end{proposition}

We define a deformed version of the Fukaya algebra by counting
deformed broken maps:

\begin{definition} {\rm(Deformed Fukaya algebra)} 
  Let $\ul \eta$ be a coherent deformation datum, and let
  $\ul \Pe=(\Pe_\tGam)_\tGam$ be a coherent regular perturbation datum
  for all types $\tGam$ of based treed disks whose base tropical graph
  $\Gamma$ is rigid.  The \em{deformed Fukaya algebra} is the graded
  vector space
  \[CF_{\deform}(\XX,L,\ul \eta):=CF^{\on{geom}}(\XX,L) \oplus \Lam
    x^{\greyt}[1] \oplus \Lam x^{\whitet}\]
  %
  equipped with composition maps
  \begin{equation}\label{eq:pdmun}
    m^{d_\white}_\deform(x_1,\dots,x_{d_\white})= \sum_{x_0,u \in {\tM_{\deform,\tGam}(\ul \Pe,\ul \eta, \ul{x})_0}} w_s(u) x_0,
  \end{equation}
  where the type $\tGam$ of $u$ ranges over all rigid types with
  $d_\white$ inputs (see Definition T-\ref{T-def:basemap}
  \eqref{T-part:basemap3} for rigidity), and
  \begin{equation}
    \label{eq:wtwsudef}
    w_s(u):=(-1)^{\heartsuit}(d_\black(\Gam)!)^{-1} (s(\tGam)!)^{-1} \Hol([\partial u]) \eps(u)
    q^{A(u)},
  \end{equation}
  where $s(\tGam)$ is the number of split edges in the type $\tGam$,
and the other symbols are as in \eqref{eq:wtwsubr}. 
\end{definition}

As in the case of broken maps, for a one-dimensional component of the
moduli space, the configurations with a boundary edge of length zero
constitute a fake boundary, whereas those with a broken boundary edge
constitute the true boundary of the moduli space.  The
$A_\infty$-axioms for the Fukaya algebra $CF_{\deform}(\XX,L,\ul \eta)$
follow from counts of the true boundary points of one-dimensional
components of the moduli space of deformed maps, and we obtain the
following result.

\begin{proposition}
  The composition maps in \eqref{eq:pdmun} satisfy the \ainfty axioms,
  and the Fukaya algebra $CF_{\deform}(\XX,L,\ul \eta)$ is strictly
  unital.
\end{proposition}

The Fukaya algebras are independent of deformation parameters
$\ul \eta$ up to homotopy equivalence. The proof is by constructing an
\ainfty morphism by counts of a version of quilted disks which we now
define, which is a variation of Definition T-\ref{T-def:pert-morph} for the unbroken case
and Definition T-10.34 for the broken case. We recall that
given regular perturbation data $\ul \Pe^0$, $\ul \Pe^1$, 
a perturbation morphism $\ul \Pe^{01}$ extending them consists of domain-dependent almost complex structures and Morse functions defined on the moduli space of treed quilted disks, which restrict to
$\ul \Pe^0$ and $\ul \Pe^1$ on light and dark unquilted components.
As in the unquilted case, $\ul \Pe^{01}$ does not depend on the tropical structures. 
Quilted treed disks are also equipped with a ``distance to the quilting circle'' function
\begin{equation}
  \label{eq:dist-seam}
  d_\Gamma : \U_\Gamma \to [-\infty,\infty]
\end{equation}
 that is
equal to $0$ on quilted components,
and on any surface component $S$ has a constant value equal to the sum of (signed) lengths of treed segments lying between $S$ and the nearest quilted component. This function is composed with a fixed diffeomorphism $\delta : [-\infty,\infty] \to [0,1]$,
and on any surface component the deformation parameter is fixed to be $\eta^{\delta \circ d_\Gamma}$ where $\{\eta^t\}_{t \in [0,1]}$ is a path of deformation data for curves of type $\Gamma$. The domain of a deformed quilted map is a 
\em{quilted based curve}, which is analogous to a based curve (Definition \ref{def:basedcurve}) with ``treed disk'' replaced by ``quilted treed disk''. 

\begin{definition} {\rm(Deformed quilted holomorphic maps)}
  Let $\ul \eta^0$, $\ul \eta^1$ be deformation data, and let
  $\{\ul \eta^t\}_{t \in [0,1]}$ be a path of deformation data
  connecting $\ul \eta^0$ and $\ul \eta^1$.  Let $\ul \Pe^k$ be a
  regular perturbation datum for $\ul \eta^k$, $k=0,1$, and
  $\ul \Pe^{01}$ be a perturbation morphism extending $\ul \Pe^0$ and
  $\ul \Pe^1$ as described in the preceding paragraph.
  A \em{deformed holomorphic quilted map} $u: C \to \XX$ is the same
  as a deformed holomorphic map (Definition \ref{def:admap}) with the
  difference that the domain $C$ is a quilted based treed
  disk. Assuming that the type of $C$ is $(\tGam, \Gamma)$, $u$ is
  $\Pe^{01}_\tGam$-holomorphic and the deformation parameter at the
  split edge $e \in \Edge_s(\Gam)$ is
  $\eta^{\delta(d_\tGam(w_e))}_{\tGam}(C,e)$.  Here, we note that
  $w_e$ is a point in the universal curve $\ol \U_\tGam$,
  $\delta(d_\tGam(w_e)) \in [0,1]$ is a function of the distance
  \eqref{eq:dist-seam} from the quilting circle.
\end{definition}

Homotopies between \ainfty morphisms are defined by counts of twice-quilted disks. Deformed twice-quilted disks are defined in the same way as broken or unbroken twice-quilted disks.
The deformation parameter is taken to be a function of a version
of the  ``distance from the quilting circle'' parameter from \cite[Definition
5.9]{cw:flips}. This parameter is a function of lengths of treed segments, and remains unchanged if nodes and markings move within a surface component. 

The following Proposition is analogous to Proposition
T-\ref{T-prop:samedegree}, and the proof of that Theorem carries over.
\begin{proposition}
  \label{prop:pd3} Given a path of deformation data
  $\{\eta^t\}_{t \in [0,1]}$, and regular perturbation data
  $\ul \Pe^0$, $\ul \Pe^1$ for the end-points, there is a co-meager set
  of regular perturbation morphisms extending $\ul \Pe^0$,
  $\ul \Pe^1$. Any such perturbation morphism induces a convergent
  unital \ainfty morphism
  \[ \phi=(\phi^r)_{r \ge 0} : CF_{\deform}(\XX,L,\ul \Pe^0,\ul \eta^0)
    \to CF_{\deform}(\XX,L,\ul \Pe^1,\ul \eta^1)\]
  defined by counts of quilted deformed disks. The \ainfty morphism
  $\phi$ is a \ainfty homotopy equivalence and has zero-th composition
  map $\phi^0(1)$ with positive $q$-valuation, that is
  $\phi^0(1) \in \Lam_{>0}\bran{\hat \cI(L)}$.
\end{proposition}

\section{Split tropical graphs }
\label{sec:splitgr}
We introduce the idea of split maps in Sections \ref{subsec:idea} and
\ref{subsec:fs}, and then go on to define split tropical graphs, which are the
combinatorial objects underlying split maps.
\subsection{The idea of a split map}
\label{subsec:idea}
Split maps arise as limits of deformed maps as the deformation
parameters go to infinity. A split map is a version of a broken map
where there is no matching condition along split edges.  The tropical
properties of deformed maps, such as the tropical graph and the
symmetry group, look exactly like that of broken maps; this is
explained by the fact that deformation of an edge matching condition
is a small scale phenomenon that is not seen by the tropical graph
which only detects large scale behavior. As the deformation parameters
go to infinity, split maps, the limit objects, have a new kind of
tropical graph called \em{split tropical graph}, which is introduced
in this section. Split tropical graphs do not satisfy the 
\hyperref[eq:bslope]{(Direction)} condition 
at split edges; instead they satisfy a \em{cone condition}
which ensures that the tropical symmetry group is as large as
possible. In particular, in the rigid case, the dimension of the
tropical symmetry group is equal to the codimension of the matching
conditions at split edges.

We work through an example to motivate the formal definitions.  Let
$X=F_1$ be the Hirzebruch surface and let $L \subset X$ be a toric
Lagrangian. For a fixed point $p$ as in Figure \ref{fig:h2split},
there is exactly a single disk in the class
$\delta_{D_2} + \delta_E \in \pi_2(X,L)$ that passes through
$p$. Under a multiple cut, there is a single broken disk $u_\br$ in
this class.  We denote by $\Gamma$ the tropical graph of $u_\br$.
We fix a direction 
\begin{equation}
  \label{eq:etasp1}
  \etasp=\pi^\perp_{(1,1)}(1,-1) \in \t/\bran{(1,1)}
\end{equation}
called the \em{cone direction}, and deform the matching condition at the edge $e$ in the direction
$\etasp$. That is, we consider 
a family $\{u_\tau\}_\tau$  of deformed maps with deformation parameters $\tau \etasp$. 
In the deformed family,  $u_0$ is the broken map $u_\br$, and 
$u_\tau=(u_{\tau,v_0}, u_{v_2})$,  with the map $u_{v_2}$ staying the same for all $\tau$.  In
the limit $\tau \to \infty$, the intersection point of $u_{v_0}$ with
the divisor $X_{P_{03}}$ and the corner $X_{P_\cap}$ merge, and there
is a bubble in the neck piece $\XX_{P_\cap}$.  In the proof of
Proposition \ref{prop:defconv}, we show that the tropical graph of the
limit $u_\infty$ of the deformed maps $u_\tau$, called a \em{split
  tropical graph} and denoted by $\tGam$, satisfies a \em{cone
  condition}, namely that
\begin{equation}
  \label{eq:cone-eg}
  \Diff_e(\fw(\tGam,\Gamma)) \ni \etasp.   
\end{equation}
Here, $\fw(\tGam, \Gamma)$ is the cone generated by the vertex
positions $\cT=(\cT(v))_v$ of the tropical graph $\tGam$ relative to
positions in $\Gamma$, that is,
\[\fw(\tGam,\Gamma):=\R_{\geq 0}\{\cT=(\cT_{\tGam}(v) -
  \cT_\Gamma(\kappa(v)))_{v \in \Ver(\tGam)}: \cT_\tGam \in \W(\tGam),
  \cT_\Gamma \in \W(\Gamma)\};\]
and for any element $\cT \in \fw(\tGam, \Gamma)$,
\[\Diff_e(\cT):=\pi_{\cT(e)}^\perp(\cT(v_0) - \cT(v_2)).\]
is the amount by which $\cT$ fails to satisfy the direction condition
at the edge $e=(v_0,v_2)$.  Note that if $e$ is not a split edge,
$\Diff_e(\cT)=0$.  The set $\Diff_e(\fw(\tGam,\Gamma))$ is a cone in a
vector space, and it is called the \em{discrepancy cone}.  In the
current example $\dim \t=2$, and we have
\[\fw(\tGam, \Gamma) \simeq \R_{\geq 0} \{\cT_{\tGam}(v_0'') - \cT_\Gamma(v_0)\} = \R_{\geq 0}(2,-1)\]
is one-dimensional, the discrepancy cone is $[0,\infty)$, and therefore,
$\dim(T_\trop(\tGam))=2$. Consequently, in this example,
\begin{equation}
  \label{eq:dimeq}
  \dim(T_\trop(\tGam))=\codim(\text{matching condition at $e$}). 
\end{equation}

\begin{figure}[h]
  \centering\scalebox{.65}{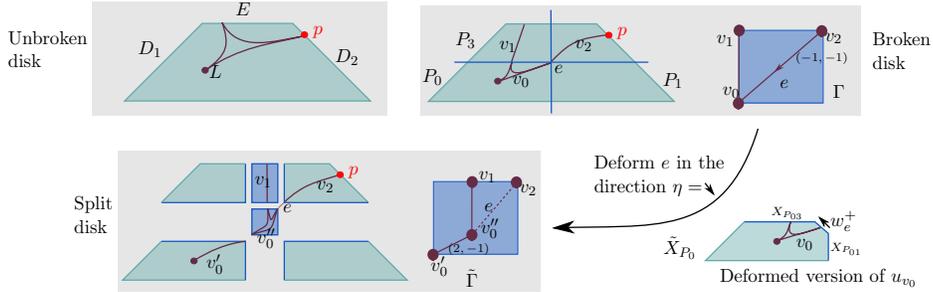}
  \caption{Split disk corresponding to a disk in the Hirzebruch
    surface $X=F_1$ of homology class $\delta_{D_2} +\delta_E$ passing
    through a fixed point $p$.}
  \label{fig:h2split}
\end{figure}

Heuristically, our result is that the moduli space of broken maps is
cobordant to the tropical symmetry orbit space of split maps; that is, for a rigid type
$\Gamma$ of broken maps, there is a cobordism
\begin{equation}
  \label{eq:cobord1}
  \M_\Gamma^\br(\XX, L) \sim \bigsqcup_\tGam(\M_\tGam^\spl(\XX,L,\etasp)/T_\trop(\tGam)) \cdot \mult(\tGam),   
\end{equation}
where $\tGam$ ranges over all split tropical graphs with base
$\Gamma$, and whose discrepancy cone contains the cone direction
$\etasp$, and the multiplicity $\mult(\tGam) \in \Z$ arises out of the
tropical symmetry group.

We emphasize that the moduli space of split maps depends on the cone
direction.  In the above example, there are two possible cone
directions, since $\t/\bran{(1,1)} \simeq \R$.  If instead of $\etasp$
in \eqref{eq:etasp1}, we take the cone direction to be
$\etasp'=\pi^\perp_{(1,1)}(-1,1)$, there are no split maps in the
homology class of $u_\br$ that satisfy the given point constraint.

More generally, when $\dim(\t) \geq 3$, we require the cone direction
$\etasp$ to be generic.  For a generic element $\etasp$, the cone
condition \eqref{eq:cone-eg} implies the dimension condition
\eqref{eq:dimeq}. Indeed, the discrepancy cone is generated by
rational elements, and such a cone contains a generic element only if
it is top-dimensional.

When there are multiple split edges, a split map is obtained as the
limit of a sequence of deformed maps whose deformation parameters
satisfy an increasing condition: Assuming the split edges are ordered
as $e_1,\dots,e_n$, the sequence of deformation parameters is taken to
be
\[\ul \eta_\nu=(c_\nu(i) \pi^\perp_{\cT(e_i)}(\etasp))_i \]
where $\etasp \in \t$ is a generic element and the sequence of tuples
$\ul c_\nu=(c_\nu(1),\dots,c_\nu(n)) \in (\R_+)^n$ is an increasing
sequence in the sense that
\[\lim_\nu\tfrac {c_\nu(j)}{c_\nu(i)}=0 \quad \text{if } i<j.\]
In the proof of Proposition \ref{prop:defconv}, we show that for a
convergent sequence of $\ul \eta_\nu$-deformed maps, the discrepancy
cone $\Disc(\tGam)$ of the tropical graph $\tGam$ of the limit split
map contains $\ul \eta_\nu$ for all $\nu$. Therefore, by Lemma
\ref{lem:inc-gen-cone}, the discrepancy cone is top-dimensional, and
consequently, the dimension of the tropical symmetry group satisfies
\eqref{eq:dimeq} and is the maximum possible.

\subsection{Splitting the diagonal}
\label{subsec:fs}
We point out relations to results in the literature that involve
splitting the diagonal, for which we first
the main result Theorem \ref{thm:tfuk} informally.
 For simplicity, we consider broken maps of type $\Gamma$
with a single split edge $e$. Any such broken map $u$ may be viewed as
a pair of maps $(u_+,u_-)$ with a matching condition at the node
$w_e$, namely $u_+(w_e^+)=u_-(w_e^-)$, or in other words,
\[(u_+(w_e^+),u_-(w_e^-)) \in \Delta \subset X_e^2 \quad \text{where }
  X_e=X_{P(e)}/T_{\cT(e),\C} \simeq T_\C/T_{\cT(e),\C}.\]
Deforming the matching condition at the
node $w_e$ produces a split map, which is a pair
$u_\infty=(u_\infty^+,u_\infty^-)$ with no matching condition at the
node $w_e$, and the set of vertex positions on either side is a cone
$\mC_\pm$ such that $\mC_+ \times \mC_- \subset \t/t_{\cT(e)}$ is
top-dimensional. The cone $\mC_+ \times \mC_-$ generates the tropical symmetry group
$T_\trop(\tGam)$ for $u_\infty$.  There is a homotopy equivalence of moduli spaces
\begin{equation}
  \label{eq:br2spl}
  \M^\br_\Gamma(L) \sim \bigsqcup_\tGam \mult(\tGam) \M^\spl_\tGam(L,\etasp)/T_\trop(\tGam), 
\end{equation}
where $\tGam$ ranges over split tropical graphs whose base graph is
$\Gamma$, and the multiplicity $\mult(\tGam) \in \Z$ arises from the
graph $\tGam$. The moduli space of split maps is a product
\begin{equation}
  \label{eq:splproduct}
  \M^\spl_\tGam(L,\etasp)/T_\trop(\tGam)= (\M_{\tGam^+}/T_\trop(\tGam^+)) \times (\M_{\tGam^-}/T_\trop(\tGam^-))
\end{equation}
where both $\tGam^+$, $\tGam^-$ contain an end of the split edge $e$.

The splitting technique described in the previous paragraph is
reminiscent of the Fulton-Sturmfels formula \cite{fulton} for the splitting of a
diagonal in a toric variety. First, observe that in a toric variety
$X$ with moment map $\Phi : X \to \t^\dual$, if there is an element
$\xi \in \t$, for which the component $\Phi_\xi:=\bran{\Phi,\xi}$ of
the moment map is Morse, then, the Morse deformation gives a splitting
of the diagonal
\[H_*(X \times X) \ni [\Delta] \sim \cup_{p \in
    \crit(\Phi_\xi)}[W^u(p)] \times [W^s(p)] . \]
Here $p$ is a torus fixed point and the stable and unstable manifolds
$W^s(p)$ and $W^u(p)$ are torus-invariant submanifolds that intersect
transversely at $p$.  If $\Phi_\xi$ is not Morse, a splitting is given
by the Fulton-Sturmfels formula
\begin{equation}
  \label{eq:fs-split}
  [\Delta] \sim \sum n_{Q_+,Q_-} [X_{Q_+}] \times [X_{Q_-}]  
\end{equation}
where each pair $(X_{Q_+},X_{Q_-})$ consists of transversely
intersecting torus-invariant submanifolds, and the multiplicity
$n_{Q_+,Q_-} \in \Z$ is determined by the fan of $X$. The product of
the cones $C_{Q_+}$, $C_{Q_-}$ corresponding to $X_{Q_+}$, $X_{Q_-}$
span $\t$.  Whereas in the Fulton-Sturmfels splitting in
\eqref{eq:fs-split} there is a finite set of possible cones
$C_{Q_\pm}$, in our splitting in \eqref{eq:br2spl},
\eqref{eq:splproduct}, $\mC_+$, $\mC_-$ may be any rational cone in
$\t$.

Our techniques also resemble the Morse deformation in Charest-Woodward
\cite{cw:flips} and Bourgeois \cite{bo:com}, where the matching
condition for broken maps in a single cut is deformed to a split
form. We describe the construction in \cite{cw:flips}: Consider a
single cut with cut spaces $X_{P_+}$, $X_{P_-}$ and relative divisor
$X_Q \subset X_{P_\pm}$.  The matching condition at a node $w$ is
given by
\[u_+(w_+)=u_-(w_-) \in X_Q. \]
Let
\[ H : X_Q \to \R, \quad \phi_t^H:X_Q \to X_Q \]
be a Morse function, whose time $t$ gradient flow is given by
$\phi_t^H$.  Given a deformation parameter $t \in (0,\infty)$, the
node matching condition for a $t$-deformed map is
\[u_+(w_+)=\phi_t^H u_-(w_-).\]
As $t \to \infty$ the Morse trajectory at any node degenerates to a
broken Morse trajectory. The matching condition degenerates to the
following condition: The lifts of the node lie on two transversely
intersecting Morse cycles, namely the stable and unstable submanifold
of some critical point of the Morse function.  The Morse deformation
is thus based on a deformation of the diagonal
$\Delta_{X_Q} \subset X_Q \times X_Q$ to a split form, which is a
union of products
\begin{equation}
  \label{eq:morse-degen}
  \bigcup_{p \in \crit(H)} W^u(p) \times W^s(p) \subset X_Q \times X_Q.
\end{equation}
\subsection{Preliminaries on cones and genericity}
We define a notion of genericity and a notion of an increasing
sequence in a vector space.  These two properties are then put
together to describe a sufficient condition for cones in a vector
space to be top-dimensional, which is used in Section \ref{sec:stg} to
describe the cone condition for split tropical graphs.

\begin{definition}
  Let $V \simeq \R^{n}$ be a vector space equipped with a dense
  rational lattice $V_\Q \simeq \Q^n$.
  \begin{enumerate}
  \item {\rm(Rational subspace)} A \em{rational subspace} of $V$ is a
    linear subspace $W \subset V$ in which $W \cap V_\Q$ is dense.
  \item {\rm(Generic vector)} A vector $\eta \in V$ is \em{generic} if
    it is not contained in any proper rational subspace of $V$.
  \item {\rm(Rational linear map)} Suppose $W$ is a real vector space
    with a dense rational lattice $W_\Q$. A linear map $f:V \to W$ is
    \em{rational} if $f(V_\Q) \subseteq W_\Q$. The subspaces $\ker(f)$
    and $\on{im}(f)$ are rational subspaces.
  \end{enumerate}
\end{definition}

\begin{remark}\label{rem:quo2sub}
  The notion of genericity is applicable to the Lie algebra $\t$ of
  the torus $T$ since it has an integral lattice $\t_\Z \subset \t$
  and consequently a rational lattice.  For a direction
  $\cT(e) \in \t_\Z$ of a tropical edge, the projection
  \[\pi_{\cT(e)}^\perp: \t \to \t/\t_{\cT(e)} \]
  projects the rational lattice $\t_\Q \subset \t$ to a dense subset
  $(\t/\t_{\cT(e)})_\Q \subset \t/\t_{\cT(e)} $, which happens to be
  the rational lattice obtained by viewing the quotient
  $\t/\t_{\cT(e)} $ as the Lie algebra of the torus $T/T_{\cT(e)}$.
  Thus the projection $\pi_{\cT(e)}^\perp$ maps generic elements of
  $\t$ to generic elements in the quotient.
\end{remark}

\begin{definition}\label{def:ratcone}
  {\rm(Rational cone)} Let $V$ be a vector space with a dense rational
  subspace $V_\Q \subset V$.  A  subset $C \subset V$ is a
  \em{cone} if it is convex and for any $v \in C$, $\alpha v \in C$ for all
  $\alpha \geq 0$.  A cone $C \subset V$ is \em{rational} if $C$ is
  contained in a rational subspace $W \subset V$, and $C$ contains an
  open neighborhood of $W$, or in other words, $\dim(C)=\dim(W)$. 
\end{definition}

\begin{lemma}
  Suppose $C \subset V$ is a rational cone containing a
  generic vector. Then $C$ is top-dimensional.
\end{lemma}
\begin{proof}
  The unsigned cone
  $C_\pm:=\R \bran C \subset V$
  defined as the $\R$-span of $C$ is a rational subspace of $V$ that
  contains a generic vector. Therefore $C_\pm=V$ and $C$ is
  top-dimensional.
\end{proof}
An \em{increasing cone} is a special kind of a top-dimensional cone in
$(\R_{\geq 0})^n$:
\begin{definition}{\rm(Increasing cone)}
  %
  A rational cone $C \subset (\R_{ \geq 0})^n$ is an
  \em{increasing cone} if for all $i=1,\dots,n$,
  $C \cap \{x_{i+1}=\dots=x_n=0\}$ is an $i$-dimensional cone.
\end{definition}
The increasing condition on cones can be equivalently stated in terms
of increasing sequences of tuples.
\begin{definition}{\rm(Increasing sequences of
    tuples)}\label{def:inc-tup}
  A sequence of tuples
  \[ \ul x_\nu=(x_{1,\nu},\dots,x_{n,\nu}) \in \R_+^n\]
  is \em{increasing} if $x_{i,\nu} \to \infty$ for all $i$ as
  $\nu \to \infty$, and
  \[\lim_\nu\tfrac {x_{i,\nu}} {x_{i-1,\nu}}=0 \]
  for $i=2,\dots,n$.
\end{definition}
\begin{lemma}\label{lem:inc-equiv}
  The following are equivalent for a cone
  $C \subset (\R_{\geq 0})^n$.
  \begin{enumerate}
  \item $C$ is an increasing cone.
  \item {\rm(Weak sequence)} \label{part:weakseq} There exists an
    increasing sequence of tuples in $C$.
  \item {\rm(Strong sequence)} \label{part:strongseq} The tail of any
    increasing sequence of tuples lies in $C$.
  \end{enumerate}
\end{lemma}
\begin{proof}
  The increasing cone condition can be characterized inductively :
  $C \subset (\R^n_+)$ is an increasing cone if and only if
  $e_1:=(1,0,\dots,0) \in C$ and the \em{normal cone}
  \[N_{e_1}C:=\{(v_2,\dots,v_n) \in \R_+^{n-1} : \exists t_0 >0 :
    (1,tv_2,\dots,tv_n) \in C \quad
    \forall 0 \leq t \leq t_0\}\]
  is an increasing cone in $(\R_+)^{n-1}$.  Conditions
  \eqref{part:weakseq} and \eqref{part:strongseq} can also be
  characterized in the same way. That is, a cone $C \subset (\R^n_+)$
  satisfies the (Weak sequence) resp. (Strong sequence) condition if
  and only if $(1,0,\dots,0) \in \ol C$ and the normal cone $N_{e_1}C$
  satisfies the (Weak sequence) resp. (Strong sequence) condition.
  The proof of the lemma then follows by induction.
\end{proof}

The increasing property and genericity are combined to give the
following sufficient condition for a cone to be top-dimensional:
\begin{lemma}\label{lem:inc-gen-cone}
  Let $V=\sum_{i=1}^kV_i$ be a direct sum of vector spaces, each of
  which has a dense rational lattice $V_{i,\Q} \subset V_i$. For each
  $i$, let $\eta_i \in V_i$ be a generic element.  Suppose
  $C \subset V$ is a rational cone that contains a sequence
  \[\eta_\nu:=\sum_{i=1}^k c_{\nu,i}\eta_i\]
  where $c_\nu=(c_{\nu,i})_i \in \R_+^n$ is an increasing
  sequence. Then $C$ is top-dimensional in $V$.
\end{lemma}
\begin{proof}
  We show $C$ is top-dimensional by proving that the $\R$-span of $C$,
  denoted by $C_\pm \subset V$ is, in fact, equal to $V$. First, we
  observe that $\eta_1 \in C_\pm$, since, by the increasing condition,
  $\eta_1$ is the limit of the sequence
  $\frac {\eta_\nu} {c_{\nu,1}} \in C$. Then, since $V_1 \cap C_\pm$
  is a rational subspace of $V_1$ and it contains the generic element
  $\eta_1$, we conclude $V_1=V_1 \cap C_\pm$. Therefore, $V_1$ is
  contained in $C_\pm$. The proof can now be finished by induction:
  The projection of $C$ to $\oplus_{i=2}^nV_i$ is a rational
   cone and it contains the sequence
  $\sum_{i=2}^k c_{\nu,i}\eta_i$.
\end{proof}

\subsection{Split tropical graphs}\label{sec:stg}
We start by describing a preliminary version called a ``quasi-split tropical
graph'' and then describe the genericity condition under which such a
graph is a split tropical graph. The main result of this section is that under this
genericity condition called the \em{cone condition}, there is a lower
bound on the dimension of the space of relative vertex
positions. Later, this would imply 
that the dimension of the tropical
symmetry group of split maps is at least the codimension of the
matching condition at split edges (see Remark \ref{rem:dimtrop}).

\begin{definition} \label{def:qsplit} {\rm(Quasi-split tropical
    graph)}
  A \em{quasi-split tropical graph}, denoted by $\tGam \xrightarrow{\kappa} \Gamma$ or $(\tGam, \Gamma)$, consists of
  \begin{enumerate}
  \item a tropical graph $\bGam$, called a \em{base tropical graph}
    with a collection of split edges
    $\Edge_s(\Gamma) \subseteq \Edge(\Gamma)$ as in Definition
    \ref{def:splitedge},
  \item a graph $\tGam$ with an edge collapse morphism
    $\kappa:\tGam \to \bGam$,
    %
  \item \label{part:qsplitorder} an ordering $\prec_\tGam$ on the
    split edges $\Edge_s(\bGam)$,
  \item and a tropical structure on each connected component of the
    graph $\tGam \bs \Edge_s(\bGam)$ so that the restricted map
    \[\kappa: \tGam \bs \Edge_s(\bGam) \to \bGam \bs
      \Edge_s(\bGam)\]
    is a tropical edge collapse.  Consequently for any vertex
    $v \in \Ver(\tGam)$, $P_{\tGam}(v) \subset P_{\bGam}(v)$.
  \end{enumerate}
\end{definition}
\begin{remark}
  A quasi-split tropical graph can alternately be defined as the
  combinatorial type $\kappa : \tGam \to \Gamma$ of a curve with base
  (see Definition T-\ref{T-def:basedcurve}) together with a tropical
  structure on the connected components of $\tGam \bs \Edge_s(\Gamma)$
  so that $\kappa$ is a tropical edge collapse on
  $\tGam \bs \Edge_s(\Gamma)$.
\end{remark}
\begin{figure}[h]
  \centering\scalebox{.6}{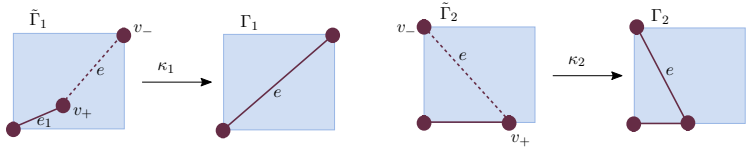}
  \caption{Quasi-split tropical graphs $(\tGam_1,\Gamma_1)$,
    $(\tGam_2,\Gamma_2)$. In both cases $e$ is the only split edge,
    and the direction condition is not satisfied for $e$ in
    $\tGam_i$.}
  \label{fig:w-egs}
\end{figure}

\begin{definition}\label{def:rvp}
  {\rm(Cone of relative vertex positions)} For a quasi-split tropical
  graph $\tGam \xrightarrow{\kappa} \bGam$, the \em{cone of relative
    vertex positions} is
  \begin{multline} \label{eq:fwdef}
    \fw(\tGam,\Gamma):=\R_{\geq 0}\{(\tilde \cT(v) - \cT(\kappa v) \in \Cone(\kappa,v) \subset \t_{P_{\tGam}(v)})_{v \in \Ver(\Gamma')} :\\
    \tilde \cT \in \W(\tGam), \cT \in \W(\Gamma)\},
  \end{multline}
  where
  $\Cone(\kappa,v):=\Cone_{P_\Gamma(\kappa(v))^\dual}P_\tGam(v)^\dual
  \subset \t_{P_{\tGam}(v)}$.  Note that
  \begin{equation*}
    \fw(\tGam,\Gamma)= \Cone_{\W(\Gamma)}\W(\tGam).    
  \end{equation*}
\end{definition}

The notation $\fw(\tGam,\Gamma)$ for the cone of relative vertex
positions is the same as the space of relative translations
(Definition \ref{def:relweight}). These two are indeed the same space,
since both objects satisfy the \hyperref[eq:direction-rel]{(Direction)}
condition on non-split
edges. The (Direction) condition for relative vertex positions is stated
in the following Lemma, whose proof is left to the reader.  

\begin{lemma}{\rm(Direction condition on relative vertex positions)}
  Let $\tGam \xrightarrow{\kappa} \Gamma$ be a quasi-split tropical
  graph.  A tuple $(\cT(v) \in \Cone(\kappa,v))_{v \in \Ver(\tGam)}$
  is a positive multiple of a relative vertex position, that is $\cT \in \fw(\tGam,\Gamma)$, if
  and only if
  \begin{equation}\label{eq:slope-split} \text{\rm(Direction)} \quad
    \cT(v_+) - \cT(v_-) \in
    \begin{cases}
      \R_{\geq 0} \cT(e), & e \in \Edge(\tGam) \bs \Edge(\bGam),\\
      \R\cT(e), & e \in \Edge(\bGam) \bs \Edge_s(\bGam).
    \end{cases}
  \end{equation}
  There is no (Direction) condition on the split edges
  $e \in \Edge_s(\bGam)$.
\end{lemma}
\begin{remark}\label{rem:trans-cone}{\rm(Cone of relative vertex
    positions)}
  For a quasi-split tropical graph
  $\tGam \xrightarrow{\kappa} \Gamma$,
  \[\fw(\tGam,\Gamma) \subset \oplus_{v \in
      \Ver(\tGam)}\Cone(\kappa,v) \subset \oplus_{v \in
      \Ver(\tGam)}\t_{P_\tGam(v)}\]
  is a rational cone (Definition \ref{def:ratcone}), which is seen as follows.  For any
  vertex $v$, $\Cone(\kappa,v)$ is a
  top-dimensional cone in $\t_{P_\tGam(v)}$. Furthermore, $\fw(\tGam,\Gamma)$ is cut
  out of $\oplus_{v \in \Ver(\tGam)}\Cone(\kappa,v)$ by the \hyperref[eq:slope-split]{(Direction)}
  condition; and for each of the
  edges, the direction condition can be written as a set of linear
  equalities with rational coefficients.
\end{remark}
\begin{example}
  The cones of relative vertex positions for the quasi-split tropical
  graphs in Figure \ref{fig:w-egs} are
  \[\fw(\tGam_1,\bGam_1)=\R_{\geq 0}, \quad \fw(\tGam_2,\bGam_2)=\R.\]
  Indeed, in $\tGam_2$, the vertex $v_+$ is free to move to both sides
  of its position in $\bGam_2$. On the other hand in $\tGam_1$, the
  edge $e_1$ is collapsed by $\kappa_1$, the vertex $\kappa_1(v_+)$ in
  $\bGam_1$ lies at the vertex of the dual complex, and in $\tGam_1$,
  $v_+$ is free to move along a line of direction $\cT(e_1)$.  It
  follows that the set of relative positions of $v_+$ in $\tGam_1$ is
  $\R_{\geq 0}$.
\end{example}
\begin{definition}{\rm(Discrepancy cone)}
  Let $\tGam \to \Gamma$ be a quasi-split tropical graph.  Define
  \begin{multline}
    \label{eq:diff}
    \Diff:=(\Diff_e)_{e \in \Edge_s(\bGam)}
    :  \fw(\tGam,\bGam) \to \t_\Gam \simeq \bigoplus_{e \in \Edge_s(\bGam)} \t/\lan \cT(e) \ran, \\
    \cT \mapsto (\pi_{\cT(e)}^\perp(\cT(v_+) -\cT(v_-)))_{e=(v_+,v_-)
      \in \Edge_s(\bGam)}
  \end{multline}
  as the amount by which a relative translation $\cT$ fails to satisfy
  the (Direction) condition at split edges $e \in \Edge_s(\bGam)$.
  The \em{discrepancy cone} for a quasi-split tropical graph is the
  image
  \begin{equation}
    \label{eq:discrepcone}
    \Disc(\tGam,\bGam) :=\Diff(w(\tGam,\bGam)) \subset \bigoplus_{e
      \in \Edge_s(\bGam)} \t/\lan \cT(e) \ran.  
  \end{equation}
  %
\end{definition}

\begin{remark}\label{rem:disc-poly}
  The set $\Disc(\tGam,\bGam)$ is in fact a rational cone, because
  $w(\tGam,\bGam)$ is a rational cone (Remark \ref{rem:trans-cone})
  and $\Diff$ is a rational linear map.
\end{remark}

\begin{definition} {\rm (Split tropical graph)}
  \label{def:splitgr}
  Let ${\etasp} \in \t$ be a generic element, which we call the
  \em{cone direction}.  A \em{split tropical graph} $\tGam$ with cone
  direction ${\etasp}$ is a quasi-split tropical graph $\tGam$ (as in
  Definition \ref{def:qsplit}) whose discrepancy cone
  $\Disc(\tGam,\Gamma)$ (see \eqref{eq:discrepcone}) satisfies the
  following \em{cone condition}: There is an increasing cone
  $C \subset (\R_{\geq 0})^n$ such that for all
  $(c_1,\dots, c_n) \in C$,
  \begin{equation}
    \label{eq:cc}
    \text{(Cone condition)} \quad (c_i \pi^\perp_{\cT(e_i)}({\etasp}))_i \in \Disc(\tGam,\Gamma).  
  \end{equation}
\end{definition}

In Lemma \ref{lem:spliteq}, we give other equivalent characterizations of the  cone condition in terms of increasing sequences of tuples.
\begin{definition}\label{def:inc-tup-e}
  {\rm(Increasing sequence of tuples for split edges)} A sequence of
  tuples
  \[(c_\nu(e))_{e \in \Edge_s(\bGam)} \in (\R_+)^{\Edge_s(\bGam)}\]
  is \em{increasing} if $c_\nu(e) \to \infty$ for all split edges
  $e \in \Edge_s(\bGam)$ and for a pair of split edges
  $e_i \prec e_j, e_i,e_j \in \Edge_s(\bGam)$ (by the ordering in
  Definition \ref{def:qsplit}\eqref{part:qsplitorder})
  \begin{equation}
    \label{eq:incsplit}
    \lim_\nu \tfrac {c_\nu(e_j)} {c_\nu(e_i)}=0.  
  \end{equation}
\end{definition}

\begin{lemma}\label{lem:spliteq}
  Let $\tGam \to \Gamma$ be a quasi-split tropical graph with cone
  direction ${\etasp} \in \t$. The following are equivalent for the
  discrepancy cone $\Disc(\tGam,\Gam)$:
  \begin{enumerate}
  \item $\Disc(\tGam,\Gam)$ satisfies the \hyperref[eq:cc]{(Cone condition)}.
  \item {\rm (Strong sequential cone condition)} For any increasing
    sequence of tuples
    $(c_\nu(e))_{e \in \Edge_s(\bGam)} \in (\R_+)^{\Edge_s(\bGam)}$,
    there exists $\nu_0$ such that
    \begin{equation}
      \label{eq:strongseq}
      (c_\nu(e) \pi^\perp_{\cT(e)}({\etasp}))_e  \in \Disc(\tGam,\bGam) \quad \forall \nu \geq \nu_0.
    \end{equation}
  \item {\rm (Weak sequential cone condition)} There exists an
    increasing sequence of tuples
    $(c_\nu(e))_{e \in \Edge_s(\bGam)} \in (\R_+)^{\Edge_s(\bGam)}$
    such that
    \begin{equation}
      \label{eq:weakseq}
      (c_\nu(e) \pi^\perp_{\cT(e)}({\etasp}))_e  \in \Disc(\tGam,\bGam) \quad \forall \nu.
    \end{equation}
  \end{enumerate}
\end{lemma}
\begin{proof}
  The result follows by applying Lemma \ref{lem:inc-equiv} to the cone
  \[\{c \in (\R_{\geq 0})^{\Edge_s(\Gamma)} : (c(e)\pi^\perp_{\cT(e)}(\etasp))_e \in \Disc(\tGam,\Gamma)\}.\]
\end{proof}

The main result of the section is that the cone condition for a split
tropical graph implies that the discrepancy cone is top-dimensional:
\begin{proposition}{\rm(Dimension of discrepancy cone)}
  \label{prop:discrdim}
  Suppose $\tGam \to \Gam$ is a split tropical graph, with cone
  direction ${\etasp}$. Then,
  \begin{equation}
    \label{eq:diffdim}
    \dim(\Disc(\tGam,\Gam))
    = |\Edge_s(\Gam)|(\dim(\t) -1). 
  \end{equation}
  Consequently, $\dim(\fw(\tGam,\Gamma)) \geq |\Edge_s(\Gam)|(\dim(\t) -1).$
\end{proposition}
\begin{proof}
  The discrepancy cone $\Disc(\tGam,\Gamma)$ is a rational cone in
  $\oplus_{e \in \Edge_s(\tGam)}\t/\bran{\cT(e)}$ by Remark
  \ref{rem:disc-poly}.  By Lemma \ref{lem:spliteq},
  the \hyperref[eq:cc]{(Cone condition)} for the split tropical graph $(\tGam, \Gamma)$ implies
  the weak sequential cone condition \eqref{eq:weakseq}, which, by
  Lemma \ref{lem:inc-gen-cone} implies that $\Disc(\tGam,\Gamma)$ is
  top-dimensional. The last statement follows from the fact that the
  discrepancy cone is the image of $\fw(\tGam,\Gamma)$ under the
  linear map $\Diff$.
\end{proof}

\begin{definition}
  A split tropical graph $\tGam \to \Gamma$ is \em{rigid} if
$\Gamma$ is rigid and
  \[\dim(\fw(\tGam,\Gamma)) = |\Edge_s(\Gam)|(\dim(\t) -1).\]
\end{definition}

\subsection{Examples of split graphs}
We give some examples of split tropical graphs.  In most of our
examples the multiple cut $\PP$ consists of two orthogonally
intersecting single cuts, so that the dual complex is a square as in
Figure \ref{fig:2ddef}. The torus $T$ in this case is $(S^1)^2$.
\begin{figure}[h]
  \centering\scalebox{.6}{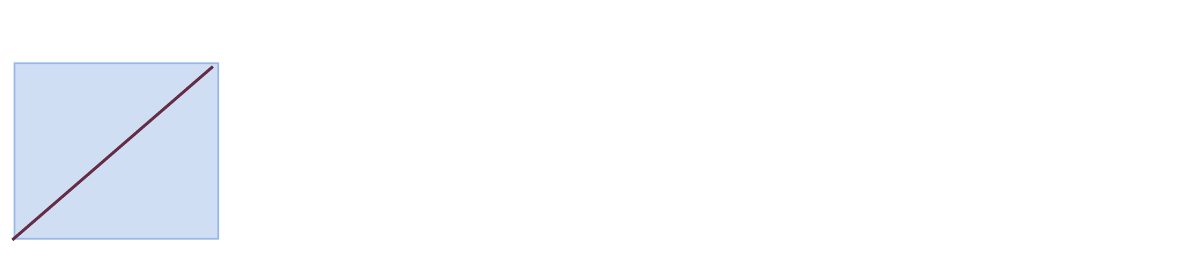}
  \caption{Split graphs $(\tGam_1,\bGam)$, $(\tGam_2,\bGam)$ with
    split edge $e$ and a cone direction ${\etasp}$ satisfying
    $\bran{{\etasp},(1,-1)} > 0$.}
  \label{fig:2ddef}
\end{figure}
\begin{example}
  To split the edge $e$ in $\Gamma$ in Figure \ref{fig:2ddef}, we take
  the cone direction $\etasp$ such that $\bran{{\etasp},(1,-1)} > 0$.
  Both $\tGam_1$, $\tGam_2$ are split tropical graphs.
  %
  \begin{enumerate}
  \item For $\tGam_1$, the cone of relative vertex positions is
    \[\fw(\tGam_1,\Gamma)=\{(\cT(v_+))\}=\{(2,1)t: t \geq 0\},\]
    where $\cT(v_+)$ is to be interpreted as the position of $v_+$ in
    $\tGam_1$ relative to its position in $\Gamma$, and we do not list
    the vertices for which the positions in $\tGam_1$ and $\Gamma$ are
    the same (which is the vertices in $\tGam_1 \bs v_+$). The
    discrepancy cone is
    \[\Disc(\tGam_1,\bGam)=\pi_{\cT(e)}^\perp\{(\cT(v_+) -
      \cT(v_-))\}= \pi^\perp_{(1,1)}\{(2,1)t : t \geq 0\} \]
    which contains the vector $\etasp$.
  \item In a similar way, for $\tGam_2$,
    \[\fw(\tGam_2,\Gamma)=\{(\cT(v_-))\}=\{(0,-1)t: t \geq 0\},\]
    and the discrepancy cone is
    \[\Disc(\tGam_2,\bGam)=\pi_{\cT(e)}^\perp\{(\cT(v_+) -
      \cT(v_-))\}= \pi^\perp_{(1,1)}\{(0,1)t : t \geq 0\} .\]
  \end{enumerate}
  Both split tropical graphs $\tGam_1$, $\tGam_2$ are
  rigid. 
\end{example}
\begin{figure}[h]
  \centering\scalebox{.55}{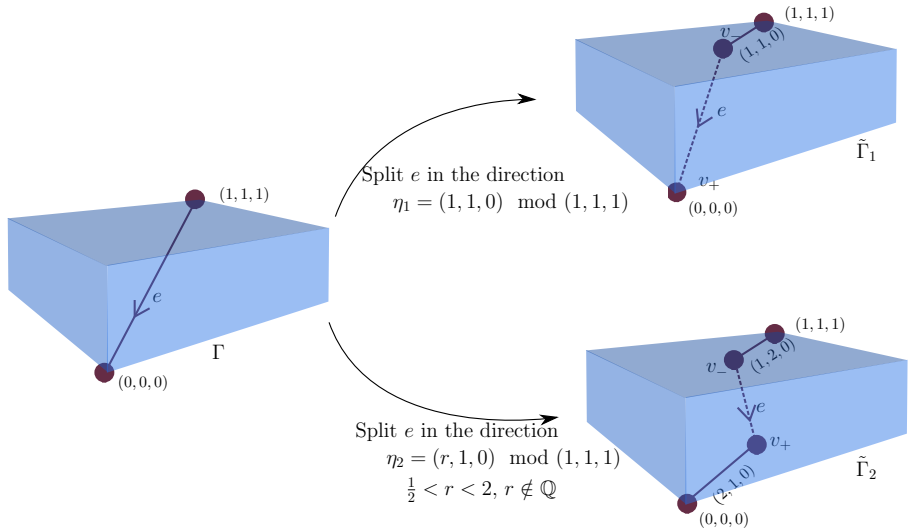}
  \caption{$(\tGam_1,\Gamma)$ is not a split graph. $(\tGam_2,\Gamma)$
    is a split graph since the cone direction $\eta_2$ is generic. }
  \label{fig:3ddef-intro}
\end{figure}
\begin{example}
  In this example, we will see that the genericity of the cone
  direction is necessary to ensure that the discrepancy cone is
  top-dimensional.  Consider a multiple cut $\PP$ consisting of three
  orthogonal single cuts, so the structure torus is $T=(S^1)^3$, and
  the dual complex $B^\dual \subset \t$ is a cube as in Figure
  \ref{fig:3ddef-intro}. 
%
  \begin{enumerate}
  \item The quasi-split graph $(\tGam_1,\Gamma)$ satisfies the \hyperref[eq:cc]{(Cone condition)} for the cone direction $\eta_1:=(1,1,0) \mod (1,1,1)$,
    which is not a generic direction in $\R^3/\lan(1,1,1)\ran$.  The
    relative vertex position cone is
    \[\fw(\tGam_1,\Gamma)=\{(\cT(v_-))\}=\{-(1,1,0)t:t \geq 0\},\]
    and the discrepancy cone is
    \[\Disc(\tGam_1,\Gamma)= \pi^\perp_{\cT(e)}\{(\cT(v_+)
      -\cT(v_-))\} = \pi^\perp_{(1,1,1)}\{(1,1,0)t:t \geq 0\} .\]
    The discrepancy cone contains the cone direction $(1,1,0)$, but is
    not top-dimensional in $\R^3/\lan(1,1,1)\ran$.  Consequently
    $\tGam_1$ is not a split graph.
  \item The quasi-split tropical graph $(\tGam_2,\Gamma)$ satisfies
    the \hyperref[eq:cc]{(Cone condition)} for a generic cone direction
    \[ \eta_2:=(r,1,0), \ r \notin \Q, \ \hh<r<2\]
    and therefore $\tGam_2$ is a split tropical graph.  The relative
    vertex position cone is
    \[\fw(\tGam_2,\Gamma)=\{(\cT(v_+),\cT(v_-)))\}=\{((2,1,0)t_1,-(1,2,0)t_2):
      t_1,t_2 \geq 0\},\]
    and the discrepancy cone is
    \[\Disc(\tGam_2,\Gamma)= \pi^\perp_{(1,1,1)}\{(2,1,0)t_1 +
      (1,2,0)t_2:t_1,t_2 \geq 0\}, \]
    which contains the cone direction $(r,1,0)$ and is top-dimensional
    in $\R^3/\bran{(1,1,1)}$.
  \end{enumerate}
\end{example}
\begin{figure}[h]
  \centering\scalebox{.7}{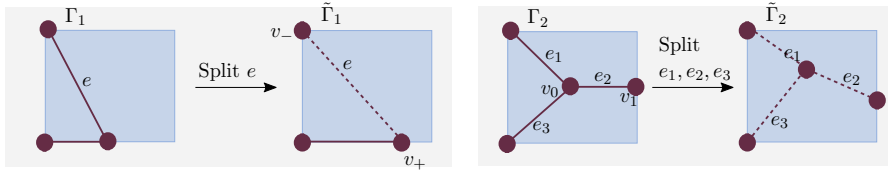}
  \caption{In 
  the examples shown, 
  dropping the direction condition on split edges automatically
  produces split tropical graphs.}
\label{fig:drop-match}
\end{figure}

\begin{example}\label{ex:drop-match}
  A split tropical graph need not always have additional vertices
  compared to the base tropical graph.  In some cases, dropping the
  direction condition \eqref{T-slopecond} on split edges (and
  analogously the matching condition on the corresponding split nodes
  $w(e)$) produces a top-dimensional discrepancy cone
  $\Disc(\tGam)$. For example in Figure \ref{fig:drop-match} dropping
  the direction condition at edge $e$ in the graph $\Gamma_1$ produces
  one dimension of symmetry in the tropical graph $\tGam_1$. That is,
  the graph $\tGam_1$, which is $\Gamma_1$ with the direction
  condition at $e$ forgotten, is a split tropical graph with
  discrepancy cone $\fw(\tGam_1)=\R$. Similarly dropping the direction
  condition on edges $e_1$, $e_2$, $e_3$ in $\Gamma_2$ in Figure
  \ref{fig:drop-match} produces three dimensions of symmetry in
  $\tGam_2$. Indeed in $\tGam_2$, the vertex $v_0$ is free to move in
  two dimensions, and the vertex $v_1$ is independently free to move
  in one dimension. Therefore, $\tGam_2$ is a split tropical graph
  with discrepancy cone $\fw(\tGam_2)=\R^3$.  In general, dropping the
  direction condition on split edges does not guarantee a
  top-dimensional discrepancy cone, such as in the graph $\Gamma$ in
  Figure \ref{fig:2ddef} and in Figure \ref{fig:4free} below.
\end{example}
\begin{figure}[h]
  \centering\scalebox{.8}{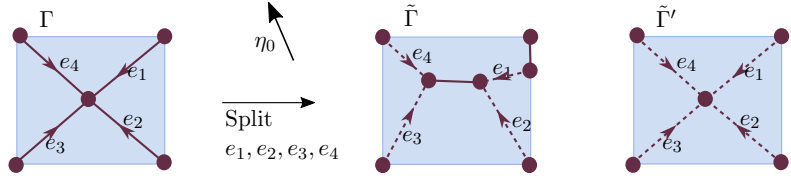}
  \caption{ $(\tGam$, $\Gamma)$ is a split graph with cone direction
    ${\etasp}$ and split edges ordered as $e_1,\dots,e_4$. $(\tGam'$,
    $\Gamma)$ is not a split tropical graph.}
  \label{fig:4free}
\end{figure}
\begin{example}\label{ex:4free}
  In this example we show that the increasing cone condition for split
  tropical graphs is necessary, and it can not be replaced by the
  condition that the cone direction is contained in the individual
  discrepancy cones of each of the split edges. For example, the
  quasi-split graph $(\tGam',\Gamma)$ in Figure \ref{fig:4free}
  satisfies the latter condition, namely, for each split edge
  $e=(v_+,v_-) \in \Edge_s(\tGam')$,
  \[\pi_{\cT(e)}^\perp({\etasp}) \in \R_+\bran{\{\cT(v_+) - \cT(v_-):
      \cT \in \W(\tGam',\Gamma)\}},\]
  but the discrepancy cone is not top-dimensional. In fact the
  discrepancy cone of $\tGam'$ is two-dimensional since the vertex
  $v_0$ can only move in two dimensions and the other vertices are
  fixed. However, $(\tGam,\Gamma)$ in Figure \ref{fig:4free} satisfies
  the \hyperref[eq:cc]{(Cone condition)}, and therefore, is a split tropical graph. We
  continue analyzing this graph in the next example.
\end{example}

\begin{example}\label{ex:inc-4free}
  The increasing cone condition for split graphs is equivalent to
  saying that we may split one edge at a time (following the ordering
  $\prec$), and at each step the cone direction ${\etasp}$ is in the
  discrepancy cone. More precisely, given a tropical graph $\Gamma$
  with a set of ordered split edges $e_1,\dots,e_n$, the increasing
  cone condition is equivalent to the existence of a sequence of split
  tropical graphs $(\tGam_1,\Gamma), \dots, (\tGam_n,\Gamma)$ such
  that for any $k$ the split edges of $(\tGam_k,\Gamma)$ are
  $e_1,\dots,e_k$, and for a small enough relative vertex position
  $\cT \in \W(\tGam_{k-1},\Gamma)$,
  \begin{multline*}
    \pi^\perp_{\cT(e_k)}({\etasp}) \in \R_+\{\Diff_{e_k}(\cT') : \cT'
    \in \W(\tGam_k,\Gamma), \\ \Diff_e(\cT')=\Diff_e(\cT) \text{ for }
    e=e_1,\dots,e_{k-1}\}.
  \end{multline*}
  For the split tropical graph $(\tGam,\Gamma)$ in Figure
  \ref{fig:4free} the intermediate split graphs are given in Figure
  \ref{fig:inc-4free}. Splitting each of the last two edges $e_3$,
  $e_4$ has the effect of `forgetting the edge direction', and is
  therefore not shown separately in Figure \ref{fig:inc-4free}.
\end{example}
\begin{figure}[h]
  \centering\scalebox{.7}{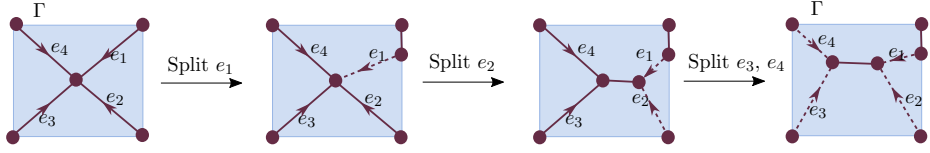}
  \caption{Splitting edges one at a time yields the increasing cone
    condition in the final graph.}
  \label{fig:inc-4free}
\end{figure}

\section{Split maps and split Fukaya algebras}

\subsection{Split maps}
A split map is modelled on a split tropical graph, and it is a version
of a broken map where there is no matching condition along split
edges. The definition of a split tropical graph requires us to choose
a generic cone direction ${\etasp} \in \t$ as in Definition
\ref{def:splitgr}.  Therefore, the definition of a split map also
requires the choice of a vector ${\etasp}$, although it is not
explicitly mentioned in this Section.
%
\begin{definition} {\rm(Split map)}
  \label{def:split-map}
  Given a split tropical graph $\tGam \to \bGam$, a \em{split map} consists of
  \begin{itemize}
  \item a collection of maps
  \[ u: C \to \XX = (u_v: C_v \to \ol \XC_{P_{\tGam}(v)})_{v \in
      \Ver(\tGam) }\]
  whose domain $C$ is a treed disk of type $\tGam$,
  \item   and framing
  isomorphisms for all \em{non-split} tropical edges
  \[\fr : T_{w_+(e)}\tC \tensor T_{w_-(e)}\tC \to \C, \quad e \in
    \Edge_\trop(\tGam) \bs \Edge_s(\Gamma),\]
  such that the edge-matching condition (as in Definition
  T-\ref{T-def:bmap}) is satisfied for edges
  $e \in \Edge(\tGam) \bs \Edge_s(\bGam)$. 
  \end{itemize}
  At the split edges
  $e \in \Edge_s(\bGam)$, the edge matching condition is dropped.
\end{definition}

The moduli spaces decompose into strata indexed by combinatorial
types, as before.  The \em{type} of a split map is given by
the split tropical graph $\tGam \to \bGam$ and the tangency and
homology data on
$\tGam$. 
The type of a split map is denoted by $\tGam$ when the base tropical
type $\bGam$ is clear from the context.

We wish to count rigid split maps to define the split Fukaya
algebra.

\begin{definition}\label{def:rigidsplit}
  {\rm(Rigidity for split maps)}
  The type $\tGam \to \bGam$ of a split map is \em{rigid} if
  \begin{enumerate}
  \item $\bGam$ is a rigid tropical graph,
  \item in $\tGam$ the only non-tropical edges
    $e \in \Edge(\tGam) \bs \Edge_\trop(\tGam)$ are boundary edges
    $e \in \Edge_{\white}(\tGam)$ with finite non-zero length
    $\ell(e) \in (0,\infty)$,
  \item for all interior markings $z_e$,
    $e \in \Edge_{\black,\to}(\tGam)$, the intersection multiplicity
    with the stabilizing divisor is $1$,
  \item and the cone of tropical vertex positions $\fw(\tGam,\bGam)$
    has dimension
    \begin{equation} \label{rightdim} \dim(\fw(\tGam,\bGam)) =
      |\Edge_s(\bGam)|(\dim(\t) -1).
    \end{equation}
  \end{enumerate}
\end{definition}

The definition of tropical symmetry in the case of split maps is
adjusted to reflect the lack of matching conditions at split edges.
\begin{definition}\label{def:trop-sym-split}
  A \em{tropical symmetry} on a split map with graph $\tGam \to \bGam$
  consists of a translation $g_v \in T_{P(v),\C}$ for each vertex
  $v \in \Ver(\tGam)$, and a framing translation $z_e \in \C^\times$
  for each non-split tropical edge
  $e \in \Edge_\trop(\tGam) \bs \Edge_s(\Gamma)$ that satisfies
  \begin{equation} \label{eq:symmedge} g(v_+) g(v_-)^{-1} =
    z_e^{\cT(e)}, \quad \forall e=(v_+,v_-) \in \Edge_\black(\tGam)
    \bs \Edge_s(\bGam),
  \end{equation}
  where we assume $z_e=1$ for non-tropical edges
  $e \in \Edge_-(\tGam) \bs \Edge_\trop(\tGam)$.
  Denote the group of tropical symmetries as
  \begin{equation}
    \label{eq:tsym-split}
    T_{\on{trop}}(\tGam) = \{ ( (g_v)_v, (z_e)_e ) \}.
  \end{equation}
\end{definition}

\begin{remark}\label{rem:dimtrop}
  {\rm(Dimension of the tropical symmetry group)} For a split tropical
  graph $\tGam \to \bGam$, the cone of relative vertex positions
  $\fw(\tGam,\Gamma)$ generates a subgroup of tropical symmetries: A
  relative vertex position $\cT_\tGam - \cT_\Gamma$ generates a
  subgroup
  \[\{(\exp z(\cT_\tGam(v) - \cT_\Gamma(\kappa v)))_{v \in
      \Ver(\tGam)} : z \in \C^\times\} \subset T_\trop(\tGam),\]
  see proof of Lemma T-\ref{T-lem:wtgen}.  For a split tropical graph
  $\tGam$, by Proposition \ref{prop:discrdim} the dimension of the cone $\fw(\tGam,\Gam)$ is at least
  $|\Edge_s(\Gam)|(\dim(\t) -1)$, and therefore,
  \begin{equation}
    \label{eq:dimlowerbd}
    \dim_\C(T_\trop(\tGam)) \geq |\Edge_s(\Gam)|(\dim(\t) -1).  
  \end{equation}
  If $\tGam$ is rigid, then \eqref{eq:dimlowerbd} is an equality.
\end{remark}

\begin{remark}
  {\rm(Splitting of the tropical symmetry group)} For a split tropical
  graph $\tGam \to \bGam$, the group $T_\trop(\tGam)$ is a product
  \begin{equation}
    \label{eq:trop-split}
    T_\trop(\tGam)=T_\trop(\tGam_1) \times \dots \times T_\trop(\tGam_s)
  \end{equation}
  of tropical symmetry groups of the connected subgraphs
  \[\tGam_1,\dots,\tGam_s \subset \tGam \bs \Edge_s(\bGam),\]
  since there are no matching conditions \eqref{eq:symmedge} on split
  edges.
\end{remark}

\begin{example}
  \begin{enumerate}
  \item In Figure \ref{fig:2ddef}, for the split tropical graph
    $(\tGam_1, \Gamma)$, the tropical symmetry group
    $T_\trop(\tGam_1) \simeq \C^\times$ splits as
    \[T_\trop(\tGam_1^+)=\{g_{v_+} \in e^{(2,1)t} : t \in \C^\times\}
      \simeq \C^\times, \quad T_\trop(\tGam_1^-)=\{\Id\}.\]
    For the graph $(\tGam_2, \Gamma)$,
    $T_\trop(\tGam_2)\simeq \C^\times$, and the splitting is
    \[T_\trop(\tGam_2^+)=\{\Id\}, \quad T_\trop(\tGam_2^-)=\{g_{v_-}
      \in e^{(1,0)t} : t \in \C^\times\}.\]
  \item In Figure \ref{fig:3ddef-intro}, for the split tropical graph
    $(\tGam_2, \Gamma)$, the tropical symmetry group
    $T_\trop(\tGam_2) \simeq (\C^\times)^2$ splits as
    \[T_\trop(\tGam_2^+)=\{g_{v_+} \in e^{(2,1,0)t} : t \in
      \C^\times\} \simeq \C^\times, \quad T_\trop(\tGam_2^-)=
      \{g_{v_-} \in e^{(1,2.0)t} : t \in \C^\times\} \simeq
      \C^\times.\]
  \end{enumerate}
\end{example}

To relate deformed maps to split maps, we will need a variation of
split maps where there is a matching condition on split nodes
also. Such maps should be thought of as lying on a slice of the
$T_\trop(\tGam,\Gamma)$-action.

\begin{definition}\label{def:fr-split}
  \begin{enumerate}
  \item {\rm(Framed split map)} A \em{framed split map} is a split map
    $u$ together with an additional datum of framings on split edges
    \[\fr_e:T_{w_+(e)}\tC \tensor T_{w_-(e)}\tC \to \C, \quad e \in
      \Edge_s(\Gam)\]
    such that the \hyperref[eq:nodematch-intro]{(Matching condition)}
    is satisfied at the split edges.
    %
  \item {\rm(Tropical symmetry for framed split maps)} A \em{tropical
      symmetry} for a framed split map with graph $\tGam \to \bGam$ is
    a tuple
    \[((g_v \in T_{P(v),\C})_{v \in \Ver(\Gamma)}, (z_e \in
      \C^\times)_{e \in \Edge_\trop(\tGam)})\] satisfying a matching
    condition on all interior edges :
    \begin{equation} \label{eq:symmedge-slice} g(v_+) g(v_-)^{-1} =
      z_e^{\cT(e)} \quad \forall e \in \Edge_\black(\tGam).
    \end{equation}
    The group of tropical symmetries for framed split maps is denoted
    by
    \begin{equation}
      \label{eq:tfr}
      T_{\trop,\fr}(\tGam,\Gamma) =  \{ ( (g_v)_v, (z_e)_e ) \ | \
      \eqref{eq:symmedge-slice} \}.
    \end{equation}
  \item {\rm(Multiplicity of a split tropical graph)} For a rigid
    split tropical graph $\tGam \to \Gam$, the group
    $T_{\trop,\fr}(\tGam)$ is finite, and
    \begin{equation}
      \label{eq:splitmult}
      \mult(\tGam):=|T_{\trop,\fr}(\tGam)|  
    \end{equation}
    is called the multiplicity of $\tGam$. See Example
    \ref{ex:framed-trop-sym}.
  \end{enumerate}
\end{definition}

\begin{remark}
  A framed split map is not a broken map, because the underlying
  tropical graph $\tGam$ does not satisfy the (Direction) condition on
  split edges.
\end{remark}

\begin{example}\label{ex:framed-trop-sym}
  {\rm(Framed tropical symmetry group)} The split tropical graph
  $\tGam_2 \to \Gam$ in Figure \ref{fig:3ddef-intro} is rigid, and
  consequently the framed tropical symmetry group
  $T_{\trop,\fr}(\tGam_2)$ is finite. However, this group is
  non-trivial. An element $(g,z) \in T_{\trop,\fr}(\tGam_2)$ satisfies
  the equations
  %
  \[ g_{v_0}=g_{v_1}=\Id, \quad g_{v_0}g_{v_+}^{-1}=z_{e_+}^{(2,1,0)},
    \quad g_{v_+}g_{v_-}^{-1}=z_{e}^{(1,1,1)}, \quad
    g_{v_-}g_{v_1}^{-1}=z_{e_-}^{(1,2,0)}.\]
  There are 3 elements in $T_{\trop,\fr}(\tGam_2)$ given by
  $z_{e_+}=z_{e_-}=\om$ where $\om$ is a cube root of unity.
\end{example}

\subsection{Split Fukaya algebras}
We define moduli spaces of split maps and use them to define
composition maps for split Fukaya algebras.  To define the moduli
spaces, we fix a generic cone direction ${\etasp} \in \t$. A
perturbation datum for split maps $\ul \Pe =(\Pe_\tGam)_\tGam$ is the
same as a perturbation datum for broken maps as in Section
T-\ref{T-sec:domdep}, and consists of maps
\[\Pe_\tGam=(J_\tGam, F_\tGam), \quad J_\tGam : \S_\tGam \to
  \J^\cyl_{\om_\XX,\tau}(\XX), \enspace F_\tGam : \T_\tGam \to C^\infty(L,\R)\]
 for all types $\tGam$ of curves
with base, with coherence conditions corresponding to morphisms of
stable treed curves.  For each such type $\tGam$, $J_\tGam$ is a
domain-dependent cylindrical $\om_\XX$-tamed almost complex structure (see \eqref{eq:omxxtau}) on the broken manifold, and
$F_\tGam$ is a domain-dependent Morse function on the Lagrangian
$L$. Note that $(J_\tGam, F_\tGam)$ depends only on the domain treed
curve, and not the tropical structure.

\begin{definition}\label{def:mod-split}
  {\rm(Moduli spaces of split maps)}
  Let ${\etasp} \in \t$ be a generic cone direction.  The \em{moduli
    space of split maps} of type $\tGam$ with cone direction
  ${\etasp}$ modulo the action of domain reparametrizations is denoted
  \[\tM^\spl_{\tGam}(L, \Pe_\tGam,{\etasp}).\]
  Quotienting by the action of the tropical symmetry group defines the
  \em{reduced moduli space}
  \[\M^\spl_{\tGam,\red}(L,{\etasp}):=\tM^\spl_{\tGam}(L,\Pe_\tGam,
    {\etasp})/T_\trop(\tGam,\bGam).\] The moduli space of framed split
  maps of type $\tGam$ modulo the action of domain reparametrizations
  is denoted
  \[\tM^\spl_{\fr,\tGam}(L, \Pe_\tGam,{\etasp}).\]
  The quotient $\tM^\spl_{\fr,\tGam}/T_{\trop,\fr}(\tGam)$ is equal to
  the reduced moduli space $\M^\spl_{\tGam,\red}(L,{\etasp})$.  (Split
  maps and framed split maps are defined in Definitions
  \ref{def:split-map}, \ref{def:fr-split} respectively.)  This ends
  the definition.
\end{definition}
\begin{remark}
  If the split tropical graph $\tGam$ is rigid then
  $T_{\trop,\fr}(\tGam)$ is a finite group, and
  $\tM^\spl_{\fr,\tGam}(L,{\etasp})$ is a finite cover of the reduced
  moduli space.
\end{remark}
\begin{remark} \label{rem:dim-split} The expected dimension of strata
  of split maps can be computed as follows. Let $(\ti{\Gamma},\Gamma)$
  be a type of split map. Let $\hat \Gamma$ be the broken map type
  obtained by collapsing all the tropical edges of $\tGam$ that do not
  occur in $\Gamma$, that is we collapse the edges
  \[e \in \Edge_{\trop}(\ti{\Gamma}) \bs \Edge(\Gamma).\]
  Thus the tropical graph of $\hat \Gamma$ is equal to $\Gamma$, and
  the boundary edges of $\hat \Gamma$ are the same as that of $\tGam$.
  Then the moduli space of split maps $\tM^\spl_{\tGam}(L,\eta)$ with
  boundary end points $\ul x \in (\cI(L))^{d(\white)}$ has expected
  dimension
%
  \[i^\spl(\tGam,\ul x):=i^\br(\hat\Gam,\ul x) +
    2|\Edge_s(\bGam)|(\dim(\t) - 1). \]
  Here $i^\br$ is the index function for types of broken maps defined
  in T-\eqref{T-eq:expdim}.  Indeed dropping the node matching
  condition \eqref{eq:match-proj} adds $2(\dim(\t)-1)$ dimensions for
  each split edge, and gluing along tropical edges $e$ doesn't change
  dimension as proved in Proposition T-\ref{T-prop:expdim}.  For a
  rigid split map, the tropical symmetry group
  $T_{\on{trop}}(\tGam,\bGam)$ has dimension
  $2|\Edge_s(\bGam)|(\dim(\t) - 1)$.  The expected dimension of the
  quotiented moduli space is then
  \[i^\spl_{\red}(\tGam,\ul x):=i^\spl(\tGam,\ul x)-
    2|\Edge_s(\bGam)|(\dim(\t) - 1)=i^\br(\hat \bGam,\ul x). \]
  This ends the Remark.
\end{remark}

\begin{proposition} \label{prop:defmoduli} 
  {\rm (Regularity for split  maps)} 
Let ${\etasp} \in \t$ be a generic element.  Let
  $(\tGam,\Gamma)$ be a type of an uncrowded split map, and suppose
  regular perturbation data for types $(\ti\Gamma',\Gamma)$ of based
  treed disks with $\ti\Gamma' < \ti\Gamma$ is given (where the
  ordering on types is as in \eqref{T-eq:type-order}). Then there is a
  co-meager subset $\PPe^{reg}_\tGam \subset \PPe_\tGam$ of regular
  perturbations for split maps of type $\tGam$ coherent with the
  previously chosen data such that if $\Pe_\tGam \in \PPe^{reg}_\tGam$
  then the following holds.  Let $\ul x \in \cI(L)^{d(\white)}$ be
  boundary end points for which the index
  $i^\spl_\red((\tGam,\bGam),\ul x)$ is $\leq 1$.  Then
  \begin{enumerate}
  \item \label{part:transv} {\rm(Transversality)} the moduli space of
    split maps $\tM^\spl_{\tGam}(\Pe_\tGam,L,{\etasp})$, framed split
    maps $\widetilde \M^\spl_{\fr,\tGam}(\Pe_\tGam,L,{\etasp})$ and
    the reduced moduli space
    $\M^\spl_{\tGam,\red}(\Pe_\tGam,L,{\etasp})$ are manifolds of
    expected dimension.
  \item \label{part:cpt}{\rm(Compactness)} The moduli space
    $\tM^\spl_{\tGam,\red}(\Pe_\tGam,L,{\etasp})$ is compact if
    $i^\spl_\red(\tGam \to \bGam,\ul x)=0$. If this index is $1$, then
    the compactification
    $\ol {\tM}^\spl_{\tGam,\red}(\Pe_\tGam, L , {\etasp})$ consists of
    codimension one boundary points which contain a boundary edge
    $e \in \Edge_\white(\tGam)$ with $\ell(e)=0$ or $\ell(e)=\infty$.
  \item \label{part:tub-split} {\rm(Tubular neighborhoods)} Let
    $\tGam$ be a type of split map with a single edge $e$ that is
    broken or has length $\ell(e)$ zero, and let the input/output
    tuple $\ul x$ be such that $i^\spl_{\red}(\tGam,\ul x)=0$. Then
    $\M^\spl_{\tGam,\red}(L,{\etasp})$ has a one-dimensional tubular
    neighborhood in any adjoining one-dimensional strata of split maps
    with the expected orientations. (The set of adjoining strata and
    the expected orientations are as in Theorem
    T-\ref{T-thm:bdry-glue}.)
  \item \label{part:bdry-split} {\rm(True boundary)} The true boundary
    of the oriented topological manifold
    \[\bigcup_{\tGam : \tGam \text{is rigid, } i^\spl_{\red}(\tGam,
        \ul x)=0}\ol \M^\spl_{\tGam,\red}(L,\ul x).\]
is in bijection with 
to the reduced moduli space of split maps with a
    single broken edge, with the orientation sign adjusted by a factor
    $(-1)^\circ$ which depends only on the type $\tGam$, $\ul x$ (and
    is the same as the corresponding sign for broken maps in Theorem
    T-\ref{T-thm:bdry-glue} \eqref{T-part:orien2}).
  \end{enumerate}
\end{proposition}

\begin{proof}[Proof of Proposition \ref{prop:defmoduli}]
  The proof of the first statement is by a Sard-Smale argument similar
  to the case of broken maps. The only variation is that there is no
  matching condition {\eqref{eq:symmedge-slice}} on the split edges.

  We prove the compactness result for the moduli space of framed split
  maps since that is a finite cover over the reduced moduli space.
  The compactness statement for framed split maps can be proved in
  exactly the same way as the corresponding proof for broken maps.  In
  fact, given a sequence of framed split maps $u_\nu$ of the same
  type, the limit is obtained by translations
  $\{t_\nu(v)\}_{v \in \Ver(\tGam)}$ that satisfy the
  \hyperref[eq:direction-rel]{(Direction)} 
  condition for all edges, including split
  edges. Consequently, the limit is a framed split map $u$.  The
  codimension one boundary strata are as required in the Proposition
  via the arguments in the proof of Proposition
  T-\ref{T-prop:truebdry}. The finiteness of the tropical symmetry
  group used in the proof of Proposition T-\ref{T-prop:truebdry} is
  replaced by the finiteness of the framed symmetry group
  $T_{\trop,\fr}(\tGam)$ of the split tropical graph $\tGam$.

  The proofs of parts \eqref{part:tub-split} and
  \eqref{part:bdry-split} are the same as the proofs of the
  corresponding results (Theorem T-\ref{T-thm:bdry-glue} and Remark
  T-\ref{T-rem:true-fake}) for broken maps.
\end{proof}

\begin{proposition}\label{prop:finsplit}
  Given $E>0$, there are finitely many types of split maps that have
  area at most $E$.
\end{proposition}
\noindent The proof is identical to the corresponding result Proposition
T-\ref{T-prop:finno} for broken maps, and is omitted.

Next we define the \em{split Fukaya algebra} whose composition maps
are given by counts of rigid split maps. In order to obtain homotopy
equivalence with other versions of Fukaya algebras, the definition
must use framed split maps. We instead use a count of split maps in
the reduced moduli space weighted by the size of the discrete tropical
symmetry group, since the latter is a finite quotient of the former.

\begin{definition} \label{def:tfa}
  Let ${\etasp} \in \t$ be a generic cone direction, and let $\ul \Pe$
  be a regular perturbation datum for all based curve types.  The
  \em{split Fukaya algebra} is the graded vector space
  \[CF_{\split}(\XX,L,{\etasp}):=CF^{\on{geom}}(L) \oplus \Lam
    x^{\greyt}[1] \oplus \Lam x^{\whitet}\]
  equipped with composition maps
  \begin{equation}\label{eq:md-split}
    m^{d(\white)}_\spl(x_1,\ldots,x_{d(\white)}) =
    \sum_{x_0,u \in \M_{\tGam,\spl,\red}(\XX, L,D,\ul{x})_0}
    \mult(\tGam)  w_s(u)   x_0 \end{equation}
  where
  \begin{equation}
    \label{eq:wtwsu-split}
    w_s(u):=(-1)^{\heartsuit}(d_\black(\Gam)!)^{-1} (s(\tGam)!)^{-1} \Hol([\partial u]) \eps(u)
    q^{A(u)},
  \end{equation}
  where $s(\tGam)$ is the number of split edges in the type $\tGam$
  and the other symbols in \eqref{eq:wtwsu-split} are as in
  T-\eqref{T-eq:wtwu}. The orientation sign for split maps is defined
  in the same way as broken maps, see Remark
  T-\ref{T-rem:orientmap}. (Indeed, dropping a matching condition,
  which is a complex condition does not affect the determinant line
  bundle.) In \eqref{eq:md-split}, the combinatorial type $\tGam$ of
  the split map $u$ ranges over all rigid types with $d(\white)$
  inputs (see Definition \ref{def:rigidsplit} for rigidity), and
  $\mult(\tGam) := |T_{\trop,\fr}(\tGam)|$ is the multiplicity of the
  split map from \eqref{eq:splitmult}.
\end{definition}

\begin{remark}
  Any one-dimensional moduli space of split maps has true and fake
  boundary strata as in Figure \ref{fig:truebdry-split}, which look
  very similar to their analogue for broken maps shown in Figure
  T-\ref{T-fig:true-fake}.
\end{remark}

\begin{figure}[h]
  \begin{center}
    \scalebox{.6}{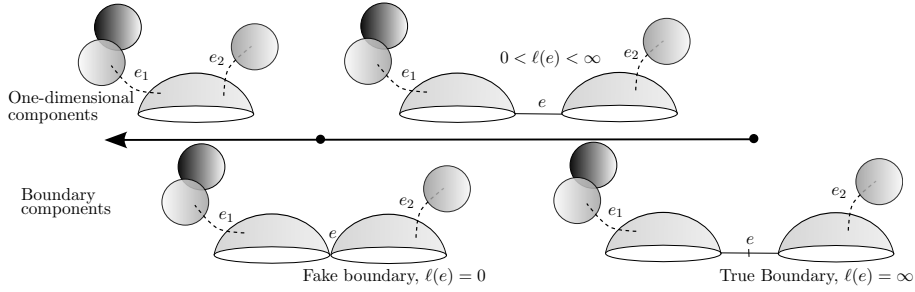} \end{center}
  \caption{True and fake boundary strata of a one-dimensional
    component of the moduli space of split maps with split edges
    $e_1$, $e_2$. The sphere components lie in different pieces of the
    tropical manifold.}
  \label{fig:truebdry-split}
\end{figure}

The composition maps $(m^{d(\white)}_\spl)_{d(\white) \ge 0}$ satisfy
the \ainfty-axioms. Indeed, as in the broken and unbroken cases, the
\ainfty-axioms are proved by a count of end-points of one-dimensional
moduli spaces of rigid split maps; the boundary components are
described by Proposition \ref{prop:defmoduli} \eqref{part:bdry-split};
and the combinatorial factors arising from the distribution of
interior markings are accounted exactly as in the unbroken case in
Theorem T-\ref{T-thm:yields-unbr}.  Therefore,
$ CF_{\split}(\XX,L,\ul \Pe, {\etasp})$ is an \ainfty-algebra.
 
\begin{remark}\label{rem:factorize}
  The composition map $m_\trop^d$ can be expressed as a sum of
  products. The sum is over rigid types $\tGam$ of split maps, and the
  product is over the connected components
  $\tGam_1,\dots,\tGam_{s(\tGam)}$ of the graph
  $\tGam \bs \Edge_s(\bGam)$. Such a decomposition is possible because
  split maps do not have any edge matching conditions. Therefore, for
  any type $\tGam$ of a rigid tropical map, the reduced moduli space
  $\M^\spl_{\tGam,\red}(L,\ul \eta)$ is a product of reduced moduli
  spaces :
  \[ \M^\spl_{\tGam,\red}(L,\ul \eta)=
    \prod_{i=1}^{s(\tGam)}\M_{\tGam_i,\red}.\]
  Here, $\M_{\tGam_i}$ is the quotient
  \begin{equation}
    \label{eq:moduli-product}
    \M_{\tGam_i,\red}:=\tM_{\tGam_i}/T_\trop(\tGam_i,\bGam_i),  
  \end{equation}
  of the moduli space $\tM_{\tGam_i}$ of maps modelled on the subgraph
  $\tGam_i \subset \tGam$ by the component of the tropical symmetry
  arising from $\tGam_i$ as in \eqref{eq:trop-split}. Thus the
  composition map decomposes as
  \begin{equation}
    \label{eq:mdsplit}
    m^d_\spl=\sum_{\tGam:d_\white(\tGam)=d} m^d_{\spl,\tGam}, \quad m^d_{\spl,\tGam}=
    \tfrac {(-1)^{\heartsuit} \mult(\tGam)} {d_\black(\tGam)! s(\tGam)!} \prod_{i=1}^{s(\tGam)} m_{\tGam_i},
  \end{equation}
  where
  \begin{itemize}
  \item the sum ranges over all rigid types $\tGam$ of split maps with
    $d$ inputs
  \item and
    \[m_{\tGam_i}:= \sum_{u \in
        \M_{\tGam_i,\red}(L,{\etasp})_0}\Hol([\partial u]) \eps(u)
      q^{A(u)}\] is a weighted count of rigid elements in
    $\M_{\tGam_i,\red}$, where $y$, $\eps$ are as in
    \eqref{T-eq:wtwu}.
  \end{itemize}
  We remark that in the product decomposition
  \eqref{eq:moduli-product} of moduli spaces, exactly one of the
  components $\tGam_0$ has boundary on the Lagrangian.  For all other
  components $\tGam_i$, $i \neq 0$, the domain components are
  spheres. Therefore, the number $m_{\Gamma_i}$ is invariant under
  choices of perturbations, and may be viewed as a relative
  Gromov-Witten invariant.  Finally, we point out that the
  decomposition in \eqref{eq:moduli-product} is not a full splitting
  since any component $\tGam_i$ may contain tropical nodes satisfying
  matching conditions. \label{page:factor}
\end{remark}

\section{Homotopy equivalences}
\label{sec:homo2split}

In this section, we will show that deformed Fukaya algebras are
\ainfty homotopy equivalent to split Fukaya algebras. Recall that
for the split Fukaya algebra $CF_{\split}(\XX,L,{\etasp})$ the
composition maps are counts of split disks with a generic cone
direction ${\etasp} \in \t$.  On the other hand, defining a deformed
Fukaya algebra $CF_{\deform}(\XX,L,\ul \eta)$ requires us to choose a
coherent deformation datum $\ul \eta=(\eta_\tGam)_\tGam$, which
consists of a deformation parameter for each split edge
\begin{equation}
  \label{eq:tgamc}
  (\eta_\tGam(e,[C]))_{e \in \Edge_s(\tGam)} \in T_{\Gam,\C}:= \prod_{e \in \Edge_s(\Gam)} T_\C
  /T_{\cT(e),\C}  
\end{equation}
varying smoothly with the domain curve $C$.  To prove the homotopy
equivalence we choose a sequence of deformation data
$(\ul \eta_\nu)_\nu$ compatible with the cone direction ${\etasp}$,
which roughly means that for any fixed domain curve $[C]$ and any
split edge $e$ the deformation parameter
$\eta_{\nu,\tGam}([C]) \in T_\C/T_{\cT(e),\C}$ goes to infinity in the
direction ${\etasp}$.  Split maps with cone direction ${\etasp}$ occur
as limits of sequences of $\ul \eta_\nu$-deformed maps of bounded
area. Conversely a rigid split map can be glued to yield a
$\ul \eta_\nu$-deformed map for large enough $\nu$.  This bijective
relation leads to the proof of the \ainfty homotopy equivalence.

We point out that for a converging sequence of deformed maps of index
zero, the index of the limit split map is equal to
$ \dim(T_{\Gam,\C})$ (defined in \eqref{eq:tgamc}), which is the
dimension of its tropical symmetry group.  The higher index of the
limit split map is accounted for by the absence of matching conditions
on split edges.  Thus in this section, we produce a cobordism from the
moduli space of framed split maps, which is a slice of the action of
the tropical symmetry group on the moduli space of split maps, to the
space of $\ul \eta$-deformed maps for any deformation datum
$\ul \eta$.

We start by defining a sequence of deformation data compatible with
the cone direction.
\begin{definition} {\rm(Compatible deformation data for a cone
    direction)}
  \label{def:compdef}
  Given a generic cone direction ${\etasp} \in \t$ and a type
  $(\tGam,\Gamma)$ of based curves, a sequence of deformation data
  \[ \eta_{\tGam,\nu} : \M_{\tGam} \to \t_\Gam \simeq \bigoplus_{e \in
      \Edge_s(\Gam)} \t/\t_{\cT(e)} \]
  for a curve type $\tGam$ with base $\Gam$ is compatible with the
  cone direction ${\etasp}$ if
  \[\eta_{\tGam,\nu}=c_{\tGam,\nu} \pi^\perp_{\cT(e)}({\etasp}), \]
  and
  \[c_{\tGam,\nu}:\M_{\tGam} \times \Edge_s(\Gamma) \to \R_+\]
  is a continuous function such that for any converging sequence
  $m_\nu \to m$ in $\ol \M_{\tGam}$,
  $(c_{\tGam,\nu}(m_\nu,e))_{e \in \Edge_s(\Gamma),\nu}$ is an
  increasing sequence of tuples (as in Definition
  \ref{def:inc-tup-e}).
\end{definition}

For the gluing proof, we need a more restrictive class of deformation
data.
\begin{definition}
  {\rm(Uniformly continuous deformation data)} A sequence of
  deformation data $\ul \eta$ compatible with a cone direction
  ${\etasp}$ is \em{uniformly continuous} if for any split tropical
  graph $\tGam$ there is a constant $k$ such that for any two curves
  $C$, $C'$ of type $\tGam$
  \[\abs{\eta_{\tGam,\nu}([C]) - \eta_{\tGam,\nu}([C'])} \leq k
    \sum_{e \in \Edge_{{\white},-}}\abs{\ell(T_e,[C]) -
      \ell(T_e,[C'])} \]
  for all $\nu$. We recall that $\ell(T_e,[C])$ is the length of the
  treed segment $T_e$ in the curve $C$.
\end{definition}

Next, we construct a sequence of deformation data that is coherent,
compatible with the cone direction ${\etasp}$ and uniformly
continuous. We recall that moduli spaces of deformed maps are defined
using domain-dependent deformation data that are coherent under
morphisms of based curve types, and satisfy a
\hyperref[eq:taufactor]{(Marking independence)} property.

\begin{lemma}\label{lem:defseq}
  Given a generic cone direction ${\etasp} \in \t$, there exists a
  compatible sequence $\{\ul \eta_\nu\}_\nu$ of uniformly continuous
  coherent deformation data (as in Definition \ref{def:coh-deform}).
\end{lemma}
\begin{proof}
  We first describe connected components of the moduli space of based
  curves.  A connected component is determined by a tuple
  $\gamma:=(\bGam, m,n,\gamma_0)$ consisting of the base tropical
  graph $\bGam$, the number $m$ of boundary markings, the number $n$
  of interior markings, and a function
  \[\gamma_0: \{1,\dots,n\} \to \Ver(\bGam) \]
  that maps a marked point $z_i$ on a curve component $C_v \subset C$
  to $\kappa(v) \in \Ver(\bGam)$, and we denote the component by
  $\M_\gamma$.  The sequence of functions
  \begin{equation}
    \label{eq:cgndef}
    c_{\gamma,\nu} : \M_\gamma   \times \Edge_s(\bGam) \to \R    
  \end{equation}
  is constructed by induction on $m$, $n$, $|\Ver(\bGam)|$.  We
  consider a connected component $(\bGam,n,m,\gamma)$, and assume that
  a sequence of coherent data has been constructed for smaller
  types. The data on the smaller types determines $c_{\gamma,\nu}$ on
  the true boundary of $\M_{\gamma}$, which consists of curves with at
  least a single broken boundary edge.

  The deformation sequence on $\M_\gamma$ is constructed via a
  partition
  \[\M_\gamma=\cup_{i \geq 0} \M_\gamma^i,\]
  where $\M_\gamma^i$ consists of treed curves with exactly $i$
  boundary edges.  On $\M_\gamma^0$, for any split edge $e$ we fix
  $c_{\gamma,\nu}(m,e)$ to be a sequence of constant functions
  \[c_{\gamma,\nu}(m,e):=c_\nu(e) \in \R_+ \]
  such that $(c_\nu(e))_{e \in \Edge_s(\Gamma)}$ is an increasing
  sequence of tuples (as in \eqref{eq:incsplit}) with respect to the
  ordering $\prec_\tGam$ of split edges. Let
  \[C_\nu := \sup_i \{c_\nu(m,e) : m \in \M^0_\gamma \cup \partial
    \M_\gamma\}. \]

  We extend the deformation sequence to all of $\M_\gamma$ by
  interpolating between the codimension one strata
  $\partial \M_\gamma^0$ and $\partial \M_\gamma$.  The boundary
  $\partial \M_\gamma^0$ partitions into sets
  \[\partial \M_\gamma^0=\cup_{i \geq 1} \M_{\gamma}^{0,i}, \quad
    \M_{\gamma}^{0,i}:=\ol{\M_\gamma^0} \cap \M_\gamma^i,\]
  and $\M_\gamma^i$ is a product
  \begin{equation}
    \label{eq:bdryproduct}
    \M_\gamma^i=\M_\gamma^{0,i} \times [0,\infty]^i,   
  \end{equation}
  such that projection to the second factor is equal to the edge
  length function.  We first consider $\M_\gamma^1$.  Under the
  splitting \eqref{eq:bdryproduct} we define
  \begin{equation}
    \label{eq:01def}
    c_{\nu,e}(m,t):=(1-\tau_{\nu,e}(t))c_{\nu,e}(m,0) + \tau_{\nu,e}(t)
    c_{\nu,e}(m,\infty), \quad (m,t) \in \M_\gamma^{0,1} \times [0,\infty]
  \end{equation}
  where $\tau_{\nu,e}:[0,\infty] \to [0,1]$ is a diffeomorphism.  In
  \eqref{eq:01def}, $c_{\nu,e}$ satisfies the \hyperref[eq:taufactor]{(Marking independence)}
  property because on each connected component it
  factors through the projection $\M_\gamma^1 \to [0,\infty]$. Indeed
  $m \mapsto c_{\nu,e}(m,\infty)$ is locally constant on
  $\partial \M_\gamma \cap \M_\gamma^1$ by the \hyperref[eq:taufactor]{(Marking independence)}
  property for smaller strata, and
  $m \mapsto c_{\nu,e}(m,0)$ is locally constant on
  $\partial \M_\gamma \cap \M_\gamma^1$ by definition
  \eqref{eq:cgndef}.  We obtain uniform continuity on $c_{\nu,e}$ by
  requiring that the derivative of $\tau_{\nu,e}$ is bounded by
  $C_\nu^{-1}$ for all $\nu$. This can be arranged by allowing
  $d\tau_{\nu,e}$ to be supported in an interval of length $C_\nu$. In
  a similar way, the functions $c_{\nu,e_j}$ are extended to
  $\M_\gamma^i$, assuming that its value on the boundary
  $\partial \M_\gamma^i$ factors through the projection to
  $[0,\infty]^i$.
\end{proof}

\begin{proposition} {\rm(Convergence)}
  \label{prop:defconv}
  Let ${\etasp} \in \t$ be a generic cone direction, and let
  $\{\ul \eta_\nu\}_\nu$ be a coherent sequence of deformation data
  compatible with ${\etasp}$ (as in Definition \ref{def:compdef}).
  Let $u_\nu : C_\nu \to \XX_\cP$ be a sequence of
  $\ul \eta_\nu$-deformed maps with uniformly bounded area. Then, a
  subsequence converges to a framed split map $u_\infty:C \to \XX_\cP$
  of type $(\tGam_\infty,\Gamma)$. The limit map is unique up to the
  action of the identity component of the framed tropical symmetry
  group $T_{\trop,\fr}(\tGam_\infty,\Gamma)$ (see \eqref{eq:tfr}).  If
  the maps $u_\nu$ have index $0$, and the perturbation datum for the
  limit is regular, then, the split tropical graph $\tGam_\infty$ is
  rigid.
\end{proposition}

If the limit split tropical graph is rigid, the limit is unique,
because the framed symmetry group $T_{\trop,\fr}(\tGam_\infty,\Gamma)$
is finite. The following terminology is used in the proof:

\begin{definition}
  In a collection of sequences $\{t_\nu(i) \in \t\}_\nu$, $i=1,\dots,n$, we say $t(i)$ is the \em{fastest growing sequence} if for any $j \neq i$, $\lim_\nu \frac {|t_\nu(j)|}{|t_\nu(i)|}$ exists and is finite.
\end{definition}

\begin{proof}[Proof of Proposition \ref{prop:defconv}]  
  As in the convergence of broken maps, the first step is to find the
  component-wise limit map.  There are finitely many types of deformed
  maps that satisfy an area bound; the proof is the same as the
  analogous result Proposition T-\ref{T-prop:finno} for broken maps.
  Therefore, after passing to a subsequence, we can assume that the
  maps $u_\nu$ have a $\nu$-independent type $\tGam \to \Gamma$, with
  deformation datum
  \[ \eta_\nu(e):=\eta_{\tGam,\nu}([C_\nu],e) \in \t/\t_{\cT(e)},
    \quad \forall e \in \Edge_s(\Gamma).\]
  We apply Gromov convergence for broken maps (Theorem
  T-\ref{T-thm:cpt-broken}, outlined in Section \ref{subsec:conv})  on each
  connected component of $\tGam\bs \Edge_s(\Gamma)$.
  The result is a  translation sequence
  \[t_\nu(v) \in \Cone(\kappa,v), \quad v \in \Ver(\tGam_\infty)\]
  that satisfies the \hyperref[eq:slope-split]{(Direction)} condition for
  all non-split edges $e \in \Edge(\tGam_\infty) \bs \Edge_s(\Gamma)$,
  and a collection of limit maps denoted by $u_\infty$.  The
  limit $u_\infty$ is modelled on a quasi-split tropical graph
  $\tGam_\infty \to \Gamma$, equipped with a tropical edge collapse
  morphism  $\kappa : \tGam_\infty\bs \Edge_s(\Gam) \to \tGam \bs
  \Edge_s(\Gam)$.

  The limit map $u_\infty$ is a quasi-split map; for it to be a split
  map, it remains to show that the cone condition \eqref{eq:cc} is
  satisfied.  For this purpose, we refine the translation sequences
  $t_\nu$ to produce a sequence of relative
  translations satisfying the weak sequential cone condition of Lemma
  \ref{lem:spliteq}.
  For any 
  split edge $e$, the sequences
  \begin{equation}
    \label{eq:trseq}
    \pi^\perp_{\cT(e)}(t_\nu(v_+) - t_\nu(v_-))-\eta_\nu(e), \quad e=(v_+,v_-) \in \Edge_s(\Gamma) 
  \end{equation}
  are uniformly bounded: The proof is the same as the analogous
  statement for broken maps in T-\eqref{T-eq:tdiff}; it follows from
  the fact that the projected evaluations of $u_\nu$ at the nodal
  lifts $w_{\nu,e}^+$ and $w_{\nu,e}^-$ differ by $\eta_\nu$, and that
  the projected evaluation at $z_e^\pm$ translated by $t_\nu(v_\pm)$
  forms a converging sequence.  The next step is to adjust the
  translation sequences by a uniformly bounded amount so that the
  discrepancy at the split edges is equal to the deformation
  parameters, i.e.  the sequences in \eqref{eq:trseq}, vanish.  The
  adjustment to the translation sequences is by the iterative process
  used in Step 2 of the proof of Lemma T-\ref{T-lem:approxpoly}, and
  is as follows: The iteration is run simultaneously on the vertex
  sequences $\{t_\nu(v)\}_{v,\nu}$ and the edge sequences
  $\{\eta_\nu(e)\}_{e,\nu}$.  In each step a fastest growing sequence
  (which exists after passing to a subsequence) is subtracted, with
  the resulting sequence in the $i$-th step being denoted by
  $\{t^i_\nu(v)\}_v$, $\{\eta^i_\nu(e)\}_e$.  That is, if at the
  $i$-th step, $r_\nu^i$ is the fastest growing sequence, we define
  \[t_\nu^{i+1}(v):=t_\nu^i(v) - r_\nu^i \lim_\nu \tfrac {t_\nu^i(v)}{r_\nu^i},\]
  and $\eta_\nu^{i+1}(e)$ is defined similarly.  After any step, the
  quantity
  $\pi^\perp_{\cT(e)}(t_\nu^i(v_+)-t_\nu^i(v_-)) - \eta_\nu^i(e)$ is
  the same as the corresponding quantity in \eqref{eq:trseq} for any
  split edge $e=(v_+,v_-)$.  At any step, if both a vertex sequence
  and an edge sequence grow at the same rate, and are both fastest
  growing, then we break the tie by subtracting the edge sequence.
  The iteration ends when the sequences corresponding to vertices and
  split edges are uniformly bounded, say, after $k$ steps.  Our tie
  breaking method, together with the fact that the sequences
  corresponding to a pair of split edges grow at different rates 
  (since $(c_{e,\nu})_{e,\nu}$ is a sequence of increasing tuples),
  ensures that $\eta^k_\nu(e)=0$ for all split edges $e$ and all
  $\nu$.  We subtract the uniformly bounded final sequence
  $\{t^k_\nu(v)\}_{v,\nu}$ from the initial sequence
  $\{t_\nu(v)\}_{v,\nu}$ and denote the result by
  $\{\ol t_\nu(v)\}_{v,\nu}$, which satisfies
  \begin{equation}
    \label{eq:diffeta}
    \pi^\perp_{\cT(e)}(\ol t_\nu(v_+) - \ol t_\nu(v_-))-\eta_\nu(e)=0, \quad e=(v_+,v_-) \in \Edge_s(\Gamma).
  \end{equation}
  Since for any vertex $v$ the translation sequences $t_\nu(v)$,
  $\ol t_\nu(v)$ differ by a uniformly bounded amount, a subsequence
  of $e^{-\ol t_\nu}u_\nu$ converges to a limit $\ol u_{\infty,v}$.
  The limit map $\ol u_\infty:=(\ol u_{\infty,v})_v$ satisfies the
  matching condition on all edges, including split edges, because by
  \eqref{eq:diffeta} the translations $\ol t_\nu$ compensate for the
  deformation parameters. Finally $\tGam \to \Gam$ is a split graph: 
  By \eqref{eq:diffeta}, $\Diff(\ol t_\nu)=\eta_\nu$, and $\eta_\nu \in \Disc(\tGam,\Gam) 
    \quad\forall \nu$. Since the weak sequential condition is satisfied, Lemma \ref{lem:spliteq}           implies that $\Disc(\tGam,\Gam)$ satisfies the \hyperref[eq:cc]{(Cone condition)}. 

    We have shown that $\ol u_\infty$ is a framed split map. The proof
    of uniqueness of limits is the same as Step 6 in the proof of
    Theorem T-\ref{T-thm:cpt-breaking}.  The idea is that if there are
    different limits $u_\infty$, $u_\infty'$, the difference between
    the translation sequences is uniformly bounded, and the the limit
    of a subsequence of the differences yields an element in
    the identity component of the tropical symmetry group which relates
    $u_\infty$, $u_\infty'$. The last statement about the rigidity of
    the limit split map follows from a dimensional argument, see
    Remark \ref{rem:dim-split}.
\end{proof}

Gluing a rigid split map with a cone direction ${\etasp}$ yields a
sequence of deformed maps for any sequence of uniformly continuous
deformation data compatible with ${\etasp}$. We assume that the
perturbation datum for the deformed maps is obtained from that of the
split map by gluing on the necks.
\begin{remark}\label{rem:def-family}
  In Lemma \ref{lem:defseq}, we constructed a sequence of uniformly
  continuous deformation data $\{\ul \eta_\nu\}_{\nu \in \Z_+}$ that
  is compatible with a cone direction ${\etasp}$. By the same proof,
  we can also a construct a family of uniformly continuous deformation
  data $\{\ul \eta_\nu\}_{\nu \in \R_+}$ such that for any sequence
  $\nu_i \to \infty$, $\{\ul \eta_{\nu_i}\}_i$ is compatible with the
  cone direction ${\etasp}$.
\end{remark}
\begin{proposition} \label{prop:splitglue} {\rm(Gluing a split map)}
  Suppose $u$ is a regular rigid framed split map of type
  $(\tGam,\bGam)$ with a generic cone direction ${\etasp} \in \t$.
  Let $(\ul \eta_\nu)_{\nu \in \R_+}=\{\eta_{\tGam,\nu}\}_{\tGam,\nu}$
  be a family of uniformly continuous coherent deformation data
  compatible with ${\etasp}$ (as in Definition \ref{def:compdef}).
  Then, there exists $\nu_0$ such that for $\nu \geq \nu_0$
  \begin{enumerate}
  \item {\rm(Existence of glued family)} \label{part:splitglue-a}
    there is a regular rigid $\ul \eta_\nu$-deformed map $u_\nu$ such
    that the family $\{u_\nu\}_\nu$ of deformed maps converges to $u$
    as $\nu \to \infty$.
  \item {\rm(Surjectivity of gluing)} For any sequence of
    $\ul \eta_\nu$-deformed maps $u_\nu'$ that converges to the framed
    split map $u$, for large enough $\nu$ the map $u_\nu'$ is
    contained in the glued family constructed in
    \eqref{part:splitglue-a}.
  \end{enumerate}
\end{proposition}
\begin{proof}
  We first describe the type of the deformed maps in the glued family.
  Given that the type of the split map $u$ is
  $\kappa : \tGam \to \Gamma$, the type of the glued deformed maps is
  $(\tGam_d,\Gamma)$. That is, the base tropical graph $\Gamma$ stays
  the same, and $\tGam_d$ is obtained by tropically collapsing the
  edges $e \in \Edge_\trop(\tGam) \bs \Edge(\Gamma)$. Thus, the
  tropical edges in $\tGam_d$ are exactly the edges of the base
  tropical graph $\Gamma$, and the edge collapse map $\kappa$ factors
  as
  \[\tGam \xrightarrow{\kappa_0} \tGam_d \xrightarrow{\kappa_1}
    \bGam.\]
  The tree part of $\tGam_d$ is the same as that of $\tGam$. We will
  construct a family of deformed maps of type $\tGam_d \to \bGam$.

  The gluing construction is very similar to the gluing of broken maps
  in Theorem T-\ref{T-thm:gluing}, so we only point out the
  differences at each step of the proof.

  \vskip .1in \noindent \em{Step 1: Construction of an approximate
    solution:} An approximate solution is constructed using relative
  vertex positions corresponding to a deformation parameter.  The
  domain $C_\nu$ of the approximate solution is constructed by gluing
  some of the interior nodes in $C$, and therefore the edge lengths
  for treed segments $T_e$, $e \in \Edge_{\white,-}(\tGam)$ in $C_\nu$
  are the same as $C$. By the \hyperref[eq:taufactor]{(Marking independence)}
  property of
  deformation data, $\eta_{\nu,\tGam_d}(C)=\eta_{\nu,\tGam_d}(C_\nu)$,
  and we denote
  \[\eta_\nu:=\eta_{\nu,\tGam_d}(C) \in \t_\Gam,\]
  where we recall from \eqref{eq:taudef} that
  $\t_\Gam=\bigoplus_{e \in \Edge_s(\bGam)}\t/\t_{\cT(e)} $.  The
  sequence
  \[ (\eta_\nu(e))_{e \in \Edge_s(\tGam),\nu} \]
  is an increasing sequence of tuples (as in Definition
  \ref{def:inc-tup-e}). Therefore the \hyperref[eq:cc]{(Cone condition)} for the split
  tropical graph $\tGam$ implies that for large enough $\nu$,
  $\eta_\nu$ lies in the discrepancy cone of $\tGam$, and so, there is
  a relative translation $\cT_\nu \in \fw(\tGam,\Gamma)$ satisfying
  \begin{equation}
    \label{eq:t4eta}
    \pi_{\cT(e)}^\perp(\cT_\nu(v_+) - \cT_\nu(v_-))= \eta_\nu(e), \quad e=(v_+,v_-) \in
    \Edge_s(\Gam).  
  \end{equation}
  Since the split tropical graph $\tGam$ is rigid, the relative
  translation $\cT_\nu \in \fw(\tGam,\Gam)$ is uniquely determined by
  \eqref{eq:t4eta}. Indeed, if there is another solution $\cT'$ of
  \eqref{eq:t4eta}, the difference $\cT - \cT'$ satisfies the
  direction condition on all edges, including split edges, and
  therefore generates a non-trivial subgroup
  $\exp( ( \cT - \cT')(\cdot))$ in $T_{\trop,\fr}(\tGam)$
  contradicting the rigidity of $\tGam$. For future use, we point out
  that the inverse
  \begin{equation}
    \label{eq:diffinv}
    \Diff_\tGam^{-1}: \t_\Gam \to \fw^\pm(\tGam,\Gam) :=\R\bran{\fw(\tGam,\Gamma)},  
  \end{equation}
  which maps deformation parameters to elements in the $\R$-span of
  relative translations, is a well-defined linear map.  For any edge
  $e=(v_+,v_-)$ that is collapsed by $\tGam \to \tGam_d$, we can
  assign a length $l_\nu(e)>0$ satisfying
  \[\cT_\nu(v_+) - \cT_\nu(v_-)=l_\nu(e) \cT(e). \]

  We now describe the domain and target spaces for the approximate
  solution.  The domain is a nodal curve $C_\nu$ of type $\tGam_d$
  which is obtained from $C$ by replacing each node corresponding to
  an edge $e \in \Edge_\trop(\tGam) \bs \Edge_\trop(\tGam_d)$ by a
  neck of length $l_\nu(e)$.  A component of the glued curve
  corresponds to a vertex $v$ of $\tGam_d$. The map
  $u|C_{\kappa_0^{-1}(v)}$ is a broken map with relative marked
  points, whose pieces map to $\ol \XC_Q$, $Q \subseteq P(v)$. Our
  goal is to glue at the nodes of $u|C_{\kappa_0^{-1}(v)}$ to produce
  a curve lying in $\ol \XC_{P(v)}$.

  We extend the definition of the deformation parameter to all
  tropical edges of $\tGam_d$ by defining
  \[\eta_\nu(e)=0 \quad \forall e \in \Edge_\trop(\tGam_d) \bs
    \Edge_s(\bGam).\]
  Thus by definition a $\eta_\nu$-deformed map satisfies the edge
  matching condition for non-split edges of $\bGam$.

  Next, we construct the approximate solution which will be shown to be
  $\eta_\nu$-deformed. Since $u$ is a framed split map, it satisfies a
  matching condition on the split edges: For any edge $e=(v_+,v_-)$ in
  $\tGam$, and domain coordinates in the neighborhood of $w_e^\pm$
  that respect the framing, the tropical evaluation maps
  (as in \eqref{eq:nodematch-intro})
  satisfy
  \begin{equation}
    \label{eq:spm}
    \ev_{w^+_e}^{\cT(e)}(u_{v_+})=\ev_{w^-_e}^{\cT(e)}(u_{v_-}). 
  \end{equation}
  For a vertex $v$ in $\tGam_d$, and a vertex $v'$ in the collapsed
  graph $\kappa_0^{-1}(v)$ (so that $P(v) \subseteq P(v')$), the translation $\cT_\nu(v')$ gives an
  identification 
  \[(e^{-\cT_\nu(v')})^{-1} : \XX_{P(v')} \dashrightarrow \XX_{P(v)},\]
  given by the inverse map of \eqref{eq:pqtrans}. If
  $P(v') \neq P(v)$, the above inverse is well-defined on an open set
  of the domain that contains the image of $u_{v'}$, and maps it to the
  $P(v')$-cylindrical end of $\XX_{P(v)}$.  The translated map
  \begin{equation} \label{eq:uvtr}
    u_{v',\eta_\nu}:=e^{\cT_\nu(v')}u_{v'}:C_{v'}^\circ \to \XB_{\ol
      P(v)}
  \end{equation}
  is well-defined on the complement of nodal points on $C_{v'}$.  For
  any uncollapsed tropical edge $e=(v_+,v_-)$ in $\tGam \to \tGam_d$,
  the matching condition \eqref{eq:spm} and the translation in
  \eqref{eq:uvtr} together imply a deformed matching condition
  \begin{equation}\label{eq:tvdef}
    \ev_{w^+_e}^{\cT(e)}(u_{v_+,\eta_\nu})=e^{\eta_\nu(e)}\ev_{w^-_e}^{\cT(e)}(u_{v_-,\eta_\nu}),    
  \end{equation} 
  because $\eta_\nu(e)=\cT_\nu(v_+)-\cT_\nu(v_-)$.  As in the proof of
  Theorem T-\ref{T-thm:gluing}, for each vertex $v$, translated maps
  $(u_{v',\eta_\nu})_{v' \in \kappa_0^{-1}(v)}$ can be glued at the
  nodal points to yield an approximate solution for the holomorphic
  curve equation which is denoted by
  \[u^{\pre}_\nu=(u^{\pre}_{v,\eta_\nu})_{v \in \Ver(\tGam_d)}, \quad
    u^{\pre}_{v,\eta_\nu}:C_{\eta_\nu,v} \to \ol \XC_{P(v)}.\]
  Since the patching does not alter the maps $u_{v,\eta_\nu}$ away
  from the collapsed nodes of $\tGam \to \tGam_d$, we conclude by
  \eqref{eq:tvdef} that the pre-glued map $u_\nu^{\pre}$ is
  $\eta_\nu$-deformed.

  \vskip .1in \noindent \em{Step 2: Fredholm theory:} The Sobolev
  norms carry over entirely from the proof of Theorem
  T-\ref{T-thm:gluing}. On the curves $C_\nu$, we use Sobolev weights
  on both neck regions created by gluing, and on nodal points
  corresponding to interior edges in $\tGam_d$. The map $\F_\nu$ in
  Theorem T-\ref{T-thm:gluing} incorporated the holomorphicity
  condition, marked points mapping to divisors and matching at disk
  nodes. Now we additionally require a deformed matching condition on
  interior nodes, given by
  \[\ev_e(u) \in \Delta_e \subset (\ol \XC_{P(e)})^2, \quad e
    \in \Edge_{\trop}(\tGam_d),\]
  where $\Delta_e$ is the diagonal and $\ev_e$ is the tropical
  evaluation map on the nodal lifts twisted by the deformation
  parameter, that is,
  \begin{multline*}
    \ev_e:\tM_{\deform,\tGam_d} \times \Map(C_\nu,\XX)_{1,p,\lam} \to (\XC_{P(e)})^2,\\
    (m,u) \mapsto (\ev^{\cT(e)}_{w_+(e)}(u_{v_+})),
    \exp(\eta_{\nu,\tGam_d}(m,e))(\ev^{\cT(e)}_{w_-(e)}(u_{v_-})).
  \end{multline*}
  Incorporating these conditions we obtain a map
  \[\F_\nu: \tM_{\deform, \tGam_d}^i \times \Omega^0(C_\nu,
    (u_\nu^{\pre})^* T\XX)_{1,p} \to \Omega^{0,1}(C_\nu,
    (u_\nu^{\pre})^* T\XX)_{0,p} \oplus \ev_\tGam^* T
    \XX(\tGam_d)/\Delta(\tGam_d)\]
  whose zeros correspond to $\eta_\nu$-deformed pseudoholomorphic maps
  near the approximate solution $u_\nu^{\pre}$, and where the
  notations $\XX(\tGam_d)$ and $\Delta(\tGam_d)$ are defined in
  \eqref{T-xgam}.

  \vskip .1in \noindent \em{Step 3: Error estimate:} The error
  estimate for the approximate solution is produced in the same way as
  Theorem T-\ref{T-thm:gluing}. Indeed the only contribution to the
  error estimate is from the failure of holomorphicity, as the
  approximate solution satisfies the matching conditions at boundary
  and interior nodes.

  The next few steps of the proof, namely, the construction of a
  uniformly bounded right inverse, the proof of quadratic estimates,
  and the application of Picard's Lemma are the same as in Theorem
  T-\ref{T-thm:gluing}.

  \vskip .1in \noindent \em{Step 7: Surjectivity of gluing}: Compared
  to the corresponding step in Theorem T-\ref{T-thm:gluing}, in the
  gluing of a split map, one additionally has to deal with
  domain-dependent deformation parameters.  Consider a sequence
  $u_\nu':C_\nu' \to \XX$ of $\ul \eta_\nu$-deformed maps that
  converges to the split map $u$. To prove that the maps $u_\nu'$ lie
  in the image of the gluing map of $u$, it is enough to show that
  $u_\nu'$ is close enough to the pre-glued map
  $u_\nu^{pre}:C_\nu \to \XX$. Indeed, similar to the proof of Theorem
  T-\ref{T-thm:gluing}, one of the conclusions of Picard's Lemma is
  that the glued map $u_\nu$ is the unique $\ul \eta_\nu$-deformed map
  in an $\eps$-neighborhood of the pre-glued map $u_\nu^{pre}$, and
  $\eps$ is independent of $\nu$. The sequences of domains
  $(C_\nu)_\nu$ and $(C_\nu')_\nu$ converge to $C$ in the compactified
  moduli space $\ol \tM_{\deform,\tGam_d}$, and so, the edge lengths
  $\ell(C_\nu,T_e)$, $\ell(C_\nu',T_e)$ of the treed segments converge
  to $\ell(C,T_e)$ for all boundary edges
  $e \in \Edge_{\white,-}(\tGam_d)$.  The edge length $\ell(C,T_e)$ in
  the split map $u$ is finite by the rigidity of $u$.  Therefore the
  differences converge
  \[\abs{\ell(C_\nu',T_e)-\ell(C_\nu,T_e)} \to 0.\]
  The deformation parameters for $u_\nu'$ are
  $\eta_\nu':=\eta_{\tGam,\nu}(C'_\nu)$. On a fixed stratum $\tGam_d$
  of curves, the deformation parameter depends only on the lengths of
  the treed edges by the \hyperref[eq:taufactor]{(Marking independence)} property.
  Together with the
  uniform continuity of the deformation datum, we conclude
  $\abs{\eta_\nu-\eta_\nu'} \to 0$.
  By the proof of convergence of deformed maps, the translation
  sequence $\{t_\nu'(v)\}_{\nu,v}$ for the convergence of
  $\{u_\nu'\}_\nu$ can be chosen to satisfy
  \[t_\nu(v_+)-t_\nu(v_-)=\eta_\nu'(e), \quad \forall e=(v_+,v_-), \]
  see \eqref{eq:diffeta}.  The inverse $\Diff_\tGam^{-1}$ is a
  well-defined map that maps
  \[\eta_\nu \mapsto \cT_\nu, \quad \eta_\nu' \mapsto t_\nu.\]
  Hence $\cT_\nu - t_\nu \to 0$. as in Theorem T-\ref{T-thm:gluing}
  the closeness of the relative translations $\cT_\nu$, $t_\nu$
  implies that the difference between the gluing parameters of
  $C_\nu$, $C_\nu'$ (that is, the gluing parameters used to construct
  $C_\nu$, $C_\nu'$ from $C$) converge to $0$. The rest of the proof
  is the same as the proof of Theorem T-\ref{T-thm:gluing}.
\end{proof}

 The following is an analogue of Proposition T-\ref{T-prop:breakingpert} and the proof is analogous.
\begin{proposition}\label{prop:defmorph}
  Let ${\etasp} \in \t$ be a generic
  cone direction and let $\{\ul \eta_\nu : \nu \in \R_{\geq 0}\}$ be a family of
  uniformly continuous coherent deformation data compatible with
  ${\etasp}$. Let $\Pe^\infty$ be a regular perturbation datum for split maps on
  $\XX$.
 For any
 $E_0>0$, there exists $\nu_0(E_0)$ such that the following holds.
   \begin{enumerate}
   \item \label{part:bij-split}
     {\rm(Bijection of moduli spaces)}
     For any $\nu \in \Z_{>0}$, there exist regular perturbation data $\ul \Pe^\nu$
     $\ul \eta_\nu$-deformed maps such that $\nu \geq \nu_0$
  there is a bijection between the moduli space of rigid framed split
  disks $\tM^\spl_{\fr}(\ul \Pe^\infty, {\etasp})_{<E_0}$ and the moduli space
  $\tM_{\deform}(\ul \Pe,\ul \eta_\nu)_{<E_0}$ of rigid
  $\ul \eta_\nu$-deformed disks of area less than $E_0$.
\item \label{prop:defmorph2}
   For any $\nu \in \Z_{>0}$, there exists a regular 
  perturbation morphism $\Pe^{\nu,\nu+1}$ extending $\Pe^\nu$ and
  $\Pe^{\nu+1}$ such that for all $E_0>0$ and $\nu \geq \nu_0(E_0)$,
  the \ainfty morphisms
  \[\phi_\nu^{\nu+1}: CF_\deform(L,\ul \Pe^\nu,\ul \eta_\nu) \to
    CF_\deform(L,\ul \Pe^{\nu+1},\ul \eta_{\nu+1}),\]
  \[\psi_{\nu+1}^\nu: CF_\deform(L,\ul
    \Pe^{\nu+1},\ul \eta_{\nu+1}) \to CF_\deform(L,\ul \Pe^\nu,\ul
    \eta_\nu)\]
 induced by $\Pe^{\nu,\nu+1}$
 are 
 identity modulo $q^{E_0}$, and are homotopy equivalences. 
 Furthermore, the homotopy equivalence between 
$\phi_\nu^{\nu+1} \circ \psi_{\nu+1}^\nu$ and the identity,
and 
$ \psi_{\nu+1}^\nu \circ \phi_\nu^{\nu+1}$ and the identity
may be taken to be zero modulo $q^{E_0}$.
\end{enumerate}
\end{proposition}
\begin{proof}
  The proof is analogous to that of 
  Proposition T-\ref{T-prop:breakingpert}, with 
  the parameter $\nu$ playing the role of the neck length parameter. 
  
  Proposition \ref{prop:defmorph} \eqref{part:bij-split} is a
  consequence of convergence and gluing results for split maps, namely
  Proposition \ref{prop:defconv} and Proposition \ref{prop:splitglue}.
  For any integer $\nu$ for strata $\Gamma$ of low area, that is, the
  strata for which the bijection from the gluing result (Proposition
  \ref{prop:splitglue}) is applicable, $\Pe^\nu_\Gamma$ is the
  perturbation obtained from $\Pe^\infty_\Gamma$ by gluing on the
  neck. (In fact, note that such a perturbation is assumed in the
  hypothesis of the gluing result.)  The perturbation datum $\Pe^\nu$
  is extended to higher area strata by the transversality result for
  deformed maps (Proposition \ref{prop:pdtrans}).

  To define \ainfty morphisms, the 
  family of deformation data interpolating between $\ul \eta_\nu$ and
  $\ul \eta_{\nu+1}$ is taken to be the restriction of the family
  $\{\ul \eta_\mu : \mu \in \R_+\}$  to $\mu \in [\nu,\nu+1]$.  The perturbation morphism
  $\Pe^{\nu,\nu+1}$ is defined by gluing $\Pe^\infty$ on the neck for low
  area strata, with neck-length parameter $\nu + \delta \circ d_\Gamma(z)$ given by the distance to the quilting circle \eqref{eq:dist-seam}, 
  and extended in a regular coherent way to strata of
  higher area.  Here, `low area' is in a similar sense as 
  the previous paragraph, 
  and means that the low area strata of rigid
  $\ul \Pe^{\tau}$-holomorphic $\ul \eta_\nu$-deformed maps are bijective
  to corresponding strata of rigid framed split disks for any $\tau \in [\nu,\nu+1]$. 
  Following the arguments in the proof of Proposition
  T-\ref{T-prop:breakingpert}, we conclude that there are no
  non-constant $\Pe^{\nu,\nu+1}$-holomorphic deformed quilted disks
  with area at most $E_0$, and therefore,
  \[(\phi_\nu^{\nu+1})^1=\Id \mod q^{E_0}, \quad (\phi_\nu^{\nu+1})^k =0 \mod q^{E_0},
    \enspace k \neq 1. \] Here the identity term in $(\phi_\nu^{\nu+1})^1$ 
  counts constant quilted maps with a single input. The idea of the
  proof is that if there are $\Pe^{\nu,\nu+1}$-holomorphic quilted
  disks with area $\leq E_0$ for arbitrarily large $\nu$, then a
  sequence of such maps converge to a $\ul \Pe_\infty$-holomorphic
  framed split map; the limit map is unquilted and hence has index
  $-1$, which contradicts the regularity of $\ul \Pe_\infty$.
  Perturbations for twice-quilted are defined in the same way, that is, on low area strata the perturbation is  by gluing $\ul \Pe_\infty$ with neck length given by a version of the distance to the quilting circle  parameter from \cite[Definition
5.9]{cw:flips}. Similar arguments as in the quilted case show that the homotopies $\phi_\nu^{\nu+1} \circ \psi_{\nu+1}^\nu$ and the identity,
and 
$ \psi_{\nu+1}^\nu \circ \phi_\nu^{\nu+1}$ and the identity 
are zero modulo $q^{E_0}$.
\end{proof}

The following is a consequence of Proposition \ref{prop:defmorph} \eqref{part:bij-split}:

\begin{corollary}
  Let $CF_\deform(L,\ul \Pe^\nu,\ul \eta_\nu)$ and
  $CF_{\split}(L,{\etasp},\ul \Pe^\infty)$ be the deformed and split
  Fukaya algebras defined by the perturbation data
  $(\ul \Pe^\nu,\ul \eta_\nu)$, $(\ul \Pe^\infty ,\etasp)$ from
  Proposition \ref{prop:defmorph}.  The structure maps
  $(m^\nu_\deform)^d$ defining $CF_\deform(L,\ul \Pe^\nu,\ul \eta_\nu)$
  converge as $\nu \to \infty$ to the structure maps $m^d_\spl$
  defining $CF_{\split}(L,{\etasp},\ul \Pe^\infty)$, in the sense that
  for any $E_0$ there exists $\nu(E_0)$ so that $(m^\nu_\deform)^d$ agrees with
  $m^d_\spl$ up to terms divisible by $q^{E_0}$ for $\nu > \nu(E_0)$.
\end{corollary}

The following result completes the proof of \ainfty homotopy
equivalence between deformed and split Fukaya algebras. Together, with
the fact that any two deformed Fukaya algebras are homotopy equivalent
(Proposition \ref{prop:pd3}), this finishes the proof of the main
result (Theorem \ref{thm:tfuk}) that broken Fukaya algebras are
\ainfty homotopy equivalent to split Fukaya algebras.

\begin{proposition} \label{prop:hequiv} {\rm(Homotopy equivalence)}
  Let ${\etasp}$ be a generic cone direction, and let $\ul \eta_\nu$
  be a compatible sequence of coherent deformation data that is
  uniformly continuous.  For any $\nu_0 \in \Z_+$, let
  $\ul \Pe_{\nu_0}$ be the perturbation defined in Proposition
  \ref{prop:defmorph} for $\ul \eta_\nu$-deformed maps.  Then, there
  are strictly unital convergent \ainfty morphisms
  \[\phi:CF_\deform(L, \ul \eta_{\nu_0}, \ul \Pe^{\nu_0}) \to
    CF_{\split}(L,{\etasp},\ul \Pe^\infty), \quad
    \psi:CF_{\split}(L,{\etasp},\ul \Pe^\infty) \to CF_\deform(L, \ul
    \eta_{\nu_0}, \ul \Pe_{\nu_0}). \]
  such that $\psi \circ \phi$ and $\phi \circ \psi$ are
  \ainfty-homotopy equivalent to identity.
\end{proposition}
\begin{proof}
  The proof is verbatim the same as the proof of Proposition
  T-\ref{T-prop:unbreak-break}.  In particular, we define
  $\phi:=\lim_n \phi_n$, where $\phi_n$ is the composition
  \[ \phi_{n} := \phi_{\nu_0+n-1}^{\nu_0 + n} \circ \ldots \circ
    \phi_{\nu_0 + 1}^{\nu_0+2} \circ \phi_{\nu_0}^{\nu_0+1} :
    CF_\deform(L, \ul \eta_{\nu_0}, \ul{\Pe}^{\nu_0}) \to
    CF_\deform(L, \ul \eta_{\nu_0+n},\ul{\Pe}^{\nu_0+n }) \]
  of morphisms $\phi_\nu^{\nu+1}$ in Proposition \ref{prop:defmorph};
  and $\psi$ is defined similarly. Homotopies are also constructed as
  limits of homotopies $h_n$ from $\phi_n \circ \psi_n$ to $\Id$, and
  $g_n$ from $\psi_n \circ \phi_n$ to $\Id$. By Proposition
  \ref{prop:defmorph}, the homotopies stabilize and have limits $h$,
  $g$ that satisfy
  \[\phi \circ \psi - \Id = m_1(h),
    \quad \psi \circ \phi - \Id = m_1(g). \]
\end{proof}

\section{Unobstructedness and disk potentials}
\label{sec:apps}

In this section we apply the theory in the previous sections to show
that tropical tori, such as torus fibers in almost toric fibrations,
are weakly unobstructed, and compute the potentials of the split 
Fukaya algebras in some cases. 

\subsection{Unobstructedness  for tropical tori}
\label{sec:unobst} 
The unobstructedness result is for tropical moment fibers which we now define. 
\begin{definition}\label{def:trop-torus}
  Let $\XX$ be a broken manifold corresponding to a tropical
  Hamiltonian action $(X,\PP,\Phi)$.  The Lagrangian $L$ is called a
  \em{tropical moment fiber} if
  \begin{itemize}
  \item $\ol X_{P_0}$ is a toric manifold, and the restriction of the
    tropical moment map $\Phi:\ol X_{P_0} \to \t^\dual$ is an honest
    moment map,
  \item any torus-invariant divisor of $\ol X_{P_0}$ is a relative
    divisor, that is, the divisor is $X_Q$ for some polytope
    $Q \in \PP$ satisfying $Q \subset P_0$, and
  \item $L = \Phi^{-1}(\lambda)$ for a point $\lam$ in the interior of
    the polytope $P_0$,
  \end{itemize}
\end{definition}

\begin{example} Any fiber of an almost toric manifold is a tropical
  moment fiber. We recall from Symington \cite{sym:2to4} that a four-dimensional 
  \em{almost toric manifold} is a symplectic manifold with a
  Lagrangian fibration $\pi : X \to B$ with compact fibers, whose
  regular fibers are tori, such that any critical point $p$ of $\pi$
  is either a toric singularity or is an isolated \em{nodal singularity}. In the latter case, $p$ 
  has a Darboux neighborhood (with symplectic form $dx \wedge dy$) in which
  the projection $\pi$ is  
  \[\pi(x,y)=(x_1 y_1 + x_2 y_2 , x_1 y_2 - x_2 y_1).\]

  Given a Lagrangian torus $L \subset X$ that is a smooth fiber of
  $\Phi : X \to B$, there are many ways to construct polyhedral
  decompositions in which $L$ is a tropical moment fiber.  One such
  scheme is to construct a polyhedral decomposition $\PP$ that is a
  collection of single cuts intersecting each other orthogonally:
  Assuming $B \subset \t^\dual \simeq \R^2$, consider a
  rational basis $\{e_1,e_2\}$ of $\t$. For each $e_i$,
  consider a pair of parallel cuts $P_i$, $P_i'$ normal to $e_i$, such
  that the Lagrangian fiber $\Phi(L) \in B$ lies between $P_i$ and
  $P_i'$. See Figure \ref{fig:poly} for example.
  In the resulting decomposition of $B$, $L$ is a tropical
  fiber.  Such a decomposition $\PP$ has a cutting datum, as discussed
  in Example T-\ref{T-eg:singles}. 
\end{example}

\begin{figure}[ht]\begin{center} 
    \scalebox{.6}{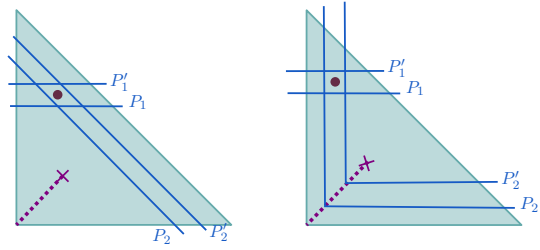} 
  \end{center}
  \caption{Polyhedral decompositions of an almost toric diagram}
  \label{fig:poly}
\end{figure}

We restate and prove Corollary \ref{cor:mc} stated in the
Introduction.  Weak unobstructedness of a Lagrangian brane is defined
in the Introduction.

\begin{proposition} \label{prop:mc-restate} {\rm(Unobstructedness of a
    tropical torus)} Suppose that for a  polyhedral decomposition $\PP$, 
  $L \subset X_{P_0} \subset \XX$ is a
  tropical torus as in Definition \ref{def:trop-torus}, and suppose
  $L$ is equipped with a brane structure.  For any generic cone
  direction ${\etasp} \in \t$, we have
  \[ m^0_{CF_{\split}(\XX,L,{\etasp})}=w x^{\blackt} \]
  where $x^{\blackt}$ is the unique maximum point of the Morse
  function on $L$ and $w\in \Lambda_{ > 0}$ is an element of the
  positive part of the Novikov ring.
    %
%
  In the split Fukaya algebra $CF_{\split}(\XX,L,{\etasp})$,
  \[ b:=w x^{\greyt} \in MC(L) \]
  is a solution of the Maurer-Cartan equation, and the potential of
  the $b$-deformed $A_\infty$ algebra $CF_{\split}(\XX,L,{\etasp}, b)$
  is
  \[W(b):=w.\]
  Consequently $CF_{\split}(\XX,L,{\etasp})$, and hence $CF(X,L)$, are
  weakly unobstructed.
\end{proposition}

\begin{proof} We prove unobstructedness for the split Fukaya algebra
  $CF_{\split}(\XX,L)$. The split Fukaya algebra is homotopy
  equivalent to the unbroken Fukaya algebra $CF(X,L)$ by Theorem
  T-\ref{T-thm:bfuk}.  Unobstructedness is preserved under homotopy
  equivalence, see \cite[Lemma 5.2]{cw:flips}.

  We first describe the perturbation datum for the moduli space of
  split maps.  Let $J_0$ be the standard almost complex structure on
  $\ol \XC_Q$ for all $Q \subseteq P_0$. Since torus-invariant
  divisors are relative divisors, Proposition
  T-\ref{T-prop:toricfibreg} implies that $J_0$-holomorphic disks and
  spheres in $\ol X_{P_0}$ are regular. Therefore, there is a coherent
  regular perturbation datum $\ul \Pe$ on $\XX$ for which the almost
  complex structure on $\ol X_{P_0}$ is standard.
  \label{page:unobs-proof}
  Next, we show that $m^0(1)$ is a multiple of the geometric unit
  $x^{\blackt}$. Consider a split disk $[u]$ of type $\tGam \to \bGam$
  contributing to $m_0(1)$, and whose boundary output asymptotes to
  $x_0 \in \cI(L)$.  Let
  \[\tGam_1 \subset \tGam \bs \Edge_s(\bGam)\]
  be the connected component containing the disk components, and let
  $u_1:=u|_{\tGam_1}$.  Let $u'$ be the split disk obtained by
  forgetting the boundary output leaf, and let $\tGam' \to \bGam$ be
  the type of $u'$.  For any torus element $t \in T$, the split map
  \[(u'_t)_v:=
    \begin{cases}
      t u'_v, \quad v \in \Ver(\tGam_1),\\
      u_v, \quad v \in \Ver(\tGam) \bs \Ver(\tGam_1).
    \end{cases}
  \]
  is not contained in the $T_\trop(\tGam,\bGam)$-orbit of $u'$.
  Indeed,
  the tropical symmetry group has a trivial action on $\ol X_{P_0}$,
  which contains the disk components.  By the same reason, for any
  pair $t_1 \neq t_2 \in T$, the maps $u_{t_1}'$, $u_{t_2}'$ lie in
  distinct $T_\trop(\tGam',\bGam)$-orbits. The regularity of $u$
  implies $u'_t$ is regular for all $t \in T$, Thus the tropical
  isomorphism classes $\{[u_t']:t \in T\}$ form a
  $\dim(T)$-dimensional family in the reduced moduli space, and 
  therefore
  \[\dim(\M^\spl_{\tGam',\red}(L)) \geq \dim(T).\]
  Since for the original split map $u$, the dimension of the reduced
  moduli space of split maps $\M^\spl_{\tGam,\red}(L)$ is zero, the
  output is necessarily the geometric unit $x^{\blackt}$. So,
  $m_0(1) = W x^{\blackt}$ for some 
  $W \in \Lambda_{> 0}$.

  The existence of a solution to the projective Maurer-Cartan equation
  now follows: We first claim that $m^1_\spl(x^{\greyt})$ only has zero
  order terms. If not, let $u$ be a split map with non-zero area
  contributing to $m^1_\spl(x^{\greyt})$. By the locality axiom, and the
  standardness of the perturbation datum on $\ol X_{P_0}$, we conclude
  that forgetting the input in $u$ produces a regular split disk $u'$
  with no inputs.  Since the index of $u'$ can not be negative, we
  conclude $u$ can not exist for dimension reasons. Indeed, the index
  of $u$ is two more than the index of $u'$ : one from the choice of a
  boundary incoming marking, and one from the weight on the incoming
  leaf. Therefore,
  \[ m^1_\spl(x^{\greyt}) = x^{\whitet} - x^{\blackt} . \]
  By a similar argument,
  \[m^d_\spl(x^{\greyt},\dots,x^{\greyt})=0 \quad d \geq 2,\]
  and consequently, $Wx^{\greyt}$ is a solution of the projective
  Maurer-Cartan equation.
\end{proof}

 \subsection{Disk potentials for torics}\label{sec:diskpot}
 In this section, we show that a toric Lagrangian brane in a toric
 symplectic manifold is unobstructed, 
 reproving a result of Fukaya-Oh-Ohta-Ono \cite{fooo:toric1}.  We also
 show that there is a solution of the
 Maurer-Cartan equation for which the leading order terms of the
 potential are the same as the leading order terms in the
 Batyrev-Givental potential (Definition \ref{def:bg}).

We consider a multiple cut 
on the toric manifold 
consisting of a
 collection of orthogonally intersecting single cuts, with one cut
 parallel to each facet of the moment polytope.  
 Let $X$ be 
 a $T$-toric variety with moment map $\Phi: X \to \t^\dual$ and
 moment polytope
 \begin{equation}
   \label{eq:XPhi}
   \Delta:=\Phi(X)=\{x \in \t^\dual: \bran{\mu_i,x} \leq c_i,
   i=1,\dots,N\}, 
 \end{equation}
 where $\mu_i \in \t$ is the primitive outward pointing normal of the
 $i$-th facet of $\Delta$, and $c_i \in \R$.  The polyhedral
 decomposition $\PP$ is defined by single cuts along the hypersurfaces
 \begin{equation}
   \label{eq:cutxx}
   \bran{x,\mu_i} = c_i -\eps_i, \quad i=1,\dots,N,
 \end{equation}
 where $\eps_i>0$ is a small constant. See Figure \ref{fig:toric-cut}.
 Thus one of the cut spaces
 \begin{equation*}
   X_{P_0}:=\Phinv(\{x : \bran{\mu_i,x} \leq c_i -\eps_i,
   i=1,\dots,N\})/\sim, 
 \end{equation*}
 where $\sim$ quotients boundaries by $S^1$-actions, is diffeomorphic
 to $X$, and for any $i$
 \[\{x \in \Phi(X) : \bran{\mu_i,x} \geq c_i -\eps_i\}/\sim\]
 is a $\P^1$-fibration.  For a given Lagrangian moment fiber
 $L=\Phinv(\lam)$, $\lam \in \on{interior}(\Phi(X))$, we assume the
 constants $\eps_i$ are small enough that the Lagrangian torus lies in
 $X_{P_0}$. That is, $\bran{\lam,\mu_i} \leq c_i-\eps_i$ for all
 $i$. Torus-invariant divisors of $\ol X_{P_0}$ are quotients of
 $\Phinv(\{x | \bran{x,\mu_i} = c_i -\eps_i \})$, and these are
 relative divisors; therefore $L$ is a tropical torus. Therefore, by
 Proposition \ref{prop:mc-restate}, the Lagrangian $\Phinv(\lam)$ is
 weakly unobstructed for any brane structure, which is the content of Proposition
 \ref{prop:blas} \eqref{part:blas1}. 
 Proposition \ref{prop:mc-restate} also gives a distinguished solution
 $b$ of the Maurer-Cartan equation. In Proposition
 \ref{prop:blas} \eqref{part:blas2},  we show that the leading order terms of the
 potential $W(b)$ agree with the leading order terms of the Batyrev-Givental potential, which we describe next.  The ``leading order
 term'' of a polynomial $W \in \Lam_{>0}$ refers to the non-zero terms
 with the least exponent of $q$.

 \begin{figure}[ht]
   \centering \scalebox{.8}{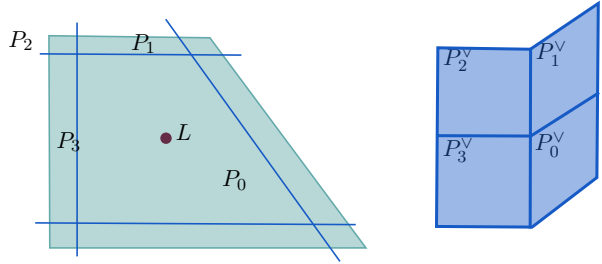}
   \caption{The multiple cut and its dual complex on a symplectic
     toric manifold.}
   \label{fig:toric-cut}
 \end{figure}

 The terms of the Batyrev-Givental potential (\cite{bat:qcr},
 \cite{giv:hom}) of a toric Lagrangian in a toric manifold correspond
 to index two Maslov disks that are projections of Blaschke disks in
 $\C^N$.  Here, we view $X$ as a git quotient $\C^N/G$ for a complex
 torus $G$, and use the classification of disks in $(X,L)$ given by
 Cho-Oh in \cite[Theorem 5.3]{chooh:toric}.  By this classification, there is a
 unique Blaschke disk of Maslov index two that intersects a given
 torus-invariant divisor, and which satisfies a point constraint on
 the boundary.
 
 \begin{definition} {\rm(Batyrev-Givental potential)}
   \label{def:bg}
   Suppose $X$, $\Phi$ are as in \eqref{eq:XPhi} above.  For any
   $\lam$ in the interior of $\Phi(X)$, the \em{Batyrev-Givental
     potential} for the Lagrangian $L:=\Phinv(\lam)$ is given by
   \begin{equation} \label{eq:bgpot}
     W_{BG}(y_1,\dots,y_n)=\sum_{i=1}^N \left(\prod_j y_j^{\mu_{i,j}}
     \right) q^{c_i-\bran{\lam,\mu_i}},\end{equation}
   where $\mu_i=(\mu_{i,1},\dots,\mu_{i,n})$ is the primitive outward
   normal of the $i$-th facet of the moment polytope.
 \end{definition}


We reprove the following result of Fukaya-Oh-Ohta-Ono \cite{fooo:toric1}:
\begin{proposition}
  {\rm(Unobstructedness and disk potential)}
  \label{prop:blas}
   Suppose $X$ is a compact symplectic toric manifold with
   an action of a torus $T$,
   and moment map
   $\Phi:X \to \t^\dual$. Let $L \subset X$ be a 
   Lagrangian moment fiber with a brane structure.
   \begin{enumerate}
   \item \label{part:blas1} The Lagrangian brane $L \subset X$ is weakly unobstructed.
   \item \label{part:blas2}
     There is a Maurer-Cartan element $b \in MC(CF(X,L))$ for which the leading order terms of the
     potential $W(b)$ coincide with the
  leading order terms of the Batyrev-Givental potential.
   \end{enumerate}
\end{proposition}
\begin{remark}\label{rem:higher-order-cs}
  The tropical disk potential $W(b)$ in Proposition \ref{prop:blas}
  \eqref{part:blas2} may contain higher order terms that are not
  present in the Batyrev-Givental potential. For example, in the
  second Hirzebruch surface (see Figure T-\ref{T-fig:h2-disks}), the
  Batyrev-Givental potential has $4$ terms, corresponding to each of
  the toric divisors. The tropical potential $W(b)$ has 5 terms; the
  additional term is the sum of a disk class and a $(-2)$-sphere
  class.
\end{remark}
\begin{figure}[t]
  \centering 
\begingroup%
  \makeatletter%
  \providecommand\color[2][]{%
    \errmessage{(Inkscape) Color is used for the text in Inkscape, but the package 'color.sty' is not loaded}%
    \renewcommand\color[2][]{}%
  }%
  \providecommand\transparent[1]{%
    \errmessage{(Inkscape) Transparency is used (non-zero) for the text in Inkscape, but the package 'transparent.sty' is not loaded}%
    \renewcommand\transparent[1]{}%
  }%
  \providecommand\rotatebox[2]{#2}%
  \newcommand*\fsize{\dimexpr\f@size pt\relax}%
  \newcommand*\lineheight[1]{\fontsize{\fsize}{#1\fsize}\selectfont}%
  \ifx\svgwidth\undefined%
    \setlength{\unitlength}{355.3282993bp}%
    \ifx\svgscale\undefined%
      \relax%
    \else%
      \setlength{\unitlength}{\unitlength * \real{\svgscale}}%
    \fi%
  \else%
    \setlength{\unitlength}{\svgwidth}%
  \fi%
  \global\let\svgwidth\undefined%
  \global\let\svgscale\undefined%
  \makeatother%
  \begin{picture}(1,0.28575578)%
    \lineheight{1}%
    \setlength\tabcolsep{0pt}%
    \put(0,0){\includegraphics[width=\unitlength,page=1]{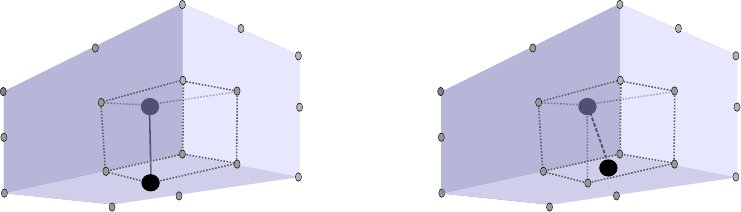}}%
    \put(0.16119112,0.12437138){\color[rgb]{0,0,0}\makebox(0,0)[lt]{\lineheight{0}\smash{\begin{tabular}[t]{l}$u_+$\end{tabular}}}}%
    \put(0.79559617,0.15999017){\color[rgb]{0,0,0}\makebox(0,0)[lt]{\lineheight{0}\smash{\begin{tabular}[t]{l}$u_+$\end{tabular}}}}%
    \put(0.16119112,0.02486589){\color[rgb]{0,0,0}\makebox(0,0)[lt]{\lineheight{0}\smash{\begin{tabular}[t]{l}$u_-$\end{tabular}}}}%
    \put(0.83630298,0.04239273){\color[rgb]{0,0,0}\makebox(0,0)[lt]{\lineheight{0}\smash{\begin{tabular}[t]{l}$u_-$\end{tabular}}}}%
    \put(0.81438367,0.10370544){\color[rgb]{0,0,0}\makebox(0,0)[lt]{\lineheight{0}\smash{\begin{tabular}[t]{l}$e$\end{tabular}}}}%
    \put(0.20848418,0.09076631){\color[rgb]{0,0,0}\makebox(0,0)[lt]{\lineheight{0}\smash{\begin{tabular}[t]{l}$e$\end{tabular}}}}%
    \put(0,0){\includegraphics[width=\unitlength,page=2]{blaschke.pdf}}%
    \put(0.44915419,0.12437138){\color[rgb]{0,0,0}\makebox(0,0)[lt]{\lineheight{0}\smash{\begin{tabular}[t]{l}Homotopic\end{tabular}}}}%
  \end{picture}%
\endgroup%

  \caption{Proposition \ref{prop:blas}: 
    The tropical graph of the
    broken and split map in the dual complex $B^\dual$ of $\PP$, when the moment polytope
    $\Phi(X)$ is a cube. 
    The edge $e$ is a split
    edge.}
  \label{fig:blas}
\end{figure}
\begin{proof}[Proof of Proposition \ref{prop:blas}]
  The terms in the Batyrev-Givental potential correspond to disks of
  Maslov index $2$ that have an intersection of $1$ with one of the
  toric divisors, and do not intersect the other toric divisors. We
  will show that the potential  consists
  of a term corresponding to each of these disks.

  It is enough to work with the potential arising from a split Fukaya algebra as we explain:
  Assuming the notation \eqref{eq:XPhi} for the toric variety $X$, we apply the
  multiple cut $\PP$ from 
  \eqref{eq:cutxx} from earlier in the section.
  The Lagrangian $L$ is then a tropical torus.
  We fix a generic cone direction $\etasp$. 
  By Proposition \ref{prop:mc-restate}, $CF_{\split}(\XX,L,{\etasp})$ is weakly unobstructed
  and $b=wx^{\greyt}$ is a solution of the Maurer-Cartan equation, where $w \in \Lam_{\geq 0}$ is given by
  \begin{equation}
    \label{eq:m0w}
    m^0_{CF_{\split}(\XX,L,{\etasp})}=w x^{\blackt},   
  \end{equation}
  and the disk potential is $W(b)=w$. It is enough to compute $w$
  working with the split Fukaya algebra, since the potential is
  preserved by \ainfty homotopy equivalences: By Theorem
  \ref{thm:tfuk}, there is an \ainfty homotopy equivalence
  \[ \F : CF_\spl(\XX,L,\etasp) \to CF(X,L). \] 
  It follows that $\F(b)$ is a
  Maurer-Cartan solution (see \cite[Lemma 5.2]{cw:flips}) for
  $CF(X,L)$, and $W(\F(b))=W(b)$. By \eqref{eq:m0w}, the quantity $w$
  is given by counting split disks for the cone direction $\etasp$
  that satisfy a point constraint on the boundary.

  First, we describe the set
  of polytopes in the polyhedral decomposition $\PP$: 
    There is a top-dimensional
  polytope $P_F \in \PP$ corresponding to every face
  $F \subseteq \Delta$, and two of these polytopes $P_{F_1}$,
  $P_{F_2}$ intersect exactly if either $F_1 \subseteq F_2$ or
  $F_2 \subseteq F_1$.  Further, any polytope in $\PP$ is of the form
  \[P_{F_1 F_2}=P_{F_1} \cap P_{F_2}, \quad F_1 \subseteq F_2\]
  for a pair of faces $F_1, F_2$ in $\Delta$, and
  $\dim(P_{F_1 F_2})=\codim_{F_2}(F_1)$.
  %
  We use the following shorthand notation for cut spaces
  \[\ol X_{F_1 F_2}:=\ol X_{P_{F_1 F_2}}, \quad \ol X_F:=\ol
    X_{P_F} \]
  for faces $F$, $F_1$, $F_2$ of $\Delta$, and a further abbreviation
  for some cut spaces
  \[\ol X_0:=\ol X_{\Delta}, \quad \ol X_i:=\ol X_{D_i}, \quad
    \ol X_{0i}:=\ol X_{(D_i,\Delta)}.\]
  where $D_i$ is a facet of $\Delta$, and $1 \leq i \leq N$. The dual
  complex $B^\dual$ consists of a top-dimensional polytope for every
  corner in $\Delta$, and a zero-dimensional cell for every face of
  $\Delta$. For example, if $\Delta$ is a three-dimensional cube, the
  dual polytope $B^\dual$ is as in Figure \ref{fig:blas}, where the
  grey dots are $0$-cells, and the dotted lines bound one of the
  top-dimensional cells.

  Corresponding to every
  prime toric boundary
  divisor
  of $X$
  there is a broken disk
  whose glued type has Maslov index two, which we now describe.  The
  perturbation datum used to define the broken Fukaya algebra on
  $\XX_{\PP}$ is such that the almost complex structure on $\ol X_0$
  is standard, and the pieces $\ol X_1,\dots,\ol X_N$ are fibrations
  \begin{equation}
    \label{eq:fibpi}
    \pi_i:\ol X_i \to \ol X_{0i}
  \end{equation}
  whose fibers are holomorphic spheres, each of which intersects the
  relative divisor $\ol X_{0i}\simeq \ol X_i \cap \ol X_0$ at a single
  point.  The index two disk incident on the $i$-th toric divisor has
  two components : $u_\br=(u_+,u_-)$. In this pair
  $u_+ : \D \to \ol X_0$ is an index two Maslov disk intersecting the
  relative divisor $\ol X_{0i}$ at $p_i \in \ol X_{0i}$, and $u_-$
  maps to $\ol X_i$ and is a fiber of $\pi_i$ in \eqref{eq:fibpi}.

  Next, we describe the deformed family corresponding to the broken disk in the previous paragraph.
  Corresponding to the broken disk $u_\br$, 
  there is a family of deformed disks
  \[T_\C/T_{\cT(e),\C} \ni \tau \mapsto u^\tau, \]
  where $u_\br$ is the element  $u^{\tau=0}$. 
  Here $u^\tau:=(u_+, u_-^\tau)$ is a $\tau$-deformed map, and
  $u_-^\tau$ is a sphere in $\ol X_{P_i}$ homotopic to $u_-$ that
  satisfies
  \begin{equation}
    \label{eq:matchtau}
    p_i:=\ev_{\scriptstyle \ol X_{0i}}(u_+)=e^{\tau}\ev_{\scriptstyle
      \ol X_{0i}}(u_-^\tau) , 
  \end{equation}
  where $\ev_{\scriptstyle \ol X_{0i}}$ is the ordinary evaluation map
  at the lift of the nodal point $w_\pm(e)$ mapping to $\ol
  X_{0i}$. 
  The component $u_+$ stays constant under variation of $\tau$ because
  it is a disk of Maslov index two in the toric variety $\ol X_{0}$,
  and is therefore rigid.

  The split map is the limit of a sequence of deformed broken maps
  whose deformation parameters $\tau_\nu$ approach the infinite end of
  the torus in a generic direction. As $\nu \to \infty$, the sequence
  of points $e^{-\tau_\nu}p_i \in \ol X_{0i}$ approaches a $T$-fixed
  point $\ol X_{p, D_i} \in \ol X_{0i}$, where $p \in D_i$ is a vertex
  on the facet $D_i$.  Here we have used the observation that all
  $T$-fixed points in $\ol X_{0i}$ are intersections of relative
  divisors.  Consequently, the sequence $u^\tau$ converges to a split
  map $u^\infty=(u_+,u_-^\infty)$ where $u_-^\infty$ maps to the neck
  piece $\ol \XC_{P_{(p,D_i)}}$.  The tropical graph underlying
  $u^\infty$ satisfies the cone condition for the following
  reason. The quasi-split tropical graph $\tGam$ underlying $u^\infty$
  has vertices $v_+$, $v_-$ corresponding to $u_+$, $u_-^\infty$,
  which are connected by a split edge $e$. The vertex $v_-$ is free to
  move in the dual polytope $P_{(p,D_i)}^\dual$ which is
  $n-1$-dimensional (see Figure \ref{fig:blas}). Thus the set of
  relative vertex positions $\W(\tGam,\bGam)$ is
  $(\dim(\t)-1)$-dimensional. The area and boundary holonomy of
  $u^\infty$ are equal to the area and boundary holonomy of $u$, and
  therefore, $u^\infty$ makes the expected contribution to the
  potential $m_0(1)$ of $CF_{\split}(\XX_{\hat P},L)$.

  Finally, 
  we show that all other terms in the  potential of
  split Fukaya algebra of the Lagrangian fiber 
  are of higher order in $q$ than at
  least one of the disks in the last paragraph.  Denote the split disk
  intersecting the divisor $D_i \subset X$ (from the last paragraph)
  by $u_i:=(u_{i,+},u_{i,-}^\infty)$.  Suppose $u$ is a split disk
  that contributes to the potential, and is not equal to $u_i$ for any
  $i$.  There are two possibilities:
  \begin{enumerate}
  \item The disk part of $u$, denoted by $u_\white : C \to \ol X_{0}$
    is not a Blaschke disk of Maslov index two. Since all the
    torus-invariant divisors of $\ol X_0$ are relative divisors, we
    may conclude that $u_\white$ intersects more than one
    torus-invariant divisors, say $\ol X_{0i}$ and $\ol X_{0j}$,
    transversely. Since the constants $\eps_i$, $\eps_j$ are small (or
    in other words the cut is close to the toric divisor of $X$), we
    conclude that the area of $u_\white$ is larger than both the split
    disks $u_i$, $u_j$, and thus $u$ contributes to a higher order
    term.
  \item The disk part of $u$, denoted by $u_\white : C \to \ol X_0$,
    is a Blaschke disk of Maslov index two intersecting the relative
    divisor $\ol X_{0i}$, but the other component(s) is not the fiber
    sphere in $\ol X_i$. Again, this configuration has a larger area
    than the split disk $u_i=(u_{i,+},u_{i,-}^\infty)$, since
    $u_\white=u_{i,+}$, $u^\infty_{i,-}$ is homologous to a fiber
    sphere of area $\eps_i$ in $\ol X_i$, and in the manifold
    $\ol X_i$ the fiber sphere has the smallest area amongst all
    non-constant symplectic spheres.
  \end{enumerate}
  This finishes the proof of Proposition \ref{prop:blas}.
\end{proof}

\bibliography{split}{} \bibliographystyle{plain}

\end{document}